\documentclass[12pt]{amsart}

\usepackage{amsmath,amssymb,fullpage}
\usepackage{verbatim}
\usepackage{paralist}
\usepackage{hyperref}

\numberwithin{equation}{section}
\theoremstyle{plain}
\newtheorem{theorem}[equation]{Theorem}

\newtheorem{proposition}[equation]{Proposition}
\newtheorem{lemma}[equation]{Lemma}
\newtheorem{corollary}[equation]{Corollary}

\theoremstyle{remark}

\newtheorem{remark}[equation]{Remark}

\theoremstyle{definition}
\newtheorem{definition}[equation]{Definition}

\newtheorem{example}[equation]{Example}

\newcommand{\h}{{\mathcal H}}

\newcommand{\1}{\chi}

\newcommand{\E}{\mathcal E}

\newcommand{\e}{{\epsilon}}

\renewcommand{\H}{\mathbb H}
\newcommand{\bbS}{\mathbb S}
\newcommand{\I}{{\mathcal I}}

\renewcommand{\L}{{\mathcal L}}

\newcommand{\N}{\mathbb N}
\renewcommand{\P}{{\mathcal P}}

\newcommand{\R}{\mathbb R}

\renewcommand{\S}{{\mathcal S}}

\newcommand{\V}{{\mathcal V}}
\newcommand{\Z}{\mathbb Z}

\newcommand{\cut}{\operatorname{Cut}}

\newcommand{\Lip}{\operatorname{Lip}}
\newcommand{\LIP}{\operatorname{LIP}}

\newcommand{\mass}{\operatorname{Mass}}

\newcommand{\per}{\operatorname{Per}}
\newcommand{\Per}{\operatorname{PER}}

\newcommand{\lines}{{\rm lines}}

\def\de{\delta}

\def\mint{-\!\!\!\!\!\!\int}

\def\Si{\Sigma}

\def\defeq{:=}

\newcommand{\no}{\noindent}

\def\XXint#1#2#3{{\setbox0=\hbox{$#1{#2#3}{\int}$}
      \vcenter{\hbox{$#2#3$}}\kern-.5\wd0}}

\begin{document}

\begin{abstract}
We prove a quantitative bi-Lipschitz nonembedding theorem for the Heisenberg group
with its Carnot-Carath\'eodory metric and
 apply it to give a lower bound on the integrality gap of the Goemans-Linial
semidefinite relaxation of the Sparsest Cut problem.
\end{abstract}

\title[Compression bounds]
{Compression bounds for Lipschitz maps \\from the Heisenberg group
to $L_1$}

\author{Jeff Cheeger}
\thanks{J.C. was supported in part by NSF grant DMS-0704404.}
\author{Bruce Kleiner}
\thanks{B.K. was supported in part by NSF grant DMS-0805939}
\author{Assaf Naor}\address{Courant Institute, New York University, 251 Mercer
Street, New York NY 10012, USA. }
\thanks{A. N. was supported in
 part by NSF grants  CCF-0635078
and CCF-0832795, BSF grant 2006009, and the Packard Foundation.}
\maketitle

{\small\tableofcontents}


\section{Introduction and statement of main results}
\label{3steps}

Theorem \ref{tmain}, the main result of this paper, is a quantitative bi-Lipschitz
nonembedding
theorem,
in which the domain is
a metric ball in the Heisenberg group, $\H$,
with its
Carnot-Carth\'eodory metric, $d^\H$,
 and the target is the space $L_1$; for the definition of $d^\H$ see the subsection of
Section \ref{overview} entitled: ``The Heisenberg group as a PI space''.
This result has consequences of a purely mathematical nature,
as well as for theoretical computer science.

Define $c_p(X,d^X)$, {\it the $L_p$ distortion of the metric space $(X,d^X)$}, to be
the infimum of those
$D>0$ for which there exists a mapping $f:X\to L_p$ satisfying
$\frac{\|f(x)-f(y)\|_{L_p}}{d^X(x,y)}\in [1,D]$ for all distinct $x,y\in X$.
The quantitative study of bi-Lipschitz embeddings of finite metric spaces in $L_p$
spaces goes back to \cite{Lind64, Enflo69}. The modern period begins with
a result of Bourgain, \cite{Bou85}, who answered a question of \cite{JohnsonLind84}
by showing that for every fixed $p$, any $n$-point metric
space can be embedded in $L_p$  with distortion
$\lesssim\log n$.\footnote{In this paper, the symbols $\lesssim, \gtrsim$, denote the
corresponding inequalities, up to a universal multiplicative constant, which in all cases
can be explicitly estimated from the corresponding proof.  Similarly,
$\asymp$ denotes equivalence up to such a factor.}
By  \cite{LLR95}, Bourgain's theorem is sharp for any fixed $p<\infty$.

Since $L_1$, equipped with the square root of its usual distance, is
well known to be isometric
to a subset of $L_2$ (see for example \cite{WellWill})
 it follows that if $(X,d^X)$  isometrically embeds
in $L_1$, then
$(X,\sqrt{d^X})$ isometrically embeds in $L_2$. Metrics for which $(X,\sqrt{d^X})$
isometrically embeds in $L_2$ are said to be
of {\it negative type}.  Such metrics will play a fundamental
role in our discussion. On the other hand, it is also well known that $L_2$ embeds
isometrically in $L_1$ (see for example \cite{MR1144277})
which implies that $c_1(X,d^X)\leq c_2(X,d^X)$, for all $(X,d^X)$. Recently,
it was shown that for $n$-point metric spaces of negative type,
 $c_2(X,d^X)\lesssim(\log n)^{1/2+o(1)}$, and  in particular
$c_1(X,d^X)\lesssim(\log n)^{1/2+o(1)}$;
see \cite{ALN08}, which improves on a corresponding result in \cite{CGR08}.
 For embeddings in $L_2$, the result of \cite{ALN08} is sharp up to the term $o(1)$;
see \cite{Enflo69}.

 As shown \cite{LN06}, the Carnot-Carath\'eodory
metric $d^\H$
{\it is} bi-Lipschitz to a metric
of negative type. From this and Corollary \ref{coro:rate distortion} of
Theorem \ref{tmain} below, it follows immediately
  that for all $n$, there exist $n$-point metric spaces of negative type
with $c_1(X,d^X)\gtrsim (\log n)^\delta$, for some explicit $\delta>0$.
 From the standpoint of such {\it nonembedding theorems},
the target $L_1$ presents  certain challenges. Lipschitz functions $f:\R\to L_1$
need not be differentiable anywhere. Therefore,
a tool which is useful for $L_p$ targets with $1<p<\infty$ is not available.
Moreover, the fact that  $L_2$ embeds isometrically  in $L_1$  implies
  that bi-Lipschitz embedding in $L_1$ is no harder than
in $L_2$ and might be strictly
easier  in cases of interest.

The Sparsest Cut problem is a fundamental NP hard problem in theoretical computer
science. This problem will formally stated in  Section \ref{SC} (an appendix)
 where some additional details of the discussion which follows
will be given; see also
our paper \cite{CKN09} (and the references therein)
which focuses on the computer science aspect of our work.
A landmark development
took place  in the 1990's, when it was realized that this optimization problem
for a certain functional defined on  all subsets of
an $n$-vertex weighted graph, is equivalent to
an  optimization problem for a corresponding functional over all  functions from
an  $n$-vertex weighted graph to the space $L_1$;
see \cite{AvisDeza, LLR95, AumanRabani}. This is a
consequence of the  {\it cut cone} representation for metrics
induced by maps to $L_1$, which also plays
a fundamental role in this paper; see Section \ref{3steps}.
Once this reformulation has been
observed, one can relax the problem to an optimization problem for the corresponding
functional over functions with values in {\it any}
$n$-point metric space. The relaxed problem
 turns out to be a linear program, and hence,
is solvable in polynomial time.
 Define the {\it integrality gap} of this relaxation to be the supremum
over all $n$ point weighted graphs
of  the ratio of  the
solution of the original problem to the relaxed one. The integrality
gap measures the
performance of the relaxation in the worst case.
It is essentially
immediate that the integrality gap
 is $\leq $ the supremum of $c_1(X,d^X)$ over
all $n$ point metric spaces, $(X,d^X)$, and hence, by Bourgain's theorem is
$\lesssim \log n$;  see \cite{LLR95, AumanRabani}. This upper bound
relies only on the form of the functional and not on any special properties of $L_1$.
A duality argument based on the cut cone representation
shows that the integrality gap is actually {\it equal to}
the supremum of $c_1(X,d^X)$ over
all $n$ point metric spaces.

  Subsequently, Goemans \cite{Goe97} and Linial \cite{Lin02} observed that
if in relaxing the Sparsest Cut problem as above, one restricts to
$n$ point metric spaces of negative type, one obtains a semidefinite programing problem, which,
 by the
Ellipsoid Algorithm, can still be solved in polynomial time with arbitrarily good
precision; see \cite{GLS93}.
As above, by the duality argument,
the integrality gap for the Goemans-Linial semidefinite relaxation is actually
equal to
the supremum of $c_1(X,d^X)$ over
all $n$ point metric spaces of negative type.
Based on certain known embedding results for particular metric
spaces of negative type, the hope was
that this integrality gap might actually be bounded
(the ``Goemans-Linial conjecture'')  or any case, bounded a very
slowly growing function of $n$. At present,
one knows the upper bound
$\lesssim(\log n)^{1/2+o(1)}$, which follows from from  \cite{ALN08}.
This result makes the Goemans-Linial semidefinite relaxation the most successful
algorithm to date for solving the Sparsest Cut problem to within a definite
factor.
In the opposite direction, it was shown in \cite{KR06} that the integrality
gap for the Goemans-Linial semidefinite relaxation is
 $\gtrsim \log\log n$. The analysis of
\cite{KR06}, improved upon that of the breakthrough result of~\cite{KV04}, which
 was the first to show that the Goemans-Linial semidefinite relaxation cannot yield a constant factor
 approximation algorithm for the Sparsest Cut problem, thus resolving the
Goemans-Linial conjecture \cite{Goe97,Lin02,Lin-open}. These lower
bounds on the integrality gap depend on its characterization
as the supremum of $c_1(X,d^X)$ over
all $n$ point metric spaces of negative type.

Motivated by the potential relevance to the Sparsest Cut problem,
the  question of whether $(\H,d^\H)$ bi-Lipschitz embeds in $L_1$ was raised in
\cite{LN06}.
 In response, it was shown in \cite{ckbv} that if
 $U\subset \H$ is open and $f:U\to L_1$ is
 Lipschitz (or more generally of bounded variation), then
for almost all $x\in U$ (with respect to Haar measure) and $y$
varying in the coset of the center of $\H$ containing $x$, one has
 $\lim_{y\to x}\frac{\|f(y)-f(x)\|_{L_1}}{d^\H(x,y)}=0$.
 Thus,
$(\H,d^\H)$ does not admit a bi-Lipschitz
embedding into $L_1$.

 For further applications as discussed above, a quantitative version of
this theorem of \cite{ckbv} was required;
see Theorem \ref{tmain} below.\footnote{
 For purposes of exposition, in Theorem \ref{tmain}, we restrict
attention to the case of Lipschitz maps, although everything we say has an
analogous statement which
apply to  BV maps  as well, sometimes with minor variations.} It follows from
Corollary \ref{coro:rate distortion} of Theorem \ref{tmain}
that there exists a sequence of $n$-point metric spaces, $(X_n,d^{X_n})$,
of negative type,
such that $c_1(X_n,d^{X_n})\gtrsim (\log n)^\delta$ and hence,
that the integrality gap
for the Goemans Linial relaxation of Sparsest Cut is
$\gtrsim(\log n)^\delta$ for some explicit $\delta>0$; compare  Remark
\ref{coro: rate distortion implies int gap}.  This represents an exponential
improvement on the above mentioned lower bound $ \gtrsim\log\log n$; compare also the
upper bound $\leq(\log n)^{1/2+o(1)}$.

 We also give a purely mathematical application of Theorem \ref{tmain}  to the behavior of the
  $L_1$ compression rate
of the the discrete Heisenberg group. The $L_1$ compression rate is
a well studied invariant of the asymptotic geometry of a finitely generated  group,
which was defined by Gromov in  \cite{Gro93}; see below for the definition.

In what follows, given $p\in \H$ and $r>0$, we
denote the $d^\H$-open ball of radius $r$ centered at $p$ by
$B_r(p)=\{x\in \H:\ d^\H(x,p)<r\}$.

 \begin{theorem}[Quantitative central collapse]
\label{tmain} There exists a universal constant $\delta\in (0,1)$ such that for every $p\in \H$,
every $f:B_1(p)\to L_1$
with $\Lip(f)\le 1$, and every $\e\in \left(0,\frac14\right)$, there exists $r\ge \e$ such that
with respect to Haar measure, for at least half \footnote{In Theorem~\ref{tmain},
by suitably changing the constant $\delta$,
``at least half'' can be replaced by any definite fraction.}  of the points $x\in B_{\frac12}(p)$, at
least half of the points
$(x_1,x_2)\in B_r(x)\times B_r(x)$ which lie on the same coset of the center satisfy:
\begin{equation}
\label{compression} \frac{\|f(x_1)-f(x_2)\|_{L_1}}{d^\H(x_1,x_2)}\leq \frac{1}{(\log(1/\e))^\delta}\, .
\end{equation}
\end{theorem}
In particular, compression by a factor of $\eta\in \left(0,\frac12\right)$ is
guaranteed to occur  for a pair of points whose
distance is $\gtrsim e^{\eta^{-c}}$, where $c=\delta^{-1}$ ($\delta$ as in (\ref{compression})).

The constant $\delta$ in Theorem~\ref{tmain} can be explicitly estimated from our proof; a crude estimate is
 $\delta=2^{-60}$ in Theorem~\ref{tmain}. At various points in the proof, we have sacrificed
sharpness in order to
simplify the exposition; this is most prominent in Proposition~\ref{pdrrvc} below, which
must be iterated several times,
 thus magnifying the non-sharpness. Obtaining the best possible $\delta$ in
Theorem~\ref{tmain} remains
an interesting open question,
the solution of which probably requires additional ideas beyond those contained in this paper.

Before discussing the  consequences of
Theorem \ref{tmain},
we briefly indicate the  reason for the form of the estimate
(\ref{compression}); see also the discussion of Section \ref{overview}.
We will associate to
$f$ a {\it nonnongative quantity}, the {\it total nonmonotonicty},
which can be written as a {\it sum over scales}, and on
each scale as
an integral over locations. The assumption $\Lip f\leq 1$
turns out to
imply an a priori bound  on the total nonmonotonicity.
 Since for $ \epsilon\in (0,1)$, the total number of scales between $1$ and $\epsilon $
is $\asymp\log(1/\epsilon)$, by
 the pigeonhole principle,
 there exists a scale as in (\ref{compression}), such that at most locations,
the total nonmonoticity is
$\lesssim 1/(\log (1/ \epsilon))$.
We show using a stability theorem, Theorem \ref{stability}, that
for suitable $\delta$,
this gives (\ref{compression}).
This discussion fits very well with a
quantitative result (and phenomenon) due in a rather different context to
Jones; see \cite{jones88} and compare the discussion in Section \ref{overview}.  As explained in
Section \ref{general},  the argument indicated above can be viewed
as a particular instance of one which is applicable in
considerable generality.

\subsection*{The discrete case.}
As we have indicated,  Theorem \ref{tmain} has
 implications in the context of
finite metric spaces. These are based on properties of the
discrete version of Heisenberg group, $\H(\Z)$ (defined below) and its metric balls,
which follow from Theorem \ref{tmain}.

We emphasize at the outset
 that our results in the discrete case are obtained  (without difficulty)
directly from the
corresponding {\it statements}
 in the continuous case,
and not by a ``discretization''
 of their {\it proofs}.
As in \cite{ckbv} and in \cite{ckmetmon}, where a different proof
of the nonquantitative result of \cite{ckbv} is given,
the proof of Theorem \ref{tmain} is carried out in the continuous case because in
that case, the
methods of real analysis are available.  Nonetheless, in the present instance,
 the quantitative issues remain highly nontrivial and
the proof  requires  new ideas  beyond those of
 \cite{ckbv},  \cite{ckmetmon}; see  the discussion at the beginning of Section
\ref{overview} and in particular Remark \ref{rem: no scale}.

We will view $\H$ as $\R^3$ equipped with the noncommutative product
$(a,b,c)\cdot(a',b',c')=(a+a',b+b',c+c'+ab'-ba')$; for further discussion
of $\H$, see Section \ref{overview}.
 From the multiplication formula, it follows directly that
for $R>0$, the map, $A_R:\R^3 \to \R^3$,
defined by $A_R((a,b,c))=(Ra,Rb,R^2c)$ is an automorphism of $\H$. It is also
a homothety of the metric $d^\H$. The discrete Heisenberg group, $\H(\Z)$, is
the integer lattice,
$\Z^3=\{(a,b,c)\, |\, a,b,c\in\Z\}$, equipped with the above product.
It is a discrete cocompact subgroup of $\H$.
 For further discussion of the Heisenberg group  see Section \ref{overview}.

 Fix a finite set of generators $T$ of  a finitely generated
group, $\Gamma$.  The word metric, $d_T$ on $\Gamma$ is the left invariant
metric defined by stipulating that  $d_T(g_1,g_2)$ is the length
of the shortest word in the elements of $T$  and their inverses which
expresses
$g_1^{-1}g_1$. Up to bi-Lipschitz equivalence, the metric $d_T$ is independent
of the choice of generating set.
 For the case of $\H(\Z)$,
we can take  $T=\{(\pm 1,0,0), (0,\pm 1,0), (0,0,\pm 1)\}$. For definiteness,
from now on this choice will be understood.
By an easy general lemma, given
 a free co-compact action of a finitely generated group
$\Gamma$ acting freely and cocompactly on a length space $(X,d^X)$,
 the metric on $\Gamma$ induced by the restriction
to any orbit of the metric $d^X$ is  bi-Lipschitz equivalent to $d_T$;
 see \cite{BBI}.
 For the case of $\H(\Z)$, we can take $(X,d^X)=(\H,d^\H)$.

 Define $\rho:\R^3\times \R^3\to [0,\infty)$ by
\begin{equation}
\begin{aligned}
\rho( &(x,y,z), (t,u,v) )\\
 &\defeq \left(\left[  \left((t-x)^2 +
 (u-y)^2\right)^2 + (v-z+2xu - 2yt)^2\right]^{\frac{1}{2}} + (t-x)^2 +
 (u-y)^2\right)^{\frac{1}{2}}\,  .
\end{aligned}
\end{equation}
It was shown in \cite{LN06} that $(\H,\rho)$ is a metric of negative type
bi-Lipschitz to $(\H,d^\H)$.

It follows from \cite{ckbv} that
\begin{equation}\label{eq:non quant dist}
\lim_{n\to \infty}c_1\left(\{0,\ldots,n\}^3,\rho\right)=\infty\, ,
\end{equation}
but no information can be deduced
on the rate of blow up as $n\to \infty$.
 From Theorem \ref{tmain} we get the following corollary, whose proof will be
explained at the end of this subsection.

\begin{corollary}\label{coro:rate distortion}
 For the constant $\delta>0$ in Theorem~\ref{tmain}, we have for all $n\in \N$,
metric spaces $\left(\{0,\ldots,n\}^3,\rho\right)$ of negative type, statisfying
\begin{equation}\label{eq:distortion rate}
c_1\left(\{0,\ldots,n\}^3,\rho\right)\gtrsim (\log n)^\delta\, .
\end{equation}
\end{corollary}

\begin{remark}
\label{coro: rate distortion implies int gap}
Since the  metric spaces $c_1\left(\{0,\ldots,n\}^3,\rho\right)$
are of negative type, relation (\ref{eq:distortion rate})
implies
 that the integrality gap of
 the Goemans Linial relaxation of Sparsest Cut is
$\gtrsim(\log n)^\delta$ for $\delta>0$ as Theorem \ref{tmain}.
\end{remark}

\begin{remark}
A metric space is said to be {\it doubling} if a metric ball $B_{2r}(x)$,
can be covered by at most $N<\infty$ metric balls of radius $r$, where
$N$ is independent of $x,r$.
 For $n$-point metric spaces, $(X,d)$,  which are doubling,
the bound in Bourgain's theorem can
be sharpened to
$c_1(X,d)\leq c_2(X,d)\lesssim \sqrt{\log n}$,
which follows from the results of~\cite{Ass83, Rao99}; see the
explanation in~\cite{GKL03}.
The metric spaces, $(\H,d^\H)$ and $(\H(\Z),d_T)$ are doubling.
(To see this, use for example, the left invariance
of $d^\H$ and the  homotheties $A_R$.)
Before  the
bi-Lipschitz non-embeddability of $\H$
into $L_1$ was established in~\cite{ckbv}, there was no known
example of a doubling metric space which does not admit a
bi-Lipschitz embedding into $L_1$. Corollary \ref{coro:rate distortion}
shows that
there is a sequence of $n$-point doubling metric spaces for which
$c_1(X,d)\gtrsim (\log n)^\delta$.
\end{remark}


\begin{remark}
The behavior of the $L_2$ distortion for $n$-point doubling
metric spaces is much easier to understand than the $L_1$ distortion.
Namely, for fixed doubling constant, the above mentioned bound,
$c_2(X,d)\lesssim\sqrt{\log n}$ cannot be improved; see  ~\cite{LP01},
and for
dependence on the doubling constant, \cite{KLMN05}.
\end{remark}

Given a  finitely generated group $\Gamma$ and a
$1$-Lipschitz function $f:(\Gamma,d_T)\to L_1$,
Gromov (\cite{Gro93}) defined the
{\em compression rate}
$\omega_f:[1,\infty)\to [0,\infty)$ by:
$$
\omega_f(t)\defeq\inf\left\{\|f(x)-f(y)\|_{L_1}\left|\ d_T(x,y)\ge t\right.\right\}.
$$
Stated differently, $\omega_f$ is the largest non-decreasing function for which
$\|f(x)-f(y)\|_{L_1}\ge \omega_f\left(d_T(x,y)\right)$ for all $x,y\in \Gamma$.

It follows  from~\cite{ckbv}
that for any $1$-Lipschitz map
$f:\H(\Z)\to L_1$ we have $\omega_f(t)=o(t)$,
 but~\cite{ckbv} does not give any information on the
rate at which $\frac{\omega_f(t)}{t}$ tends to zero.
 From Theorem~\ref{tmain}
we can obtain the  bound:
\begin{corollary}\label{coro:compression}
 For every $f:\H(\Z)\to L_1$ which is $1$-Lipschitz with respect to the word metric
$d_W$, we have for all $t\ge 1$:
\begin{equation}\label{eq:our compression}
\omega_f(t)\lesssim \frac{t}{(1+\log t)^\delta}\, ,
\end{equation}
where $\delta>0$ is the constant in Theorem~\ref{tmain}.
\end{corollary}

\begin{remark}
By a general result from~\cite{Tess08}
(see Corollary 5 there) which implies that if an increasing function
$\omega: [1,\infty)\to [0,\infty)$ satisfies
$\int_1^\infty \frac{\omega(t)^2}{t^3}dt<\infty$
 then there exists a mapping $f:\H(\Z)\to L_1$ which is Lipschitz in the metric $d_T$ and
$\omega_f\gtrsim \omega$. In fact, $f$ can be chosen to
take values in the smaller space $L_2$.
 By choosing $\omega(t)=\frac{t}{\sqrt{1+\log t}\cdot \log \log(2+t)}$,
and bringing in Corollary~\ref{coro:compression}, it follows  that
in the terminology
of~\cite{ADS06},
the discrete
Heisenberg group $(\H(\Z),d_T)$ has $L_1$ {\it compression gap}
$$
\left(\frac{t}{\sqrt{1+\log t}\cdot \log \log(2+t)},\frac{t}{(1+\log t)^\delta}\right)\,  .
$$
 It would be of interest to evaluate the supremum of those
$\delta>0$ for which there exists
$f:\H(\Z)\to L_1$ which satisfies~\eqref{eq:our compression}.
$L_2$).
\end{remark}

We close this subsection by explaining how
Corollaries~\ref{coro:compression} and \ref{coro:rate distortion}
are deduced from  Theorem~\ref{tmain}. A key point is to
pass from the discrete settings of these corollaries to the
continuous setting of Theorem~\ref{tmain} via a Lipschitz extension
theorem. The basic idea is simple and general (see Remark 1.6 in~\cite{LN06}).
Additionally,   the
homotheties
$A_R(a,b,c)$ $=(Ra,Rb,R^2c)$
 are used to convert information from Theorem \ref{tmain}
concerning small scales, into information concerning  large scales.
 The existence of these homotheties is, of course,
a special property of $(\H,d^\H)$.

We will give the details for
Corollary~\ref{coro:compression}; the case of Corollary \ref{coro:rate distortion}
is entirely similar.
\proof (of Corollary \ref{coro:compression})
 Fix a  map $f:\H(\Z)\to L_1$, which is $1$-Lipschitz with respect to  the word metric
$d_S$.
 The map $f$ can be extended to a map $\tilde f:\H\to L_1$
whose Lipschitz constant with respect to $d^\H$ satisfies
$\Lip\, \widetilde f\lesssim 1$. This fact follows
from the general result of~\cite{LN05}, which states that such an
extension is possible for any Banach space
valued mapping from any doubling subset of a metric space to
the entire metric space, but in the present simpler setting it also
follows from a straightforward partition of unity argument.

Fix $R>1$ and define $g_R:B_1((0,0,0))\to L_1$ by
$g_R(x)=\frac{1}{R}\tilde f(\delta_R(x))$.  Since
$d^\H(A_R(x),A_R(y))$ $=Rd^\H(x,y)$,
$\Lip\, (g_R)\lesssim 1$. For $\epsilon\in
\left(0,\frac14\right)$, an application of Theorem~\ref{tmain} shows
that there exist $x,y\in \H$ such that $R\gtrsim d^\H(x,y)\gtrsim
\epsilon R$, and
\begin{equation}\label{eq: compression for contradiction}
\left\|\widetilde f(x)-\widetilde f(y)\right\|_{L_1}\lesssim
\frac{1}{(\log(1/\epsilon))^\delta}\cdot d^\H(x,y)\, .
\end{equation}
Choose $\epsilon=\frac{1}{\sqrt{R}}$. Since there exist $a,b\in
\Z^3$ such that $d^\H(a,x)\lesssim 1$ and $d^\H(b,y)\lesssim 1$, and
since  $\Lip\, \widetilde f\lesssim 1$, it follows
from~\eqref{eq: compression for contradiction} that provided $R$ is
large enough,
$$
\omega_f\left(d_W(a,b)\right)\le \|f(a)-f(b)\|_{L_1}\lesssim
\frac{d^\H(a,b)}{(\log R)^\delta}\lesssim \frac{d^\H(a,b)}{\left(\log
\left(d^\H(a,b)\right)\right)^\delta}\, ,
$$
implying~\eqref{eq:our compression}.\qed

\section{The proof of Theorem~\ref{tmain}; overview and background}
\label{overview}

We begin with an informal overview of the proof  of Theorem \ref{tmain}.
Then proceed to a more detailed discussion, including relevant background
material.

Many known bi-Lipschitz non-embedding results   are based
 in essence on a differentiation argument. Roughly,
one shows that at almost all points,
in the infinitesimal limit, a Lipschitz map
converges
to a map with a special structure. One then shows
(which typically is not difficult)
that maps with this special structure cannot be
bi-Lipschitz. The term ``special structure" means different things in different settings. When
the domain and  range are Carnot
groups as in Pansu's differentiation theorem~\cite{Pan89},
special structure
means a group homomorphism.
As observed in~\cite{LN06,CK-pansu}, Pansu's theorem
extends to the case in which the domain is $(\H,d^\H)$ and the target is
an infinite dimensional Banach space with the Radon-Nikodym
property, in particular, the target can be $L_p$ for $1<p<\infty$.

The above approach
fails for embeddings into $L_1$, since even when the domain is $\R$,
Lipschitz maps need not be differentiable anywhere.
  A simple example is
provided by the map
$t\to \chi_{[0,t]}$, where $\chi_{[0,t]}$ denotes the characteristic function of $[0,t]$; see \cite{aronszajn}.
Nevertheless,
the result on central collapse proved in ~\cite{ckbv} can be viewed as
following from a differentiation theorem,
provided one
interprets this statement via a novel notion of ``infinitesimal regularity" of mappings
 introduced in~\cite{ckbv}.

The approach of~\cite{ckbv} starts out with the cut-cone representation of $L_1$ metrics
(see~\cite{dezalaur,ckbv}),
which asserts that for every $f:\H\to L_1$ we can write
$$
\|f(x)-f(y)\|_{L_1}=\int_{2^\H}|\chi_E(x)-\chi_E(y)|d\Sigma_f(E)\, ,
$$
for all $x,y$,
 where $\Sigma_f$ is a canonically defined measure on $2^\H$
(see the discussion following~\eqref{cutmetdef1}
for precise formulations). The differentiation result of~\cite{ckbv} can be viewed
as a description of the infinitesimal
structure of the measure $\Sigma_f$.
It asserts that at most locations,
in the infinitesimal limit, $\Sigma_f$ is
supported on vertical half spaces. This is achieved by first showing that the Lipschitz
condition on
$f$ implies that $\Sigma_f$
is supported  on special subsets, those with finite perimeter. One then uses
results on the local structure
of sets of finite perimeter in the Heisenberg group to complete the proof; the
argument is
described in greater detail later in this section.

An alternative approach to this result on the infinitesimal behavior of $\Sigma_f$, which does not
require the introduction
of sets of finite perimeter, was obtained in~\cite{ckmetmon}. This approach is based on the
 classification of monotone sets.
A subset $E\subset \H$ is called monotone if for every {\it horizontal} line $L$,
 up to a set of measure zero, $E\cap L$ is a
sub-ray of the line. If we view $\H=\R^3$, the horizontal lines are a certain
codimension $1$ subset of all the lines; for details,
see the subsection entitled ``The Heisenberg group as a PI space''.
The proof of~\cite{ckmetmon} proceeds by showing that infinitesimally the measure $\Sigma_f$
is supported on
 monotone subsets, and a non-trivial classification theorem which asserts that monotone
subsets are half spaces. (The proof
is recalled in Section \ref{cms}.)

Our proof of Theorem~\ref{tmain}  combines the methods of~\cite{ckbv}
and~\cite{ckmetmon},
with several significant new ingredients. Inspired by the proof in~\cite{ckmetmon},
our argument is based on
an appropriately defined notion of almost monotonicity of subsets. But, unlike~\cite{ckmetmon},
we require the use of perimeter
bounds as well. In our situation,  perimeter bounds are used in finding a controlled
scale such that at most locations, apart from a  certain collection of cuts,
the mass of $\Sigma_f$ is supported on subsets which are sufficiently close to
being monotone.
In actuality, the excluded cuts may have infinite  measure with respect to
$\Sigma_f$.
Nonetheless, using perimeter bounds and the isoperimetric inequality in $\H$,
we show that their contribution to the metric is negligibly small.

A crucial, and by far the most complicated, new ingredient in this paper is a stability version
of the classification
of monotone sets of~\cite{ckmetmon}, which asserts that sets which are almost monotone are
quantitatively close
 to half spaces. One of the inherent reasons for the difficulty of proving such a stability result is
that the need to
 work locally, i.e. to consider almost monotone subsets inside a metric ball in $\H$ of finite radius.
Here the situation is fundamentally different from the
corresponding  classification result in~\cite{ckmetmon}:  There are even precisely
monotone subsets in such a ball
which are not half spaces\footnote{
In $\R^n$,
monotone subsets of a ball are necessarily the intersection of the ball with a half space.}; see
Example \ref{global}. Inevitably,
our classification result must  take this complication into
account.  We can  only assert that  on a (controllably) smaller ball, the given almost monotone subset
is close to a half space.

In order to make the above informal description precise, we will require some additional preliminaries,
which are
explained in the remainder of this introduction. For a discussion, in a wider context, of
estimates whose form is equivalent to that of (\ref{compression}), and which arise from
our general
mechanism, see  Appendix 2, Section \ref{general}.

\begin{remark}\label{rem: no scale}
As we explained above, the present paper belongs to the general topic of controlling the scale
at which the
infinitesimal ``special structure" of a mapping (arising from a differentiation theorem) appears in
approximate form.
However, it is important
to realize that
{\it it is impossible in general to obtain a quantitative estimate for the rate at which the limiting
situation is approached}.
To see this point, consider a model problem associated to the sequence of functions $
f_n(x)=\frac{\sin (nx)}{n}$.
Although $|f_n'(x)|\leq 1$, for all $n\in \mathbb N$ and $x\in \R$, it is not possible control
independently of $n$ the
difference between $f_n$ and its first order Taylor series; note that $f''_n$, which
controls the remainder term, is not bounded independently of $n$.
Nonetheless, there is an explicit uniform estimate for the scale above
which any $f_n$ is approximated by {\it some} linear function, which might not be equal to its
first order Taylor series. As far as we are aware, the first instance of a result of this
type is in \cite{jones88}; see also for example \cite{BJLPS99},
and for additional information, the appendix by
Semmes in \cite{MR2307192}.
\end{remark}

\subsection*{PI spaces}

A natural setting for a large part of our work
is that of PI spaces:
a certain class of metric measure spaces, of which the Heisenberg group is a member.
Together with
the key concept of
 ``upper gradient''
on which their definition is based, these spaces were introduced in
\cite{HeKo}.

Let $(X,d^X)$ denote a metric space. If $f:X\to\R$, then
$g:X\to[0,\infty]$ is called an {\it upper gradient} of $f$, if
for all rectifiable curves $c:[0,\ell]\to X$, parameterized by arc length,
\begin{equation}
\label{uppergrad}
\left|f(c(\ell))-f(c(0))\right|\leq \int_0^\ell g(c(s))\, ds\, .
\end{equation}

A complete metric measure space $(X,d^x,\mu)$ is called a {\it PI space}
 if a {\it doubling condition} and  {\it Poincar\'e inequality}
hold. Thus, for all $R>0$ there exist $\beta=\beta(R), \tau=\tau(R),
\Lambda=\Lambda(R), p\geq 1$, such that for all $x\in X$ and $r\leq
R$ we have:
\begin{equation}
\label{doubling}
\mu(B_{2r}(x))\leq \beta\cdot\mu(B_r(x))\, ,
\end{equation}

\begin{equation}
\label{poincare} \mint_{B_r(x)\times B_r(x)}|f(x_1)-f(x_2)|\,
d(\mu\times \mu)(x_1,x_2) \leq \tau\cdot r\cdot \left(
\mint_{B_{\Lambda r}(x)}g^p\,d\mu\right)^{\frac{1}{p}},
\end{equation}
where we use the notation:
$$
\mint_U h\, d\mu=\frac{1}{\mu(U)}\int_U h\, d\mu\, .
$$

As shown in \cite{hako}, (\ref{doubling}) and (\ref{poincare})
imply a strengthening of (\ref{poincare}), which we will need.
Namely, for some $\chi=\chi(\beta,\tau)>1$,
$\tau'=\tau'(\beta,\tau)$, one has the
{\it Poincar\'e-Sobolev inequality},
\begin{equation}
\label{poincaresob}
\left(\mint_{B_r(x)\times B_r(x)}|f(x_1)-f(x_2)|^{\chi p}d(\mu\times
\mu)(x_1,x_2)\right)^{\frac{1}{\chi p}}
\leq \tau'\cdot r\cdot \left(\mint_{B_{\Lambda
r}(x)}g^p\,d\mu\right)^{\frac{1}{p}}.
\end{equation}
For the case of the Heisenberg group of dimension
$2n+1$, relation
(\ref{poincaresob}) holds with $\chi=\frac{2n+2}{2n+1}$; thus
if $2n+1=3$, then $\chi=\frac{4}{3}$.

Observe that (\ref{poincare}), for fixed
$p$,  implies (\ref{poincare})
for any $p'>p$.
In this paper, without further explicit mention, we will
always assume that $p=1$, which is
 necessary for the ensuing results on finite perimeter. Solely in order to have
one fewer constant to list,
we will also make the innocuous assumption  that
$\Lambda=1$. For the cases of primary interest here,
$X=\R^n$ or $X=\H$, relation
(\ref{poincaresob}) does hold with $p=1$, $\Lambda =1$.


\subsection*{Maps to $L_1$: cut metrics and cut measures.}

Let $X$ denote a set and let $E\subset X$ denote a subset.
 In place of ``subset'', we will also use the
term ``cut''. If $E\subset X$, then associated to $E$ is a
so-called {\it elementary cut metric}, denoted $d_E$, on $X$, where by
definition, $d_E(x_1,x_2)=0$ if either $x_1,x_2$ both lie in $E$, or
if $x_1,x_2$ both lie in $E'$ (here, and in what follows, we denote $E'=X\setminus E$).
Otherwise, $d_E(x_1,x_2)=1$.

\begin{remark}
Clearly, $d_E=d_{E'}$. For
 this reason, it is common to define the space of
cuts of $X$ as the quotient of
 the power set
$2^X$ by the involution, $E\mapsto E'$, although we don't do this here.
\end{remark}

By definition,  a {\it cut metric} $d_\Sigma$ is an integral of
elementary cut metrics with respect to some measure $\Si$ on $2^X$.
Thus,
\begin{equation}
\label{cutmetdef1}
d_\Si(x_1,x_2)=\int_{2^X}d_E(x_1,x_2)\, d\Si(E)\, .
\end{equation}

 For $f:U\to L_1$,
the pull-back metric induced by $f$ is defined as
$d_f(x_1,x_2)=\|f(x_1)-f(x_2)\|_1$. It follows from the ``cut-cone
characterization" of $L_1$ metrics (see~\cite{dezalaur,ckbv}) that
there is a canonically defined measure $\Sigma_f$ on $2^X$ (on an
associated sigma algebra) such that (\ref{cutmetdef1}) holds with
$\Si=\Si_f$.


Now let  $(X,\mu)$ denote a
$\sigma$-finite measure space. Let $U\subset X$ denote
a measurable
subset and
$f:U\to L_1(Y,\nu)$ be a map which satisfies:
$$
\int_{U}\|f(x)\|_{1}\, d\mu<\infty\, .
$$
There is a variant of  the description of $d_f$ in terms of $\Si_f$
in the $L_1$ framework;  see
 \cite{ckbv}.
In this context,
 $2^U$ is replaced by  a measure
theoretic version of the space of cuts. One should regard
 a {\it cut} as
 an equivalence class of measurable sets
$E\subset U$
 of finite $\mu$-measure,
where two sets are considered
equivalent if their symmetric difference
has measure zero. For our purposes,
there is no harm in blurring
the distinction between
measurable sets and their equivalence classes, which for purposes of
exposition, is
done below.

Let $\cut(U)$ denote the space of cuts.
As above, we have:
\begin{equation}
\label{cutmetdef}
d_f(x_1,x_2)=\int_{\cut(U)}d_E(x_1,x_2)\, d\Si_f(E)\, ,
\end{equation}
where $\Si_f$ is a suitable Borel measure on $2^U$.
Here we view $d_f$ as an element of $L_1^{{\rm loc}}(U\times U)$,
whose restriction to any subset $V\subset U$, with finite measure,
lies
 in $L_1(V\times V)$.  Given two such $L_1$-metrics
$d_1,d_2\in L_1(V\times V)$ there is a well-defined $L_1$-distance,
denoted $\|d_1-d_2\|_{L_1(V\times V)}$, given by:
\begin{equation}
\label{ellonedist} \|d_1-d_2\|_{L_1(V\times V)} =\int_{V\times V}
|d_1(x_1,x_2)-d_2(x_1,x_2)|\, d(\mu\times \mu)(x_1,x_2).
\end{equation}

 It is shown in \cite{ckbv}, that if $X$ is
a PI space, $U\subset B_1(p)\subset X$ is an open set, and
$f$ is $1$-Lipschitz, or more generally, $1$-BV, then the cut
measure $\Si_f$ has finite total perimeter:
\begin{equation}
\label{finitetotper}
\int_{\cut(U)}\per(E)(U)=\int_{\cut(U)}\Per(E,U)\,
d\Si_f(E)<c(\beta,\tau).
\end{equation}

Intuitively, the perimeter, $\Per(E,U)$, is
the codimension 1 measure of the
(measure theoretic) boundary of $E\cap U$.
In actuality,  $V\to\Per(E\cap V)$ defines  a  Radon measure
$\per(E)$. The restriction of this measure to subsets of $U$ is denoted
$\per(E,U)$. The mass of $\per(E,U)$ is denoted
$\Per(E,U)$ and is equal to $\per(E)(U)$;
for the definition,
see (\ref{perdef}).
 Moreover, there is a {\it total perimeter measure},
$\lambda_f$, which is a Radon measure on $U$, such that
\begin{equation}
\label{totpermeas}
\lambda_f=\int_{\cut(U)}\per(E,U)\, d\Si_f(E)\, ,
\end{equation}
\begin{equation}
\label{masstotpermeas}
\mass(\lambda_f)=\lambda_f(U)=\int_{\cut(U)}\Per(E,U)\, d\Si_f(E)\, .
\end{equation}
In particular, if $f$ has {\it bounded variation}
($f\in$ BV)
 then $\Si_f$ is
supported on cuts with {\it finite perimeter}.
For the precise definitions
and relevant properties of cut measures, BV maps and
perimeter measures, see \cite{ckbv}.


\subsection*{The Heisenberg group as a Lie group.}

Recall that the $3$-dimensional Heisenberg group
$\H$, can be viewed as $\R^3$ equipped with the
group structure
\begin{equation}
\label{groupstruct}
(a,b,c)\cdot(a',b',c')=(a+a',b+b',c+c'+ab'-ba')\, .
\end{equation}
Note that the inverse of $(a,b,c)$ is $(-a,-b,-c)$.
The center ${\rm Center}(\H)$
consists of the $1$-dimensional
subgroup
$\{0\}\times\{0\}\times \R$.
There is a natural projection
$\pi:\H\to\H/{\rm Center}(\H)= \R^2$ and the {\it cosets of the center}
are just {\it vertical lines} in $\R^3$, i.e., lines parallel to the $z$-axis.

Since correction term, $ab'-a'b$,
in (\ref{groupstruct}), which measures the failure
of the multiplication to be commutative, can
be viewed as a determinant, we get the following
very useful geometric interpretation:

\no
$(*)$
{\it The correction term, $ab'-a'b$, is the  signed area of the parallelagram
spanned by the vectors $\pi((a,b,c))$, $\pi((a',b',c'))$.}

\noindent
$(**)$ {\it Equivalently, if we regard
$\R^3$ as $\R^2\times\R$, then $ab'-a'b$ is the standard symplectic
form, $\omega$,  on $\R^2$}.

Let $K:\R^2\to \R^2$ be  an invertible linear transformation and set
\begin{equation}
\label{automorphism}
   A_K(a,b,c)\defeq(K(a,b),(\det K)\cdot c)\, .
\end{equation}
It is easily checked that $A_K$ is an {\it automorphism} of $\H$.

Let $H_g$ denote affine plane passing through $g\in \H$, which is
the image under left multiplication by $g\in\H$, of the subspace
$\R^2\times\{0\}\subset \H$. We call $H_g$ {\it the horizontal
$2$-plane at $g\in \H$}. Since the automorphism $A_K$
preserves the subspace $\R^2\times\{0\}\subset\H_e$, where
$\H_e$ denotes the tangent
space at the identity $e=(0,0,0)$, it follows that $A_K$
maps horizontal
subspaces to horizontal subspaces.
The collection $\{H_g\}_{g\in \H}$  defines a left-invariant
connection on the principle bundle, ${\R}\to\H\to {\R}^2$,
which in coordinates, has the following explicit description.

   The plane $H_{(0,0,0)}$ is given by $(u,v,0)$ where $u,v$
take arbitrary real values. In general,
$$H_{(a,b,c)}
=(a,b,c)(u',v',0) =(a+u',b+v',c-bu'+av'),$$ so putting
 $a+u'=u$, $b+v'=v$, we get
 \begin{equation}
 \label{horizontal1}
 H_{(a,b,c)}=(u,v,c+av-bu)\, .
 \end{equation}
The affine $2$-planes in $\R^3$ whose projections to $\R^2$ are surjective
are just those which admit a parameterization $(u,v,c+av-bu)$.
It follows that {\it every}
such $2$-plane arises as the horizontal $2$-plane associated
to a unique point $(a,b,c)$, and conversely, that every horizontal
$2$-plane associated to some point in ${\R}^3$ projects surjectively
onto ${\R}^2$.

A line in $\R^3$, which passes through the point $(a,b,c)$, and lies in the plane
$H_{(a,b,c)}$, can be written
\begin{equation}
\label{lines}
L=(a,b,c)+t(u,v,-bu+av)\qquad t\in\R\, ,
\end{equation}
where $u,v$ are fixed and $t$ varies.
In discussing the Heisenberg case,
unless otherwise indicated,
the term {\it line} will refer exclusively to such a {\it horizontal} line.
The collection of all such lines is denoted ${\rm lines}(\H)$.

\begin{definition}
\label{halfspace}
A  {\it half-space} $\P\subset \H=\R^3$
is the set of points lying on one side of some $2$-plane $P$,
including those points of the plane itself.
\end{definition}
The half-space $\h$ is called {\it horizontal} if its associated
$2$-plane $H$ is horizontal.  Otherwise it is called {\it vertical}.
Thus, a vertical half-space is the inverse image under $\pi$
of an ordinary
half plane in $\R^2$.



\subsection*{The Heisenberg group as a PI space.}

Consider the left-invariant Riemannian
metric $\langle\, \cdot\, ,\, \cdot \,\rangle$ on $\H$ which corresponds
to the standard Euclidean metric at the tangent space to the identity.
The {\it Carnot-Carath\'eodory distance} $d^\H(x_1,x_2)$
is defined to be the infimum of lengths of curves
$c:[0,\ell]\to\H$,
from $x_1$ to $x_2$, such that for all $s$, the tangent
vector $c'(s)$ is  {\it horizontal} i.e. $c'(s)$ lies in
the $2$-dimensional subspace of the tangent space
corresponding to the affine plane $H_{c(s)}$.
 The length of $c'(s)$ is calculated with
respect to $\langle\, \cdot\, ,\, \cdot \,\rangle$.
In particular, one sees that for all $g_1,g_2\in \H$,
\begin{equation}
\label{inverses}
d^\H(g_1,g_2)=d^\H\left(g_1^{-1},g_2^{-1}\right)\, .
\end{equation}
A well known consequence (see~\cite{Mont02}) of the definition of $d^\H$, in combination with
$(*), (**)$ above, is that if $p,q\in \H$, and $c :[0,\ell]\to \R^2$ is a curve parametrized by arc length
such that its horizontal lift $\tilde c$ starts at $p$ and ends at $q$ (so that $\pi(p)=\pi(q)$), then the
vertical separation of $p$ and $q$ (in coordinates) is the signed area enclosed by the curve $c$.

Geodesics of $d^H$  can be characterized as those smooth
horizontal curves which project
(locally) either to a circular arc or to a line segment in
$\R^2$.
Thus, the lines $L\in {\rm lines}(\H)$ are precisely those geodesics of
$d^\H$ which are affine lines in $\H=\R^3$.  Any two
points of $\H$ can be joined by a minimal geodesic of
$d^\H$, although typically, {\it not} by a horizonal line
$L\in {\rm lines}(\H)$.
Note that
when viewed as a curve in $\left(\H,d^\H\right)$, any
affine line which is not a member of ${\rm lines}(\H)$ has the property that any of its finite
sub-segments has infinite length.

The distance $d^\H\left((a,b,c),(a',b',c')\right)$, is bounded above and below by
a constant multiple of
\begin{equation}
\label{boxball0}
((a-a')^2+(b-b')^2)^{\frac{1}{2}}+|c-c'+ab'-ba'|^{\frac{1}{2}}\, .
\end{equation}
When restricted to any vertical line is just the {\it square root
of the coordinate distance}.
Thus,
there is a constant $C>0$ such that
for the metric
$d^\H$, the metric ball about $e=(0,0,0)$ satisfies the {\it box-ball principle}:
\begin{equation}
\label{boxball}
B_{C^{-1}r}(e)\subset\left\{(a,b,c)\, \big|\, |a|<r,|b|<r,|c|<r^{\frac{1}{2}}\right\}
\subset B_{Cr}(e)\, .
\end{equation}
Thus, in coordinates, $B_r(e)$ looks roughly like a cylinder
whose base has radius $r$ and whose height is $2r^2$.

In this paper we will consider
the PI space $(\H,d^\H,\L_3)$, where $\L_3$
 denotes Haar measure on $\H$, which coincides with the Lebesgue measure on $\R^3$.

 Let $r_\theta$ denote rotation in $\R^2$
  through an angle $\theta$ and let $I$ denote the identity
on $\R^2$. From now on we write
  $O_\theta$ for $A_{r_\theta}$.
Since $O_\theta$ preserves horizontal subspaces and induces
an isometry on $H_e=\R^2\times\{0\}$, it follows that $O_\theta$
is an isometry of $(\H,d^\H)$.  Clearly, $O_\theta$ preserves the
measure $\L_3$ as well. Similarly, for $\psi\in\R$, the
automorphism $A_{\psi I}$ scales the metric by
a factor $\psi$ and the measure $\L_3$ by a factor $\psi^4$; see (\ref{automorphism}).


\subsection*{Behavior under blow up of finite perimeter cuts of $\H$.} Let
$E\subset B_1(e)\subset\H$ denote a finite perimeter (FP) cut.
Then at $\per(E)$-a.e. $p\in E$,
 asymptotically under blow up, the measure
of the symmetric difference of $E$
and some
vertical half-space $\V$, goes to zero;  see~\cite{italians1}, \cite{italians2}.
Equivalently, the indicator $\chi_E$ converges to $\chi_\V$ in the $L_1^{{\rm loc}}$ sense.
The corresponding theorem for $\R^n$
is due to De Giorgi \cite{degiorgi1}, \cite{degiorgi2}.
The results of  \cite{italians1}, \cite{italians2}
depend essentially on those of
\cite{luigi1}, \cite{luigi2}, in which
an asymptotic doubling property for the perimeter measure is
proved for arbitrary PI spaces.

 If $d_{\mathcal V}$ is an elementary
cut metric associated to a vertical half-space then
the restriction of $d_{\mathcal V}$ to a coset of the center
is trivial. Thus, the results of
  \cite{italians1}, \cite{italians2},
together with  (\ref{cutmetdef1}),
 suggest
that
under blow up,
at almost all points, a
Lipschitz map
$f:\H\to L_1$ becomes degenerate in the direction
of (cosets of) the center. In particular, there exists
no bi-Lipschitz embedding of $\H$ in $L_1$.
This is the heuristic argument behind the main result
of  \cite{ckbv}.

In order to prove Theorem \ref{tmain},
we will take a different approach, leading to a {\it quantitative version}
of a somewhat crude form of
the blow up results of \cite{italians1},
\cite{italians2} (and the corresponding earlier results of \cite{degiorgi1}, \cite{degiorgi2}).
Here, ``crude'' means
that  our argument  does not give {\it uniqueness} of
the blow up, nor in the Heisenberg case, does it show that only
{\it vertical} half-spaces arise.
\cite{italians2}.
 For our purposes, neither of these properties is needed.
Our approach is based on the notion of {\it monotone sets} as introduced
\cite{ckmetmon}; the discussion there, while not quantitative,
does recover the  results on verticality and uniqueness of blow ups.


\subsection*{Monotone sets and half-spaces.}

In the definition that follows and elsewhere in this paper, $E'$ denotes the complement of $E$.
\begin{definition}\label{monotone}
Fix an open set $U\subseteq \H$. Let $\lines(U)$ denote the space of
unparametrized oriented horizontal lines whose intersection with $U$
is nonempty. Let $\mathcal N_U$ denote the unique left invariant
measure on $\lines(\H)$, normalized so that $\mathcal
N_U(\lines(U))=1$. A subset $E\subseteq U$ is {\em monotone with
respect to $U$} if for $\mathcal N_U$-almost every line $L$, both $E\cap L$ and $(U\setminus E)\cap L$ are essentially
connected, in the sense that there exist connected subsets $F_L=F_L(E),F_L'=F_L'(E)\subseteq L$ (i.e., each of $F_L,F_L'$ is either
empty, equals $L$, or is an interval, or a ray in $L$) such that the
symmetric differences $(E\cap L)\triangle F_L$ and $((U\setminus
E)\cap L)\triangle F_L'$ have $1$-dimensional Hausdorff measure $0$.
\end{definition}


Monotone subsets of $\H$
were introduced in \cite{ckmetmon} where they were used to give a
relatively short
proof
of the non-embedding theorem of \cite{ckbv}, which does not require the
introduction of FP sets and hence, does not depend on \cite{italians1}, \cite{italians2}.
Instead,
a blow up argument
 is used to directly reduce the nonembedding theorem to
the special case in which
the cut measure $\Si_f $ is supported on monotone cuts; compare the discussion
in the next subsection. For the case $U=\H$,
a nontrivial classification result asserts that
if $E$
is monotone, then
 $\L_3(E\triangle \P)=0$ for some half-space $\P$;
see \cite{ckmetmon}.

\begin{remark}
If we regard $\H$ as $\R^3$, then
horizontal lines in $\H$ are a particular codimension 1 subset
of the set of all affine lines in $\R^3$.
A typical pair of points lies on no horizontal line. However,
 the classification of monotone subsets
of $\H$ is precisely the  same as for $\R^3$
with its standard metric. In the latter case, the proof is trivial
while in the former case it is not.
\end{remark}


\subsection*{Degeneracy of cut metrics which are supported on half-spaces.}

 Once monotone subsets are known to  be half-spaces
 it follows (after the fact)
that  the connectedness condition
in Definition \ref{monotone} holds for {\it almost every} affine
line
$\underline{L}$, i.e., not
just for horizontal ones.  Thus, if a cut measure $\Sigma$
is supported on monotone cuts,  $d_\Sigma$ has
the  property that  if $x_1,x_2,x_3\in\underline{L}$ and $x_2$ lies
between $x_1$ and $x_3$, then
$d_\Sigma(x_1,x_3)= d_\Sigma(x_1,x_2)+ d_\Sigma(x_2,x_3)$.
But if $\underline{L}$ is {\it not horizontal}, i.e., a coset of the center,
 then $d^\H|_{\underline{L}}$ is comparable to the
square root of the coordinate distance, and it is trivial to verify that
this metric is not bi-Lipschitz to one with the
property mentioned above; see (\ref{cshs}).

In proving Theorem \ref{tmain}, we will show that on a definite scale,
using Theorem \ref{stability}, we can reduce modulo a  controlled error,
 to the case in which  the cut measure is supported on cuts
which are close to half-spaces.  The error term, though controlled, is larger than the term which corresponds to the
model case of monotone cuts.


\subsection*{$\delta$-monotone sets.}
Theorem \ref{stability}, which asserts that an approximately
monotone set is  close to some half-space (in the sense that
the symmetric difference has small measure)
plays a key role in the proof of Theorem \ref{tmain}. But
unlike in \cite{ckmetmon}, we cannot dispense with
consideration of FP sets. We use (\ref{finitetotper}),
the bound on the total perimeter,  to obtain
a bound on the total nonmonotonicity, and hence, to
get the estimate
for a scale on which the total nonmonontonicity is so small
that Theorem \ref{stability} applies; see  the discussion
after Theorem \ref{tmain} and compare Remark \ref{fp} below.

\begin{remark}
\label{fp}
If one restricts attention to $1$-Lipschitz maps $f:B\to L_1$,
 rather than more general BV maps, then using the fact that
the $1$-Lipschitz condition is preserved under
restriction to horizontal lines,
it is possible to derive
the above mentioned
bound on the total nonmonotonicity
without reference to FP sets. However, even for Lipschitz maps,
the introduction of FP sets cannot be avoided; see Lemma \ref{lsmallperiml}
which concerns an issue arising from the fact that
the mass of the cut measure can be infinite; see also Proposition
\ref{lsmallperiml1}.
\end{remark}

We will need to consider
subsets, $E\subset B_R(p)\subset\H$,
which are \hbox{$\delta$-monotone} on $B_R(p)$; see  Definition \ref{deltamonotone}.
Here, small $\delta$ means approximately monotone, and
$\delta=0$ corresponds to the case of monotone sets as in
Definition \ref{monotone}. After rescaling, Theorem \ref{stability} states
that  there exists  a constant $a<\infty$, such that if
$E\subset B_R(p)\subset\H$ is $\epsilon^{a}$-monotone,
and $R\geq \epsilon^{-3}$,
then  $\L_3((E\cap B_1(p))\triangle(\P\cap B_1(p)))\lesssim\epsilon$,
for some half-space $\P$.

\begin{remark}
\label{Rn}
Absent the assumption $R\geq \epsilon^{-3}$,
the conclusion of
Theorem \ref{stability} can fail, even if  ``$\epsilon^{a}$-monotone''
is replaced by ``monotone'';
see Example \ref{counterexample}.
 For the analogous result in $\R^n$, $\epsilon^{-3}$
 can indeed be
replaced by $1$.
\end{remark}

\begin{remark}
\label{laakso}
The discussion of monotone sets can be
formulated for arbitrary PI spaces; see \cite{ckmetdiff}.
But in general, monotone subsets {\it need not be rigid}.
In \cite{laaksoembed}, the flexibility of monotone subsets of Laakso
spaces \cite{laakso} is used to {\it construct}  bi-Lipschitz embeddings
of Laakso spaces into $L_1$. This is of interest since these spaces do not admit a bi-Lipschitz embedding
into any Banach space with the Radon-Nikodym Property, e.g.,
separable dual spaces, such as $L_p$, $1<p<\infty$ or $\ell_1$; see \cite{ckdppi}.
\end{remark}


\subsection*{The kinematic formula; perimeter and nonmonotonicity.}
To bound the total nonmonotoncity in terms of the total perimeter,
 we use a {\it kinematic formula}
for the perimeter of an FP set.
 From now on, we will just refer to {\it the} kinematic formula.
In $\R^n$,
the kinematic formula expresses
the perimeter of an FP set $E$,
as an integral with
respect to the natural measure on the space of lines $L$,
of the $1$-dimensional perimeter function $\Per(E\cap L)$.
In \cite{montefalcone}, a suitable kinematic formula is been proved for
Carnot groups;  see (\ref{kinematic}).
In that context, a {\it line} means a horizontal line $L\in{\rm lines}(\H)$.

Up to a set of measure zero, an FP
subset of ${\R}$ is a finite union of disjoint closed intervals
and the perimeter, $\Per(E\cap L)$, is  the number of end points
of these intervals;
see \cite{MR1857292}. On the other hand,
the condition that  $E$ is monotone can be reformulated
as the requirement that for almost every line $L$, the
$1$-dimensional perimeter
$\Per(E\cap L)$ is either $0$ or $1$.
Since our initial
quantitative data provides an integral
bound on the mass of the
total perimeter
measure
(see (\ref{finitetotper}), (\ref{standardbv}))
 it is not surprising that the kinematic
formula  plays a key role in our discussion.


\section{Preliminaries}
\label{rescale}

In this section, $X$ denotes a PI space. In particular, it could by $\R^n$
or $\H$. Fix $p\in X$ and let $f:B_r(p)\to L_1$ denote a Lipschitz map. In this paper we will often
rescale this ball to unit size and correspondingly
rescale the map $f$, and hence, the induced
metric, $d_{\Si_f}$, or equivalently, the cut measure.
 Finally, we rescale the measure,
$\mu$, so that the rescaled ball has unit measure.

We set
\begin{equation}
\label{fscale}
\breve{f}=r^{-1}f\, ,
\end{equation}
\begin{equation}
\label{cutmeaserescale}
\Si_{\breve{f}}= r^{-1}\cdot\Si_f\, ,
\end{equation}
\begin{equation}
\label{metrescale}
\breve{d}^X= r^{-1}\cdot d^X\, ,
\end{equation}
\begin{equation}
\label{measrescale}
\breve{\mu}(\,\cdot\,)=\frac{\mu(\,\cdot\,)}{\mu(B_r(p))}\, .
\end{equation}
In generalizing our considerations to BV functions, or in
particular, sets of finite perimeter, it is of interest to
consider as well, the effect of rescaling as in (\ref{metrescale}),
     (\ref{measrescale}),
with (\ref{fscale}), (\ref{cutmeaserescale}) omitted.

  Note that $\per(E)(U)$ is defined by:
 \begin{equation}
\label{perdef}
\per(E)(U)= \inf_{\{h_i\}}\,\,\liminf_i\int_U \Lip\, h_i\, d\mu\, ,
 \end{equation}
where the infimum is taken
 over all sequences of Lipschitz functions $\{h_i\}_{i=1}^{\infty}$,
with $h_i\stackrel{L_1^{\mathrm{loc}}}{\longrightarrow} \chi_E$.
 Since rescaling as in (\ref{metrescale}),
 has the effect,
 \begin{equation}
 \label{Liprescale}
\Lip(h_i)\to r\cdot\Lip(h_i)\, ,
\end{equation}
it follows that  (\ref{metrescale}), (\ref{measrescale}),
imply
\begin{equation}
\label{perrescale}
\per(E)\to r\cdot \frac{1}{\mu(B_r(p))}\cdot\per(E)\, .
\end{equation}

\begin{remark}
If (\ref{fscale}), (\ref{cutmeaserescale}), are omitted, then (\ref{perrescale}) is a relevant
rescaling.  For the case of FP sets in particular, it is the relevant rescaling.
It can be used to give a quantitative analog of Theorem \ref{tmain} for a single FP
set, $E$,  in which the set of points at  which on a controlled scale,
 $E$ is not close
to a half-space, has small codimension 1 Hausdorff content, as measured
with respect to coverings by balls of small radius.
\end{remark}

If (\ref{fscale}), (\ref{cutmeaserescale}), are not omitted
then by (\ref{cutmeaserescale}), (\ref{perrescale}), the corresponding
rescaling factor for the total perimeter measure, $\lambda_f$, is
$r\cdot \frac{1}{\mu(B_r(p))}\cdot r^{-1}$, i.e.,
\begin{equation}
\label{totperrescale}
\lambda_f\to \frac{1}{\mu(B_r(p))}\cdot\lambda_f\, .
\end{equation}

\subsection*{Normalization.}

 After rescaling $d^X,\mu,f$, as in (\ref{metrescale}),
(\ref{cutmeaserescale}),  (\ref{measrescale}), we will often
denote the rescaling of the ball, $B_r(x)$ as $\breve{B}_r(x)$.
Thus,
\begin{equation}
\label{standardrad}
\breve{B}_r(x)\subset X\, ,
\end{equation}
\begin{equation}
\label{standardmeas}
\breve{\mu}\left(\breve{B}_r(x)\right)=1\, ,
\end{equation}
\begin{equation}
\label{standarddom}
\breve{f}:\breve{B}_r(x)\to L_1\, ,
\end{equation}
\begin{equation}
\label{standardlip}
\Lip\left(\breve{f}\right)=1\, ,
\end{equation}
or more generally,
\begin{equation}
\label{standardbv}
\int_{\cut\left(\breve{B}_r(x)\right)} \per(E)\left(\breve{B}_r(x)\right)\, d\Si_{\breve{f}}(E)\leq 1\, .
\end{equation}

\subsection*{Standard inequalities}

We will make repeated use of the trivial inequality (often called
Markov's inequality),
which states that for  a measure space $(Y,\nu)$
and $f\in L_1(Y,\nu)$, one has
\begin{equation}
\label{cheb}
\nu(\{x\, |\,f(x)\geq t\})\leq t^{-1}\int_Y|f|\, d\nu\, .
\end{equation}

We also use the weak-type (1,1) inequality
for the maximal function, which is a consequence
of the doubling property of the measure; see e.g. Chapter 1 of \cite{stein}.
We now  recall this basic estimate.

Let $\beta$ denote the doubling constant of $\mu$.  Let $\zeta$
denote a Radon measure.
Fix an open subset $U$ and let $\mathcal C$ denote the collection of closed
balls $\overline{B_r(y)}$ such that
$\overline{B_{5r}(y)}\subset U$.
Given $k,j\in \N$, define
 $\mathcal B_{j,k}=\mathcal B_{j,k}(\zeta)
\subset {\mathcal C}$,
by
\begin{equation}
\mathcal B_{j,k}
\defeq\,\left \{\overline{B_r(y)}\in \mathcal C\, |\,  \zeta\left(\overline{B_r(y)}\right)
\geq kr^{-j}\cdot
\mu\left(\overline{B_r(y)}\right)\right\}\, .
\end{equation}
and set
\begin{equation}
B_{j,k}=\bigcup_{\overline{B_r(q)}\in{\mathcal B_{j,k}}}\overline{B_r(q)}\, .
\end{equation}
By the standard covering argument (see~\cite{stein,juhabook}), there is a sub-collection
$\{B_{r_i}(q_i)\}_{i=1}^\infty\subset {\mathcal B}_{j,k}$, such that
\begin{equation}
\label{badtotalperim0}
B_{j,k}\subset \bigcup_{i=1}^\infty \overline{B_{5r_i}}(q_i)\, ,
\end{equation}
\begin{equation}
\label{badtotalperim1}
\sum_{i=1}^\infty
r_i^{-j}\cdot\mu(B_{r_i}(q_i))\leq k^{-1}\cdot\beta^3\cdot
 \zeta(U)\, .
\end{equation}

\section{Reduction to the stability of individual monotone sets}
\label{outline}

In this section we reduce the proof of the main degeneration
theorem, Theorem \ref{tmain},
to Theorem \ref{stability},
a stability theorem for
 individual
 monotone sets, which states roughly, that a set which
is almost monotone is almost a half space. The complete proof of Theorem \ref{stability}
will occupy Sections \ref{qbqi}--\ref{pfleqi}.

The next few paragraphs contain an overview of this section.
 For definiteness,
we will restrict attention to the Heisenberg group.
 Everything we
say applies, mutadis mutandis, to the simpler case of $\R^n$ as well.

We begin by considering an ideal case, the proof of which is
given prior to the more general Proposition \ref{stability1}.
Then we explain how to reduce to ideal case up to an error which is controlled.

Let $d_\P$ denote a cut metric for which the cut measure
is supported on cuts $\P$ which are half-spaces.  Let $\epsilon>0$.
 We will see by an easy
argument that for every
affine line $\underline{L}$ which makes an angle $\geq \theta>0$ with the horizontal plane, for
a
definite fraction of the pairs of
points $x_1,x_2\in \underline{L}$, with $d^\H(x_1,x_2)=\epsilon$, we have
\begin{equation}
\label{cshs}
d_P(x_1,x_2)\lesssim_\theta \epsilon\cdot d^\H(x_1,x_2)\, ,
\end{equation}
where the implies constant in~\eqref{cshs} depends only on $\theta$.

In proving Theorem \ref{tmain}, we first find a scale such that at a typical location,
the support of the cut measure consists almost entirely of
cuts which are almost half-spaces, i.e., cuts which differ from  half-spaces
by sets of small measure.
 As explained below, for this, it suffices to find a scale
on which the total nonmonotonicity is small.
By the pigeon hole principle, it is easy
to find a scale on which the total perimeter is small.
On such a scale we apply Proposition \ref{tpcpnm}, which
asserts that total nonmonotonicity of a cut metric
can be bounded in terms of the the total perimeter.
Proposition \ref{tpcpnm} is proved in Section \ref{kf}.

On a scale on which the total nonmonotonicity is small,
 the effect of cuts which are not almost
half-spaces can be absorbed into the error term, after which they can be ignored.
However, to get to a situation in
which our conclusion can be obtained by appying Theorem \ref{stability},
we must further
reduce  to one in which there is
a suitable bound on the mass of the cut measure.  Otherwise, the total effect of the
small deviations of the remaining individual cuts from being half-spaces could
carry us uncontrollably
far from the ideal case considered in (\ref{cshs}). This point
is addressed in  Lemma \ref{lsmallperiml},
which states that the cuts  can be decomposed into a subset which
makes a contribution which can be absorbed into the error
term for the cut metric, and one for which the mass has
a definite bound.  Lemma \ref{lsmallperiml} is proved in Section \ref{w11}.


\subsection*{Controlling the cut measure.}
The following lemma will be applied to rescaled balls, on a scale
on which the total nonmonoticity is sufficiently small.
 For $f$ Lipschitz, relation (\ref{tpb}) in the hypothesis of Lemma \ref{lsmallperiml}
holds for all such balls. If more generally $f$ is BV, then for most such balls,
it holds after suitably controlled
rescaling.

Given  metrics, $d,d'$, define $\|d-d'\|_{L^1}$ as in (\ref{ellonedist}).
 \begin{lemma}
\label{lsmallperiml}
Let $f:B_1(p)\to L_1$ satisfy
\begin{equation}
\label{tpb}
\lambda_f(B_1(p))\leq 1\, .
\end{equation}
Given $\eta>0$, the support of $\Sigma_f$ can  be written as a disjoint union
$D_1\cup D_2$, such that if $d_f=d_1+d_2$ denotes
the corresponding decomposition
 of the cut metric $d_f$, then
\begin{equation}
\label{smallperim11}
\Si_f(D_1)\le\eta^{-3}\, ,
\end{equation}
\begin{equation}
\label{smallperim01}
\|d_f-d_1\|_{L_1}\lesssim \eta\, .
\end{equation}
If $f$ is $1$-Lipschitz, then so are $d_1,d_2$.
\end{lemma}

Lemma \ref{lsmallperiml} is a consequence of the more general
Proposition \ref{lsmallperiml1} which is proved in Section \ref{w11}.


\subsection*{Stability of monotone sets.} Let ${\rm lines}(\H)$ denote
the space of unparameterized  oriented horizontal lines, and let
${\rm lines}(U)$ denote the collection
of horizontal lines whose  intersection with $U$
is nonempty. Let
$\mathcal N$ denote the unique left-invariant measure
on ${\rm lines}(\H)$, normalized so that
\begin{equation}
\label{Nnorm}
\mathcal N({\rm lines}(B_1(e))=1\, .
\end{equation}
Let $\mathcal H_L^1$ denote
$1$-dimensional Hausdorff measure
on $L\in{\rm lines}(\H)$ with respect to the metric
induced from $d^\H$.
We note that if $E\subset \H$ is measurable, then $L\cap E$ is measurable for $\mathcal N$-a.e.
$L\in{\rm lines}(\H)$.

   Given a ball $B_r(x)\subset \H$ and
  $L\in {\rm lines}(B_r(x))$, we
define  the {\it nonconvexity} of $(E,L)$ on $B_r(x)$, denoted
${\rm NC}_{B_r(x)}(E,L)$, by
\begin{equation}
\label{dnc}
{\rm NC}_{B_r(x)}(E,L)\defeq\inf\left\{\left.\int_{L\cap B_r(x)}
                             |\chi_I-\chi_{E\cap L\cap B_r(x)}|\, d\mathcal H_L^1\, \right|\ I\subset L\cap B_r(x)\ \mathrm{subinterval}\right\}\, ,
\end{equation}
where we allow in the infimum above the subinterval $I$ to be empty.
Similarly, we define the {\it nonmonotonicity} of $(E,L)$ on $B_r(x)$ by
\begin{equation}
\label{dnm}
{\rm NM}_{B_r(x)}(E,L)\defeq {\rm NC}_{B_r(x)}(E,L)+{\rm NC}_{B_r(x)}(E',L)\, .
\end{equation}
The {\em total nonconvexity} and {\em total nonmonotonicity} of $E$ on $B_r(X)$ are defined to be the following (scale invariant) quantities:
\begin{equation}\label{eq:def NC}
NC_{B_r(x)}(E)\defeq\frac{1}{r^4} \int_{{\rm lines}(B_r(x))}{\rm NC}_{B_r(x)}(E,L) \,d \mathcal N(L)\, ,
\end{equation}
\begin{equation}\label{eq:def NM}
NM_{B_r(x)}(E)\defeq\frac{1}{r^4} \int_{{\rm lines}(B_r(x))}{\rm NM}_{B_r(x)}(E,L) \,d \mathcal N(L)\, ,
\end{equation}
Note that ${\rm NM}_{B_r(x)}(E)=0$ if $E$ is monotone,
or more generally if the
the
symmetric difference of $E$ and some
monotone subset has measure zero.
\begin{definition}
\label{deltamonotone}
A cut
$E\subset  B_r(p)\subset\H$ is said to be {\it $\de$-monotone} on $B_r(x)$ if
\begin{equation}
\label{deltaconvexset}
{\rm NM}_{B_r(x)}(E)<\delta\, .
\end{equation}
\end{definition}

We have the following stability theorem.
\begin{theorem}
\label{stability}
  There exists $a>0$ (e.g., $a=2^{52}$ works here),
such that if the
cut $E\subset B_1(x)$ is
$\epsilon^{a}$-monotone on $B_1(x)$,
  then there exists a half-space $\mathcal P\subseteq \H$ such that
$$
\frac{\L_3\left((E\cap B_{\e^3}(x))\triangle \P\right)}{\L_3\left(B_{\e^3}(x)\right)}\lesssim \e\, .
$$
\end{theorem}


\subsection*{Cut measures supported on cuts which are almost half-spaces.}
In this subsection, we begin by verifying (\ref{cshs}), which concerns
the  case of a cut measure which is supported on half-spaces. Then
we consider
  in Proposition \ref{stability1}, the more general case of a
 cut measure which is supported on cuts which are almost
half-spaces, i.e., cuts that satisfy the
 the conclusion of Theorem \ref{stability}. Of course, there is an error term
in the general case, which  leads to a weaker estimate
than in the ideal case.

\begin{proof}[Proof of (\ref{cshs})]
Consider first an elementary cut metric $d_E$ on
the real line, associated to a
subset $E$ such that $E$ and $E'$ are connected. Clearly, if $x_3$
lies between $x_1$ and $x_2$, then
$$
d_E(x_1,x_3)+d_E(x_3,x_2)=
d_E(x_1,x_2)\, .
$$
By linearity, this  holds more generally for cut metrics on the line whose
cut measures are supported on such cuts.

Let $\Sigma_\P$ denote a cut measure on $B_1(p)\subset \H$
such that every  cut in the support of $\Sigma_\P$
is of the form $\P\cap B_1(p)$, for some half-space
$\P$. Then for almost every affine (not necessarily horizontal) line $\underline{L}$, the restriction
$d_\P|_{\underline{L}}$
is a cut metric as in the previous paragraph.

 Consider a subinterval
$I\subset \underline{L}\cap B_1(p)$ with end points $y,z$
such that say $d^\H(y,z)=\frac{1}{2}$.
Let $\epsilon$ be as in (\ref{cshs})
and let $\psi>0$ be arbitrary.
Let $\alpha$ index
those intervals $J_\alpha\subset I$ whose end points $x_{1,\alpha},x_{2,\alpha}$ satisfy
$d^\H(x_{1,\alpha},x_{2,\alpha})=\epsilon $,
and such that $d_\P(x_{1,\alpha},x_{2,\alpha})\geq\psi\cdot d^\H(x_1,x_2)$.
Note that if $\underline{L}$ is as in~\eqref{cshs} (i.e., it makes a definite angle $\theta$ with the horizontal plane) then
$\L(J_\alpha)\asymp_\theta\epsilon\cdot d^\H(x_{1,\alpha},x_{2,\alpha})=\epsilon^2$.

By the standard covering argument, there is a disjoint sub-collection $\{J_{\alpha_1},\ldots, J_{\alpha_N}\}$ of $\{J_\alpha\}$, such that
$\bigcup_\alpha J_\alpha\subset (5J_{\alpha_1})\cup \cdots \cup (5J_{\alpha_N})$,
and by the length space property of $d_\P$, we have
$$N\cdot \psi\epsilon \le d_\P(x_{1,1},x_{2,1})+\cdots+d_\P(x_{1,N},x_{2,N})\leq d_\P(y,z)=\frac{1}{2}\, .$$
Thus, $N\leq \frac{1}{2}(\psi\cdot\epsilon)^{-1}$ and
$$
\L\left(\bigcup_\alpha J_\alpha\right)
\lesssim_\theta \frac{5}{2} \psi^{-1}\cdot \epsilon^{-1}\cdot \epsilon^2=\frac{5}{2} \psi^{-1}\cdot \epsilon\, .
$$
By taking $\psi=\epsilon$ we get  (\ref{cshs}).
\end{proof}

To simplify the statement of Proposition \ref{stability1},
it is stated for unit balls.  In the application (to Lipschitz maps)
we will set
$\epsilon_1=\epsilon^{40}$ and consider the rescaling
as in (\ref{fscale})--(\ref{measrescale})
of $B_{\epsilon_1^3}(x)$
to unit size.

\begin{proposition}
\label{stability1}
Fix $\psi>0$. Assume that $h:B_1(p)\to L_1$ is such that
$\Sigma_h$ satisfies the total perimeter bound (\ref{tpb}),
and assume that $\Sigma_h$ is supported on cuts $E\subset B_1(p)$, such that for
some half-space $\P_E$,
\begin{equation}
\label{es1}
\L_3((\P_E\cap B_1(p))\triangle E) \lesssim (\psi\cdot\epsilon^4)^4\, .
\end{equation}
Then for at least a fraction $1-\psi$ of the affine lines
$\underline{L}$ which
make an angle at most $\frac{\pi}{3}$ with the vertical \footnote{${}$
$\frac{\pi}{3}$ can be replaced
by any number $<\frac{\pi}{2}$.}
such that $\underline{L}\cap B_1(x)\ne\emptyset$,
and at least two thirds of the
$(x_1,x_2)\in \left(\underline{L}\cap B_1(x)\right)\times \left(\underline{L}\cap B_1(x)\right)$,
 satisfying
\begin{equation}
\label{stability11}
\frac{1}{2}\epsilon\leq d^\H(x_1,x_2)\leq \frac{3}{2}\epsilon\, ,
\end{equation}
there holds
\begin{equation}
\label{stability12}
d_h(x_1,x_2)\leq\epsilon\cdot d^\H(x_1,x_2)\, .
\end{equation}
If $h$ is $1$-Lipschitz and
\begin{equation}
\label{es111}
\L_3((\P_E\cap B_1(p))\triangle E) \lesssim \epsilon^{48}\, ,
\end{equation}
 then for all such $\underline{L}$
and at least half the points satisfying (\ref{stability11}),
\begin{equation}
\label{cshs1}
d_h(x_1,x_2)\lesssim  \epsilon\cdot d^\H(x_1,x_2)\, .
\end{equation}
\end{proposition}
\begin{proof}
Assume first that (\ref{es1}) holds.
Let $D_1,D_2$ and $d_h=d_1+d_2$ be as in Lemma \ref{lsmallperiml}
with $\eta=\psi\cdot\epsilon^{4}$.
Hence, $\Si_h(D_1)\leq (\psi\cdot\epsilon^4)^{-3}$ and
$\|d_h-d_1\|_{L_1}\lesssim\psi\cdot\epsilon^4$.

Using (\ref{es1}), it is straightforward to construct a half space
 $\P_E$, which varies measurably with respect
to $\Sigma_{h}$ (alternatively one can use a measurable selection theorem as in~\cite{Kuratowski-selection}), such that
$\L_3((\P_E\cap B_1(p))\triangle E) \lesssim (\psi\cdot\epsilon^4)^4$. Put
\begin{equation}
\label{decomp}
d_{\P,1}(x_1,x_2)
=\int_{D_1}d_{\P_E}(x_1,x_2)\, d\Sigma_1(E)\, .
\end{equation}
 From  $\Si_h(D_1)\leq (\psi\cdot\epsilon^4)^{-3}$,
$\L_3((\P_E\cap B_1(p))\triangle E) \lesssim (\psi\cdot\epsilon^4)^4$,
we obtain
\begin{equation}
\label{lsm}
\|d_1-d_{\P,1}\|_{L_1}\lesssim\psi\cdot\epsilon^4\, .
\end{equation}
This,  together with $\|d_h-d_1\|_{L_1}\lesssim\psi\cdot\epsilon^4$, implies
\begin{equation}
\label{lsm1}
\|d_h-d_{\P,1}\|_{L^1}\lesssim\psi\cdot\epsilon^4\, .
\end{equation}

Write the integral in the definition of the $L_1$-norm in
(\ref{lsm1}) as iterated integral,
integrating first over pairs of points lying on a given affine
line and then over the space of affine lines. It follows from Markov's inequality~\eqref{cheb}
that for a fraction $\geq 1-\psi$ of affine lines $\underline{L}$ as above,
\begin{equation}
\label{lsm2}
\|d_h-d_{\P,1}\|_{L_1\left(\underline{L}\times \underline{L},\L\times
\L\right)}\lesssim\epsilon^4\, .
\end{equation}


On such a line $\underline{L}$,  consider the
set of pairs satisfying (\ref{stability11}),
 for which (\ref{cshs}) holds for the metric
$d_{\P,1}$.  By noting that the measure of the space of pairs of
points on $\underline{L}$ satisfying
(\ref{stability11}) is $\gtrsim\epsilon^2$, and applying
Markov's inequality once more,
 on $\underline{L}\times\underline{L}$ to (\ref{lsm2}), we find that for
at least half of the pairs points at which (\ref{cshs}) holds for $d_{\P,1}$,
the error term arising from (\ref{lsm2}) is $\lesssim\epsilon^2$.  This
gives (\ref{stability12}).

To obtain (\ref{cshs1}), note that
since the space of affine lines has dimension $4$,  by taking
$\psi=\epsilon^8=(\epsilon^2)^4$, it follows that the set of lines
intersecting $B_1(p)$, for which
(\ref{cshs1}) holds is $\lesssim\epsilon^2$-dense with respect $d^\H$
in the space of all such lines. In case $h$ is $1$-Lipschitz, this
gives (\ref{cshs1})
for all $\underline{L}$.
\end{proof}
\subsection*{Scale estimate; total perimeter and total nonmonotonicity.}
 Next, we show how to estimate from below a scale on which,
apart from a collection of cuts which contributes negligibly
to the cut metric $d_f$,
the hypothesis of Theorem \ref{stability1} will be satisfied at
most locations; see Proposition \ref{pscest}.  It is from this estimate
that the logarithmic behavior in \eqref{compression} of Theorem \ref{tmain} arises.

Recall that the cut measure $\Si_f$ is supported on cuts $E$ with finite
perimeter.
 Fix $\delta>0$. Using the structure of finite perimeter subsets of
$\H$ and the kinematic formula, in Section \ref{kf} we will
decompose the total perimeter
measure $\lambda_f$
 as a sum,
\begin{equation}
\label{lambdascales}
\lambda_f=\sum_{j=0}^\infty\widehat{w}_j\, ,
\end{equation}
in such a way that the measure $\widehat{w}_j$ controls the {\it total non-monotonicity}
on the scale $\delta^j$ in the  following sense.


 Fix $j>0$
and as usual, let $\breve{B}_{\frac{1}{4}\delta^j}(x)$ denote the standard rescaling of
 the ball $B_{\frac{1}{4}\delta^j}(x)\subset B_1(p)$ as in
(\ref{standardrad})--(\ref{standardbv}).
\begin{proposition}
\label{tpcpnm}
If for all $r>0$,
\begin{equation}
\label{1,1}
\lambda_f(B_r(x))\lesssim \L_3(B_r(x))\, ,
\end{equation}
then
\begin{equation}
\label{nmest}
\int_{\cut\left(\breve{B}_{\frac{1}{4}\delta^j}(x)\right)}{\rm NM}_{\breve{B}_{\frac{1}{4}\delta^j}(x)}(E)\, d\Si_f(E)
\lesssim\frac{\widehat{w}_j\left(B_{\frac{1}{4}\delta^j}(x)\right)}{\L_3\left(B_{\frac{1}{4}\delta^j}(x)\right)}
+\delta\, .
\end{equation}
\end{proposition}

In order to apply Proposition \ref{tpcpnm},
we need to find $j$  such that $\widehat{w}_j(B_1(p))\leq \delta$,
and hence, at most locations,
$x$, the term
$\frac{\widehat{w}_j\left(B_{\frac{1}{4}\delta^j}(x)\right)}{\L_3\left(B_{\frac{1}{4}\delta^j}(x)\right)}$ is
$\lesssim \delta$.

\begin{proposition}
\label{pscest}
There exists $j\leq \delta^{-1}$ for which
\begin{equation}
\label{scest}
\mass\left(\widehat{w}_j\right)\lesssim \delta\, .
\end{equation}
 For such $j$ and at least half the points $x$ in
$B_1(p)$,
\begin{equation}
\label{wj1,1}
\int_{\cut\left(\breve{B}_{\frac{1}{4}\delta^j}(x)\right)}{\rm NM}_{\breve{B}_{\frac{1}{4}\delta^j}(x)}(E)\, d\Si_f(E)\lesssim \delta\, .
\end{equation}
\end{proposition}
\begin{proof}
We have
\begin{equation}
\label{scales}
1\gtrsim \mass (\lambda_f )
                 =\sum_{j=0}^\infty \, \mass(\widehat{w}_j)\, .
\end{equation}
Thus the number of terms in the sum above
for which
$
\mass\left(\widehat{w}_j\right)\geq \delta
$
is bounded by $\lesssim \delta^{-1}$, and the claim follows.
The conclusion follows from the weak-type (1,1) inequality for the
maximal function
applied to the measure
$\widehat{w}_j$, i.e., \eqref{badtotalperim1}, together with Proposition
\ref{tpcpnm}.
\end{proof}

\subsection*{Proof of Theorem \ref{tmain}.}
Let $a$
 be as in Theorem \ref{stability}.
 Fix $\epsilon>0$ and let
\begin{equation}
\label{deltachoice}
\delta=c\epsilon^{24(6+a)+2}\,
\end{equation}
for an appropriately small enough constant $c>0$.
By Proposition \ref{pscest},
can choose $j\lesssim\delta^{-1}$, such that
(\ref{wj1,1}) holds for at least half the points
$x$ in $B_1(p)$.
Below we restrict attention
to such a point $x$ and rescale $B_{\frac14\delta^j}(x)$ to standard size,
denoting this ball as usual by $\breve{B}_{\frac14\delta^j}(x)$.

By (\ref{nmest}), (\ref{wj1,1}) and Markov's inequality~\eqref{cheb}, we can write $\cut\left(\breve{B}_{\frac14\delta^j}(x)\right)$ as a disjoint union
$\cut\left(\breve{B}_{\frac14\delta^j}(x)\right)=D_3\cup D_4$, with corresponding metric decomposition
$d_f=d_3+d_4$, where
\begin{equation}
\label{rem1}
\Sigma_f(D_3)\leq c\epsilon^{ 24\cdot 6+2}\, ,
\end{equation}
\begin{equation}
\label{rem2}
{\rm NM}_{\breve{B}_{\frac14\delta^j}(x)}(E)\leq \epsilon^{24a}\qquad \forall\ E\in D_4\, .
\end{equation}
By (\ref{rem1}), we have
$d_3(x_1,x_2)<\epsilon^{24\cdot 6+2}$, for all $x_1,x_2$.
Therefore, if $d_f$ is restricted to
$B_{\epsilon^{24\cdot 6}}(x)$ and this ball is
rescaled to unit size, we get
\begin{equation}
\label{rem3}
|d_f-d_3|\leq \epsilon^2 \, ,
\end{equation}
while by (\ref{rem2}) and Theorem \ref{stability}, $d_3$ is supported on cuts
satisfy which satisfy the hyposthesis
of  Proposition \ref{stability1}. From Proposition \ref{stability1},
together with (\ref{rem3}), Theorem \ref{tmain} follows.
 \qed


\subsection*{Preview of the proof of Theorem \ref{stability}.}

After two preliminary  technical results  on
$\delta$-monotone subsets of $\H$
have been stated in Section \ref{qbqi},
and the classification of monontone subsets of $\H$ has been
reviewed in Section \ref{cms}, Theorem \ref{stability}
is proved in Section \ref{s:dm}. However, there is
a technical step in the argument (the nondegeneracy of the
initial configuration) the proof of which,
for reasons of exposition, is deferred until
Section \ref{ndic}.
The proofs of the technical results,
 Lemma \ref{leqi} and Proposition \ref{pdrrvc}, are
deferred to
Sections \ref{pfleqi}, \ref{pfpdrrvc}.
The proof of Proposition \ref{pdrrvc} is (by far) the most involved
part of this  paper.

\section{Cuts with small perimeter}
\label{w11}

In this section we prove Lemma \ref{lsmallperiml}, which shows
that cuts with very small perimeter make a negiligible contribution to
the cut metric $d_f$.  It is most natural to argue here in the context of general PI spaces.
Thus, let $X$ denote a PI space, for example, $\R^n$ or $\H$.  Fix
$p\in X$ and let $f:B_1(p)\to L_1$ satisfy,
${\rm Lip}(f)\leq 1$.

 Put
\begin{equation}
\label{Ddef11}
\begin{aligned}
D_1&=\{E\, |\, \per(E)(B_1(x))\leq \theta\}\, ,\\
D_2&=\{E\, |\, \per(E)(B_1(x))> \theta\}\, ,
\end{aligned}
\end{equation}
and let $d_f=d_1+d_2$ denote the corresponding
decomposition
of the metric $d_f$. Let $\beta$ denote the doubling constant of $X$,
and $\tau'$, $\chi>1$,
the constants in the Poincar\'e-Sobolev inequality
(\ref{poincaresob}).

\begin{proposition}
\label{lsmallperiml1}
 If $f:B_1(p)\to L_1$ satisfies
\begin{equation}
\label{tpb1}
\lambda_f(B_1(x))\leq K\, ,
\end{equation}
then for all $\theta>0$,
\begin{equation}
\label{smallperim111}
\Sigma_f(D_2)\leq K\theta^{-1}\, ,
\end{equation}
\begin{equation}
\label{smallperim011}
\|d_f-d_2\|_{L_1}\leq \frac{2\tau'K}{1-2^{\chi-1}}\cdot \theta^{\chi-1}\, .
\end{equation}
\end{proposition}
\begin{proof}
 From (\ref{tpb1}) and Markov's inequality~\eqref{cheb},
 we get (\ref{smallperim11}).

For $n\in \N$ put
$$
D_{1,n}=\{E\, |\, \theta\cdot 2^{-(n+1)}\leq\per(E)(B_1(x))
\leq\theta\cdot 2^{-n}\}\, .
$$
Then  $D_1=\bigcup_{n\in \N} D_{1,n}$.
By (\ref{tpb1}) and Markov's inequality once more, we have
$$
\Si_f(D_{1,n})\leq K2^{n+1}\theta^{-1}\, .
$$
Moreover, if $E\in D_{1,n}$ then by
 (\ref{poincaresob}),
\begin{equation}
\label{poincaresobeta}
\mint_{B_1(x)\times B_1(x)}
|\chi_E(x_1)-\chi_E(x_2)|
\, d\mu\times d\mu
\leq
\tau'\cdot (\theta\cdot 2^{-n})^\chi\, .
\end{equation}
By summing over $n$ we get (\ref{smallperim011}).
\end{proof}

Note that when $\mathrm{Lip}(f)\le 1$, by virtue of~\eqref{finitetotper}, \eqref{tpb1} holds
with $K=c(\beta,\tau')$. To  get
(\ref{stability11}), (\ref{stability12}) from (\ref{smallperim111}), (\ref{smallperim011}),
we take $\theta$ to be a suitable multiple of $\eta^3$
(noting that for $X=\H$,
we have $\beta=16$, $\chi=\frac{4}{3}$).

\section{The kinematic formula and $\delta$-monotone
sets}
\label{kf}

In this section, using the kinematic formula, we
decompose the total perimeter measure as a sum of measures
$\lambda_f=w_1+w_2\cdots$. Then
we prove Proposition \ref{tpcpnm},
which states that $w_j$ controls the total nonmonotonicity
on the scale $\delta^j$.

We rely on the simple structure
of sets of finite perimeter in dimension 1 and on
the kinematic formula, which expresses the perimeter
of an FP subset $E$ of $\R^n$ or a Carnot group, as an integral
over the space of lines $L$, of the perimeters of
the $1$-dimensional FP sets $E\cap L$.


\subsection*{FP sets in dimension 1.}

 Let $V$ denote an open subset of $\R$
and let $F\subset V$ have
finite perimeter, i.e.,  $\per(F)(V)<\infty$. Then there exists a unique
collection of finitely many
disjoint intervals, $I_1(F),\dots, I_N(F)$, which are relatively
closed in $V$, such that the symmetric difference
of $F$ and $\I(F)=\bigcup_{i=1}^N I_i(F)$ has
 measure zero. Moreover, the perimeter measure
${\rm per}(F)$, is a sum of delta functions concentrated
at the endpoints of these intervals, and
the perimeter $\per(F)(V)$  is equal to
the number of endpoints;
 see Proposition 3.52 of \cite{MR1857292}.
In what follows, we will often  assume
without explicit mention that the set $F\subset \R$,
has been replaced by
its precise representative $\I(F)=\bigcup_{i=1}^N I_i(F)$.
Note that $\I(F'\cap V)=\I(F)'$, ${\rm per}(F')={\rm per}(F)$.


\subsection*{The kinematic formula.}
A  kinematic formula exists  for the PI space
$(\H,d^X,\L)$, and more generally,
for any Carnot group, an in particular also for  $\R^n$;
see  Proposition 3.13, of \cite{montefalcone}.
Below,  the notation is as introduced prior to (\ref{Nnorm}).

Let $U\subset \H$
denote an open subset such that
$\per(E)(U)<\infty$.  The kinematic  formula
 states that
the function $L\mapsto \per(E\cap L)(U\cap L)$ 
lies in $L_1({\rm lines}(U),\mathcal N)$,
 and in addition (for some constant $c=c(\H)>0$)
\begin{equation}
\label{kinematic}
\per(E)(U)=c\cdot\int_{{\rm lines}(U)}
\per(E\cap L)(U\cap L)\, d\mathcal N(L)\, .
\end{equation}


\subsection*{Perimeter bounds  nonmonotonicity.}
  Fix $0<\delta<1$. For all $E,L$ as above, and $j\geq 0$,
let $C_j(E,L)$ denote the collection of
intervals, $I(E,L)$, occuring in
$\I(E\cap L\cap B_1(x))$
 such that
\begin{equation}
\label{intscale}
\delta^{j+1}
\leq {\rm length}( I(E,L)) < \delta^{j}\, .
\end{equation}

Let $\E_j(E,L)$ denote the collection of all end points of
intervals in $C_j(E,L)$.
Let ${\rm card}(S)$ denote the
cardinality of the set $S$.  For $c(\H)$ as in (\ref{kinematic}) and
$A\subset B_1(x)$,  put
\begin{equation}
\label{wi}
w_j(E,L)(A)=c(\H)\cdot {\rm card}(\{e\in\E_i(E,L)\, |\, e\in A\})\, ,
\end{equation}
\begin{equation}
w_j(E)(A)=\int_{{\rm lines}(B_1(x))}w_j(E,L)(A)\, d\mathcal N(L)\, ,
\end{equation}
\begin{equation}
\label{wi1}
\widehat{w}_j(E)(A)=\widehat{w}_j(E')(A)=\frac{1}{2}\left(w_j(E)(A)+w_j(E')(A)\right)\, .
\end{equation}
 Finally, set
\begin{equation}
\label{wi2}
\widehat{w}_j(A)=\int_{\cut(B_1(x))}\widehat{w}_j(E)(A)
                  \, \,d\Si_f(E)\, .
\end{equation}

By the
the kinematic formula (\ref{kinematic}), we get the decomposition of
of the perimeter measure and total perimeter measure,
\begin{equation}
\label{pbtnm}
    \per(E)=\sum_j \widehat{w}_j(E)\, ,
\end{equation}
\begin{equation}
\label{sumw}
\lambda_f=\sum_{j=0}^\infty\widehat{w}_j\, .
\end{equation}
which is a restatement of (\ref{lambdascales}).

\begin{proof}[Proof of Proposition \ref{tpcpnm}]
 Fix $E$ with finite perimeter. For
$\mathcal N$-a.e. $L\in {\rm lines}\left(B_{\frac{1}{4}\delta^j}(x)\right)$,
we can assume that
$L\cap E\cap B_{\frac{1}{4}\delta^j}(x)$ and
$L\cap E'\cap B_{\frac{1}{4}\delta^j}(x)$
consist of finitely many intervals, $I_1,\ldots, I_N$, in the
natural consecutive ordering. Note that each the intervals $I_2,\ldots,I_{n-1}$,
(including both end points),
are contained in $B_{\frac{1}{4}\delta^j}(x)$.
In particular, these intervals lie in $C_k(E,L)$, for various $k\geq j$.

Since the nonmonotonicity ${\rm NM}_{B_{\frac{1}{4}\delta^j}(x)}(E,L)$
is defined as an infimum over intervals (see (\ref{dnm})),  it can be bounded
from above by employing any specific
interval $I$.
 For definiteness, assume say $I_1\subset E$.
where $I_N=\emptyset$ if $N=1$. If
$I_N\subset E$, put $I=L\cap B_{\frac{1}{4}\delta^j}(x)$.
If $I_N\subset E'$, put $I= I'_1\cap B_{\frac{1}{4}\delta^j}(x)$.

Rescale the ball $B_{\frac{1}{4}\delta^j}(x)$ to unit size as in
(\ref{standardrad})--(\ref{standardbv}).
Then by employing the above chosen interval $I$, we get
\begin{equation}
\label{nmestimate}
{\rm NM}_{B_{\frac{1}{4}\delta^j}(x)}(E,L)\lesssim \sum_{k\ge j} \sum_{I\in C_k(E,L)}{\rm length}(I)\, .
\end{equation}
Since an interval is determined by its endpoints, (\ref{nmestimate}) implies
\begin{equation}
\label{nmestimate1}
{\rm NM}_{B_{\frac{1}{4}\delta^j}(x)}(E,L)\lesssim \sum_{k\geq j}\delta^{k-j}\mass\left(\widehat{w}_k(E,L)\right)\, ,
\end{equation}
and by integrating over the space of lines, we get
\begin{equation}
\label{nmestimate2}
{\rm NM}_{B_{\frac{1}{4}\delta^j}(x)}(E)\lesssim \mass(\widehat{w_j})+\delta \mass(\per(E))\, ,
\end{equation}
which completes the proof of Proposition \ref{tpcpnm}.
\end{proof}

\section{The quantitative interior and boundary}
\label{qbqi}

The proof (though not the statement) of Theorem \ref{stability} utilizes
quantitative notions  of the interior and boundary of
a measurable set $E$.
In the present
section, after defining  these notions, we
state two key properties
for $\delta$-convex and $\delta$-monotone subsets of $\H$;
 see Proposition \ref{pdrrvc},
Lemma \ref{leqi}. We also deduce two particular consequences
of Proposition \ref{pdrrvc} for the structure of
the quantitative boundary of $\delta$-monotone sets. These are used
in the proof of Theorem \ref{stability}.  Assuming Proposition \ref{pdrrvc}
and Lemma \ref{leqi}, the proof of
Theorem \ref{stability} is given in Section \ref{s:dm}.
The proofs of Proposition \ref{pdrrvc} and
Lemma \ref{leqi}, which are the most technical parts of our
discussion, are postponed until
Sections \ref{pfpdrrvc} and  \ref{pfleqi}, respectively.

\begin{definition}
\label{defquantbddy} Let $(X,d^X,\mu)$ be a metric measure space and $E\subset X$ a measurable subset. For $\alpha\in (0,1)$ and $u>0$ define:
\begin{equation}
\label{equanint}
{\rm int}_{\alpha,u}(E)\defeq \left\{x\, \left|\, \frac{\mu(E\cap B_u(x))}{\mu(B_u(x))}
\geq1-\alpha\right \}\right.,
\end{equation}
\begin{equation}
\label{equanbddy}
\partial_{\alpha,u}(E)\defeq \left\{x\, \left|\, \alpha<
\frac{\mu(E\cap B_u(x))}{\mu(B_u(x))}<1-\alpha\right \}\right..
\end{equation}
\end{definition}
Observe that $\partial_{\alpha,u}(E)=\emptyset$ for $\alpha\ge \frac12$. We note the following properties which are trivial consequences of the definitions.

If $\beta\leq \alpha$, then,
\begin{equation}
\label{trivial1}
\begin{aligned}
{\rm int}_{\beta,u}(E) &\subset {\rm int}_{\alpha,u}(E) \\
\partial_{\alpha,u}(E) &\subset\partial_{\beta,u}(E) .
\end{aligned}
\end{equation}
 For $u_2\leq u_1$ and $x\in X$ set
$$
c(u_1,u_2,x)=\frac{\mu(B_{u_1}(x))}{\mu(B_{u_2}(x))}\, .
$$
Then
\begin{equation}
\label{trivial11}
\begin{aligned}
x\in {\rm int}_{\alpha,u_1}(E)&\implies x\in {\rm int}_{c(u_1,u_2,x)\cdot\alpha,u_2}(E) \\
x\in \partial_{c(u_1,u_2,x)\cdot\alpha,u_2}(E)&\implies x\in \partial_{\alpha,u_1}(E) .
\end{aligned}
\end{equation}

If $x\in \partial_{\alpha,u}(E)$  then
\begin{equation}
\label{trivial4}
x\in {\rm int}_{1-\alpha,u}(E)\cap {\rm int}_{1-\alpha,u}(E') .
\end{equation}
 Finally,
\begin{equation}
\label{trivial2}
(\partial_{\alpha,u}\, (E))'={\rm int}_{\alpha,u}(E)\cup {\rm int}_{\alpha,u}(E').
\end{equation}


\subsection*{Quantitative interior of $\delta$-monotone subsets.}
If a
subset $E$ of a ball, $B_1(p)$, is
close to  a half-space,
then it follows that there exists a sub-ball of a definite size which is almost
entirely contained in either $E$ or $E'$.  Conversely, proving that
 $\delta$-monotone subsets have this property constitutes an important step in the
proof that such a set is close to a half-space. (It is the remainder
of argument, based on Proposition \ref{pdrrvc} below, which requires that we shrink
the size of the ball on which the conclusion is obtained.
This is the content of
Lemma \ref{leqi},
 which constitutes both Step A) of the proof in the precisely monotone case,
 treated in Section \ref{cms} and more generally, Step A$'$) of the
case, $\delta>0$, treated in Section \ref{s:dm}.

\begin{lemma}
\label{leqi}
There exists $0<c<\frac{1}{2}$
such that if $E\subset B_r(p)$ is $\epsilon^2$-monotone on $B_r(p)$
then there exists $q\in B_{\frac{1}{2}r}(p)$
such that $q\not\in \partial_{\epsilon, cr}(E) $, i.e.,
$q\in {\rm int}_{\epsilon,cr}(E)\cup {\rm int}_{\epsilon,cr}(E')$.
\end{lemma}


\subsection*{Quantitative convexity of $\delta$-convex sets.}

If $E\subset \R^n$ is convex and  $L$ is a line passing through points $p\in E$ and
$q\in {\rm int}(E)$,  then the segment of $L$ lying between $p$ and $q$ also
consists of interior points of $E$.
Below, for subsets of $\H$ which are almost convex,
we given a quantitative version of this statement.

We say that $E\subset B_r(p)$  is
{\it $\delta$-convex} on $B_r(p)$ if for ${\rm NC}_{B_r(p)}(E,L)$ as in (\ref{dnc}),
\begin{equation}
\label{tnc}
{\rm NC}_{B_r(p)}(E)=\int_{{\rm lines}\left(\breve{B}_r(x)\right)}{\rm NC}(E,L)
\,d \mathcal N(L)<\delta\, .
\end{equation}


\begin{proposition}
\label{pdrrvc} There exists universal constants $c\in (0,1)$ and $C\in (1,\infty)$ with the
following properties. Fix  $\kappa,\eta,\xi,r,\rho\in (0,1)$ such that
 \begin{equation}
\label{rho}
\rho\leq \min\left\{\frac12 \kappa r^2,c\right\}\,  .
\end{equation}
Set
\begin{equation}
\label{delta1first}
\delta_1=\frac{c\kappa^3\eta^2\xi\rho^3}{r}\, ,
\end{equation}
\begin{equation}
\label{delta2}
\delta_2=c\kappa^6\eta^3\xi^2\rho^3r^6\, .
\end{equation}


Let $L\in \lines(\H)$ be parameterized by arc length and assume that
\begin{equation}
\label{definite}
L(0)\in{\rm int}_{1-\eta,\rho}(E)\, ,
\end{equation}
%
\begin{equation}
\label{bigdensity}
L(1)\in {\rm int}_{\delta_1,Cr}(E)\, ,
\end{equation}




\no
and $E\subset \H$ is $\delta_2$-convex on $B_{2C}(L(0))$,
then for all $s\in [\kappa,1]$,
\begin{equation}
\label{weakened}
L(s)\in {\rm int}_{\xi,crs}(E)\, .
\end{equation}
\end{proposition}

\begin{remark}
\label{explanation}
The import of the proposition is the following:
Even if $L(0)$ is only in the very weak quantitative
interior of $E$ ($\eta$ small) we can ensure
that $L(s)$ is in the very strong quantitative interior of $E$ (i.e. $\xi$ small), provided
that $L(1)$ is in the sufficiently strong quantitative interior of $E$
($\delta_1$ sufficiently small) and $E$ is sufficiently convex
($\delta_2$ sufficiently small).
\end{remark}

\begin{remark}
\label{explanation2}
The application to the proof of Theorem
\ref{stability} requires some properties of the quantitative
boundary for $\delta$-monotone sets which are derived in
the next subsection using Proposition \ref{pdrrvc}. The assumption,
$L(0)\in{\rm int}_{1-\eta,\rho}\, E$, is guaranteed by assuming
$L(0)\in \partial_{\eta,\rho}\, E$; see (\ref{trivial4}).
\end{remark}

\begin{remark}
\label{explanation3}
Note that the constants, $1\geq\kappa,\eta,\xi,\rho>0$ can be chosen arbitrarily small,
provided  $r$ is then chosen to satisfy (\ref{rho}).
The particular form of (\ref{rho}) reflects the
multiplicative structure of $\H$; compare Remark \ref{rrho}.
In view of (\ref{trivial11}), by taking $\xi$ a definite amount smaller if necessary,
we can replace $csr$ in (\ref{weakened}) by any smaller positive radius.
\end{remark}


\subsection*{Lines in quantitative boundaries of $\delta$-monotone sets.}

Given that any proper monotone subset, $E$, is a half-space, and hence,
that $\partial \, E$ is a $2$-plane, it follows that if for $L\in {\rm lines}(\H)$,
 $y_1,y_2\in L\cap \partial\, E$, $y_1\ne y_2$, then $L\subset \partial\, E$. Moreover,
$\partial\, E$ is a union of  such lines $L\in {\rm lines}(\H)$. Conversely, these statements form
two key substeps in the proof that monotone subsets $\H$ are half-spaces;
see substeps 1) and 2) of part B)
of Section \ref{cms}.

Essentially, Corollaries \ref{step1}, \ref{step2}, of Proposition \ref{pdrrvc},
constitute the corresponding substeps in the proof that $\delta$-monotone
subsets are close to half-spaces. Specifically,  if we quantify the hypotheses of the
above two statements,
then the conclusions hold in a weaker quantitative sense.
Prior  to tackling these corollaries, whose statements
are somewhat complicated,
the reader may wish
 to look at the proofs of above mentioned
 substeps given in Section \ref{cms}. Here and in Section
\ref{cms}, the skeleton of the argument
is precisely the same, but in Section \ref{cms}, the technical complications
are absent.

Below, we have sacrificed some sharpness to avoid further
complication in the relevant expressions.

\begin{corollary}\label{coro:far part of line} Let $C,c$ be the constants from Proposition~\ref{pdrrvc}.
 For every $\kappa,\alpha_1,\alpha_2\in (0,1)$,
$u_1\in \left(0,\frac{\kappa^3}{2}\right)$,
$u_2\in (0,c\kappa^2)$ define
\begin{equation}\label{eq:def gamma}
\gamma=C\cdot\max\left\{\sqrt{\frac{2u_1}{\kappa}},\frac{u_2}{c\kappa}\right\}\, ,
\end{equation}
and
\begin{equation}\label{eq:def beta}
\beta=\frac{C^5\alpha_1^2\alpha_2 u_1^3u_2^4\kappa}{c^3\gamma^5} \, ,
\end{equation}
\begin{equation}\label{eq:def delta}
\delta= \frac{C^2\alpha_1^3\alpha_2^2u_1^3u_2^8\kappa}{c^7\gamma^2} \, .
\end{equation}
Fix $L\in{\rm lines}(\H)$ which is parameterized by arc length and $E\subset B_{2C}(L(0))$ which is
$\delta$-monotone on $B_{2C}(L(0))$.
Assume that
$$
x_1 =L(0)\in \partial_{\alpha_1,u_1}(E) \, ,\qquad
x_2 =L(\kappa)\in \partial_{\alpha_2,u_2}(E)\, ,
$$
so that in particular $d^\H(x,y)=\kappa$.
Then for all $t\in [\kappa,1]$ we have
$$
L(t)\in \partial_{\beta,\gamma}(E)\, .
$$
\end{corollary}

\begin{proof}
Assume  for the sake of contradiction that $L(t)\notin \partial_{\beta,\gamma}(E)$. Thus without loss of
generality we can assume that, say, $x_1\in \mathrm{int}_{1-\alpha_1,u_1}(E)$ and
$L(t)\in \mathrm{int}_{\beta,\gamma}(E)$.
 In order to apply Proposition~\ref{pdrrvc} rescale the metric  $d^\H\to \frac{1}{t}d^\H$,
so that the rescaled distance
between $x_1$ and $L(t)$ is equal to $1$.
Let $L'(s)=L(ts)$, so that $L'$ is parameterized according to arc length in the rescaled metric.
We will apply Proposition~\ref{pdrrvc} with the following parameters:
\begin{equation}\label{eq:rescaled parameters}
\eta=\alpha_1,\quad \rho=\frac{u_1}{t},\quad r=\frac{\gamma}{Ct},\quad  s
=\frac{\kappa}{t},\quad \xi=\alpha_2\left(\frac{Ctu_2}{c\gamma\kappa}\right)^4\, .
\end{equation}
Note that, with this notation, the parameter $\beta$, as defined in~\eqref{eq:def beta}, is at most $\delta_1$,
 as given in~\eqref{delta1first} (where we used the assumption $t\ge \kappa$).
Hence in the rescaled metric, and using the notation in~\eqref{eq:rescaled parameters}, we have
$x_1=L'(0)\in \mathrm{int}_{1-\eta,\rho}(E)$, $L(t)=L'(1)\in \mathrm{int}_{\delta_1,Cr}(E)$. Moreover,
after rescaling, $E$ is $\delta$-monotone on $B_{2C/t}(x_1)$, and hence $E$ is $\frac{\delta}{t^4}$-monotone
on $B_{2C}(x_1)$. Observe that $\frac{\delta}{t^4}$, with $\delta$ given in~\eqref{eq:def delta},
is at most $\delta_2$
as defined in~\eqref{delta2} (using the assumption $t\ge \kappa$). We are therefore in position to apply
Proposition~\ref{pdrrvc}, provided that we check that with our definitions $\kappa,\eta,\xi,r,\rho\in (0,1)$,
$s\in [\kappa,1]$, and~\eqref{rho} is satisfied.
These facts are ensured by the assumptions $u_1<\frac{\kappa^3}{2}$, $u_2<C\kappa^2$, $t\in [\kappa,1]$,
and~\eqref{eq:def gamma}.
So, the conclusion of Proposition~\ref{pdrrvc}
says that in the rescaled metric we have $x_2=L'(s)\in {\rm int}_{\xi,crs}(E)$.

Rescaling back to the original metric $d^\H$, we see that $x_2\in {\rm int}_{\xi,crst}(E)
={\rm int}_{\xi,\frac{c\gamma\kappa}{Ct}}(E)$.
 In other words:
\begin{equation}\label{eq:for contradiction boundary}
\L_3(E'\cap B_{\frac{c\gamma\kappa}{Ct}}(x_2))
\le \xi\cdot \L_3(B_{\frac{c\gamma\kappa}{Ct}}(x_2))
=\xi\cdot\left(\frac{c\gamma\kappa}{Ctu_2}\right)^4
\cdot\L_3\left(B_{u_2}(x_2)\right)\stackrel{\eqref{eq:rescaled parameters}}{=}\alpha_2\cdot
\L_3\left(B_{u_2}(x_2)\right)\, .
\end{equation}
Note that $\frac{c\gamma\kappa}{Ct}\ge u_2$, and hence
$B_{\frac{c\gamma\kappa}{Ct}}(x_2)\supset B_{u_2}(x_2)$,
as ensured by~\eqref{eq:def gamma} (since $t\le 1$).
Moreover, the assumption $x_2\in \partial_{\alpha_2,u_2}(E)$ implies that
\begin{equation}\label{eq:use boundary x_2}
\L_3(E'\cap B_{\frac{c\gamma\kappa}{Ct}}(x_2))
\ge \L_3(E'\cap B_{u_2}(x_2))>\alpha_2\cdot \L_3(B_{u_2}(x_2))\, .
\end{equation}
Inequalities~\eqref{eq:for contradiction boundary},
 \eqref{eq:use boundary x_2} yield the desired contradiction.
\end{proof}

\begin{corollary}\label{step1}
Let $C,c$ be the constants from Proposition~\ref{pdrrvc}.
 For every $\kappa\in \left(0,\frac12\right]$, $\alpha_1,\alpha_2\in (0,1)$,
$u_1\in \left(0,\frac{\kappa^7}{8C^2}\right)$,
$u_2\in \left(0,\frac{c\kappa^4}{2C}\right)$ define
\begin{equation}\label{eq:def gamma*}
\gamma_*=C^{3/2}\left(\frac{2}{\kappa}\right)^{1/2}\cdot\max\left\{\left(\frac{2u_1}{\kappa}\right)^{1/4},
\left(\frac{u_2}{c\kappa}\right)^{1/2}\right\}\, ,
\end{equation}
and
\begin{equation}\label{eq:def beta*}
\beta_*=\frac{128C^{29}\alpha_1^4\alpha_2^3u_1^6u_2^{12}}{\kappa^5\gamma_*^{19}} \, ,
\end{equation}
\begin{equation}\label{eq:def delta*}
\delta_*= \frac{2^{16}C^{41}\alpha_1^6\alpha_2^5u_1^9u_2^{20}}{81c^{16}\kappa^8\gamma_*^{26}} \, .
\end{equation}
 Fix $L\in{\rm lines}(\H)$ parameterized by arc length and
$E\subset B_{3C}(L(0))$, which is $\delta_*$-monotone on $B_{3C}(L(0))$.
Assume that
$$
x_1 =L(0)\in \partial_{\alpha_1,u_1}(E) \, ,\qquad
x_2 =L(\kappa)\in \partial_{\alpha_2,u_2}(E)\, ,
$$
so that in particular $d^\H(x,y)=\kappa$.
Then for all $t\in [0,1]$ we have
\begin{equation}\label{eq:want*}
L(t)\in \partial_{\beta_*,\gamma_*}(E)\, .
\end{equation}
\end{corollary}

\begin{proof}
The case $t\in [\kappa,1]$ is a simple consequence of
Corollary~\ref{coro:far part of line}. Indeed, since $E$ is
$\delta_*$-monotone on $B_{3C}(L(0))$, it is also
$\left(\frac{3}{2}\right)^4\cdot\delta_*$-monotone on $B_{2C}(L(0))$.
Note that $\left(\frac{3}{2}\right)^4\cdot\delta_*\le \delta$, where
$\delta$ is given in~\eqref{eq:def delta} (this follows
 immediately from the definitions and our assumed upper bounds on $u_1,u_2$).
We can therefore
apply Corollary~\ref{coro:far part of line} to deduce that
$L(t)\in \partial_{\beta,\gamma}(E)$, where $\beta,\gamma$
are given in~\eqref{eq:def beta}, \eqref{eq:def gamma}.
 Note that $\gamma_*\ge \gamma$, and hence by~\eqref{trivial11} we have
\begin{equation}\label{eq:pass to star}L(t)
\in \partial_{\left(\frac{\gamma}{\gamma_*}\right)^4\cdot\beta,\gamma_*}(E)\, .
\end{equation}
Since $\beta_*\le \left(\frac{\gamma}{\gamma_*}\right)^4\cdot\beta$
we can use~\eqref{trivial1} to deduce that $L(t)\in \partial_{\beta_*,\gamma_*}(E)$, as required.

To deal with the case $t\in [0,\kappa]$, note that since, as explained above,
$E$ is $\delta$-monotone on
$B_{2C}(L(0))$, by Corollary~\ref{coro:far part of line} we have
$L(2\kappa)\in \partial_{\beta,\gamma}(E)$.
We wish to apply Corollary~\ref{coro:far part of line} to the points
$L(2\kappa),x_2=L(\kappa), L(t)$, and
to the geodesic $L'(s)=L(-s+2\kappa)$ (so that $L'(0)=L(2\kappa)$,
$L'(\kappa)=L(\kappa)$, $L'(2\kappa-t)=L(t)$,
where $2\kappa-t\in [\kappa,1]$), with $(\alpha_1,u_1)$ replaced by $(\beta,\gamma)$
and $(\alpha_2,u_2)$ unchanged. Note that it follows
from~\eqref{eq:def gamma},\eqref{eq:def gamma*} that:
\begin{equation}\label{eq:no need second term}
\gamma_*=C\sqrt{\frac{2\gamma}{\kappa}}
=C\cdot\max\left\{\sqrt{\frac{2\gamma}{\kappa}},\frac{u_2}{c\kappa}\right\}\, ,
\end{equation}
where in the last step above we used the fact that
$\gamma\ge \frac{Cu_2}{c\kappa}$, which implies
that $\sqrt{\frac{2\gamma}{\kappa}}\ge \frac{u_2}{c\kappa}$ by our assumption on $u_2$.
We can restate~\eqref{eq:no need second term}
as saying that $\gamma_*$ corresponds to the quantity $\gamma$ of
Corollary~\ref{coro:far part of line}, where we substitute our
new choice of parameters. Moreover, by substituting these values
into~\eqref{eq:def delta} and~\eqref{eq:def beta}, we see that
for these new parameters $\left(\frac{3}{2}\right)^4\cdot\delta_*$ is equal
to the
value of $\delta$ from~\eqref{eq:def delta}, and
$\beta_*$ is equal to the value of $\beta$ from~\eqref{eq:def beta*}
(we chose the values of $\beta_*,\delta_*$ precisely for this purpose).
Note that since $E$ is $\delta_*$-monotone on
$B_{3C}(L(0))\supset B_{2C}(L(2\kappa))$,
it is also $\left(\frac{3}{2}\right)^4\cdot\delta_*$-monotone on $B_{2C}(L(2\kappa))$.
In order to apply Corollary~\ref{coro:far part of line}
we also need to ensure that $\gamma
=C\cdot\max\left\{\sqrt{\frac{2u_1}{\kappa}},\frac{u_2}{c\kappa}\right\}\le \frac{\kappa^3}{2}$
and $\beta=\frac{C^5\alpha_1^2\alpha_2 u_1^3u_2^4\kappa}{c^3\gamma^5}\le 1$,
which is indeed the case, as follows from the assumed upper
bounds on $u_1,u_2$. The conclusion of Corollary~\ref{coro:far part of line} is
precisely~\eqref{eq:want*}.
\end{proof}


\begin{corollary}\label{cor:split claim}
Let $C,c$ be the constants from Proposition~\ref{pdrrvc}. Fix $\alpha,u\in (0,1)$,
$t\in\left(0,\frac13\right]$ and $d>0$, such that:
\begin{equation}\label{eq:d range}
\frac{2C}{c}u\le d\le \min\left\{2C,\frac{C^2}{2}\right\}\cdot t\, ,
\end{equation}
\begin{equation}\label{eq:mutual condition}
\alpha u^4\le \min \left\{\frac{c^3}{2C^7}d^3t, \frac{8c^4}{C^5}dt^3\right\}\, .
\end{equation}
 Denote:
\begin{equation}\label{eq:def phi}
\phi=\frac{C^{20}}{2^9c^{12}}\cdot\frac{\alpha^4 u^{16}}{d^8t^8}\, ,
\end{equation}
\begin{equation}\label{eq:def zeta}
\zeta= \frac{C^{17}}{2^{12}c^{16}}\cdot\frac{\alpha^5u^{20}}{t^{11}d^5}\, .
\end{equation}
Assume that $x\in \H$ and
$E\subset B_{2C}(x)$ are such that $x\in\partial_{\alpha,u}(E)$.
Assume also that $E$ is $\zeta$-monotone on
$B_{2C}(x)$. Then there exists $L\in \lines(\H)$ with $L(0)=x$
such that either $L(t)$ or $L(-t)$ is in
$\partial_{\phi,d}(E)$.
\end{corollary}

\begin{proof}
Take any $L\in \lines(\H)$ with $L(0)=x$. Note that~\eqref{eq:d range},
\eqref{eq:mutual condition}, \eqref{eq:def phi}
imply that $\phi\le \frac12$. Hence, we are done if one of $\{L(-t),L(t)\}$
is in $\mathrm{int}_{\phi,d}(E)$ and the other
 is in $\mathrm{int}_{\phi,d}(E')$. Indeed, in this case
since
the collection of all oriented lines through $x$ is connected (it can be
identified with the unit circle)
 the required result
follows from the Intermediate Value Theorem.

Assume for the sake of contradiction that the assertion of Corollary~\ref{cor:split claim} is false.
Then by the above
discussion $\{L(-t),L(t)\}\subset \mathrm{int}_{\phi,d}(E)$ or
$\{L(-t),L(t)\}\subset \mathrm{int}_{\phi,d}(E')$. So, assume without loss of
generality that $\{L(-t),L(t)\}\subset \mathrm{int}_{\phi,d}(E)$.
In order to apply Proposition \ref{pdrrvc} we rescale the metric:
$d^\H\to \frac{1}{2t}d^\H$. Denote $y_1=L(-t)$, $y_2=L(t)$,
so that in the rescaled metric the distance between $y_1$ and $y_2$
is $1$, and their distance from $x$ is $\frac12$.
The geodesic $L'(s)=L(2ts-t)$ is parameterized by arc length in the
 rescaled metric with $L'(0)=y_1$, $L'\left(\frac12\right)=x$, $L'(1)=y_2$.
We will apply Proposition \ref{pdrrvc} to the geodesic $L'$ with the following parameters:
\begin{equation}\label{eq:parameters for the circle}
\kappa=\eta=s=\frac12,\quad \xi=\alpha \left(\frac{2Cu}{cd}\right)^4,\quad r
=\frac{d}{2Ct},\quad \rho=\frac{C^5}{8c^3}\cdot\frac{\alpha u^{4}}{dt^3}\, .
\end{equation}
Note that~\eqref{eq:d range}, \eqref{eq:mutual condition} ensure that
$\rho\le \min\left\{\frac12 \kappa r^2,c\right\}$
and $r,\xi\le 1$. In the rescaled metric $y_1,y_2\in \mathrm{int}_{\phi,d/(2t)}(E)
=\mathrm{int}_{\phi,Cr}(E)$.
Since $\rho\le \frac{d}{2t}$ we may conclude from~\eqref{trivial11} that
$$
y_2\in \mathrm{int}_{\left(\frac{d}{2t\rho}\right)^4\cdot\phi,\rho}(E)
\stackrel{\eqref{eq:def phi}\wedge \eqref{eq:parameters for the circle}}{=}\mathrm{int}_{\eta,\rho}(E)\, .
$$

In the rescaled metric $E$ is $\zeta$-monotone on $B_{C/t}(x)$. Since $t\le \frac13$ we
have $B_{2C}(y_1)\subset B_{C/t}(x)$, and therefore $E$ is $\frac{\zeta}{(2t)^4}$-monotone
on $B_{2C}(y_1)$. Observe that with the parameters set as in~\eqref{eq:parameters for the circle},
we have $\delta_1=\phi$ and $\delta_2=\frac{\zeta}{(2t)^4}$, where $\delta_1,\delta_2$ are as
 in~\eqref{delta1first}, \eqref{delta2}. Therfore, we can apply
Proposition \ref{pdrrvc}, which implies that (in the rescaled metric)
\begin{equation}\label{eq:for cotradiction, x in interior}
x\in \mathrm{int}_{\xi,crs}(E)\stackrel{\eqref{eq:parameters for the circle}}{=}
\mathrm{int}_{\alpha \cdot\left(\frac{2Cu}{cd}\right)^4,\frac{cd}{4Ct}}(E)\, .
\end{equation}
But, in the rescaled metric we know that $x\in \partial_{\alpha,\frac{u}{2t}}(E)$.
Note that it follows
from~\eqref{eq:d range} that $\frac{u}{2t}\le \frac{cd}{4Ct}$. Hence,
 by~\eqref{trivial11} and~\eqref{trivial2}
we deduce from~\eqref{eq:for cotradiction, x in interior} that
$x\in \mathrm{int}_{\alpha,\frac{u}{2t}}(E) \subset \left(\partial_{\alpha,\frac{u}{2t}}(E)\right)'$,
 which yields the desired contradiction.
\end{proof}

\begin{corollary}\label{step2} There exist universal constants
$\overline {c},\overline{C}>0$ with the following properties.
Fix $\alpha\in (0,1)$ and $u\in \left(0,\overline{c}\right)$.
Let $x\in \H$ and $E\subset B_{\overline{C}}(x)$ be $\delta$-monotone on $B_{\overline{C}}(x)$, where
\begin{equation}\label{eq:def delt}
\delta=\overline{c} \cdot \alpha^{29}u^{71}\, .
\end{equation}
Assume also that $x\in \partial_{\alpha,u}(E)$. Then there exists
$L\in {\rm lines}(\H)$, parameterized by arc length, such that $x=L(0)$ and
for all $s\in \left[-\frac13,\frac13\right]$ we have:
\begin{equation}
\label{estep2}
L(s)\in\partial_{\overline{c}\cdot\alpha^{19}u^{46},\overline{C}\cdot u^{1/4}}(E)\, .
\end{equation}
\end{corollary}

\begin{proof} Below we will have $\overline{C}\ge 4C$, where $C$ is the
constant from Proposition~\ref{pdrrvc}.
We will first apply Corollary~\ref{cor:split claim} with $t=\frac13$ and
$d\asymp u$, while noting that if $d$ is a
large enough multiple of $u$, and $\overline{c}$ is small enough (i.e., $u$ is small enough),
then the conditions~\eqref{eq:d range},
\eqref{eq:mutual condition} are satisfied. For $\phi$ as in~\eqref{eq:def phi} we have
$\phi\asymp \alpha^4 u^8$, and for $\zeta$
as in~\eqref{eq:def zeta} we have $\zeta\asymp \alpha^5 u^{15}$. Note that since
$E$ is $\delta$-monotone on $B_{\overline C}(x)$
it is also $\zeta$-monotone on $B_{2C}(x)$, provided $\overline c$
is small enough. Hence, it follows from Corollary~\ref{cor:split claim}
that there exists $L\in \lines(\H)$ with
$L(0)=x$ and $L\left(-\frac13\right)\in \partial_{\phi,d}(E)$.

Next, we will apply Corollary~\ref{step1} with $x_1= L\left(-\frac13\right)$,
$x_2=x$, $\kappa=\frac13$, $\alpha_1=\phi$, $u_1=d$,
$\alpha_2=\alpha$, $u_2=u$. Provided $\overline c$ is small enough,
$u_1,u_2$ satisfy the conditions of Corollary~\ref{step1}.
 For $\gamma_*$ as in~\eqref{eq:def gamma*} we have
 $\gamma_*\asymp u^{1/4}$, for $\beta_*$ as in~\eqref{eq:def beta*} we
have $\beta_*\asymp \alpha^{19}u^{181/4}$, and for $\delta_*$ as
 in~\eqref{eq:def delta*} we have $\delta_*\asymp \alpha^{29}u^{141/2}$.
Moreover, since $E$ is $\delta$-monotone on $B_{\overline C}(x)$ it is
also $\delta_*$-monotone on $B_{3C}(x_1)$, provided
$\overline c$ is small enough. Corollary~\ref{step1} now implies that
$L(s)\in\partial_{\beta_*,\gamma_*}(E)$ for all $s\in \left[-\frac13,\frac13 \right]$, as required.
\end{proof}

\section{Classification of monotone sets}
\label{cms}

It was shown in \cite{ckmetmon} that
 a proper nonempty monotone
set $E\subset \H$ is, up to a set of measure zero, a half-space $\P$.
In this section we recall the proof of this result, which is an essential preliminary for
proof of Theorem \ref{stability},
the stability
version for $\delta$-monotone sets,
given in Section \ref{s:dm}.  To give the key ideas, it suffices
to consider
the case of precisely monotone sets.
We call a set $E$  {\it precisely monotone} if
$E\cap L$ and $E'\cap L$ are connected for all $L$ (rather than
 connected modulo subsets of measure zero).

There are two main steps:
\vskip2mm

\no A) $E\cap B_r(p)$ or $E'\cap B_r(p)$ has nonempty interior for
every $r>0$ and $p\in \H$. \vskip2mm

\no B) Either $\partial E$ is contained in a $2$-plane, in which
case (by a trivial connectedness argument) the theorem holds, or
$\partial E$ has nonempty interior.
\vskip2mm

Step A) follows from the special case of Lemma \ref{leqi} in which $\epsilon=0$; the
general
case of Lemma \ref{leqi} is A$'$),
 the quantitative version of A),
which is used in
 the proof of Theorem \ref{stability}.
Below, we prove Step B).



\subsection*{Proof of Step B)}
It will suffice assume that
$\partial E$ is nonempty.

\begin{proof}The proof of B)
has the following substeps.
\vskip2mm

\no
1) If  $L\in\lines(\H)$ and $L\cap \partial E$ contains at least two points,
then $L\subset\partial E$.
\vskip2mm

\no
2) $\partial E$ is a union of lines.
\vskip2mm

\no
3) Steps 1) and 2) imply B).
\vskip2mm

Steps 1), 2), are
special cases of Corollaries \ref{step1}, \ref{step2}, respectively.
 But for clarity, we will repeat the proofs in the present simpler
(nonquantitative) situation.
The proof of step 3)
requires some additional properties of pairs of lines, which are established in the
next subsection.


\subsection*{Properties of pairs of lines.}
 $L_1,L_2\in\lines(\H)$ are called {\it parallel} if
$L_2=g\cdot L_1$, or equivalently, $L_1=g^{-1}\cdot L_2$, for some
$g\in \H$. This holds if and only if the projections
$\pi(L_1),\pi(L_2)$ are either parallel or coincide. Unless
$\pi(L_1)=\pi(L_2)$, the lines
 $L_1,L_2$
are {\it not} parallel as lines in $\R^3$.

If $L_1,L_2$ are not parallel, then their projections intersect in a
unique point. If in addition, $L_1\cap L_2=\emptyset$, then the pair
$L_1,L_2$ is called {\it skew}.  We put
\begin{equation}
\label{F} F(L_1,L_2)=\pi^{-1}(\pi(L_1)\cap \pi(L_2))\, .
\end{equation}
Define the vertical distance between $L_1,L_2$ by
\begin{equation}
\label{F1}
d^V(L_1,L_2)=d^\H(L_1\cap F(L_1,L_2),L_2\cap F(L_1,L_2))\, .
\end{equation}

The following lemmas describe the family of lines
$L$ which intersect both $L_1$ and $L_2$, where
 $L_1,L_2,$ are either parallel lines with distinct
projections, or skew. At this point, the roles of $L_1,L_2$
are symmetric and assertions about $L_1$ should  be
understood as applying to $L_2$ as well.

\begin{lemma}
\label{lemgeometry1}
If $L_1,L_2\in\lines(\H)$ are parallel, and $\pi(L_1)\neq \pi(L_2)$,
then:
\vskip2mm

\noindent 1) For $x\in L_1$, there is a unique $x^*\in L_2$, such
that $x,x^*$ lie on a line $L(x)\in\lines(\H)$.  Moreover, there is
a unique shortest line segment $b(L_1,L_2)$ from $L_1$ to $L_2$
which is orthogonal to both $L_1$ and $L_2$. \vskip2mm

\noindent 2) There exists a unique point $m=m(L_1,L_2)\in\R^2$ lying
halfway between $\pi(L_1)$ and $\pi(L_2)$, such that for all $x\in
L_1$, the line  $L(x)$ intersects the fibre $\pi^{-1}(m)$. Moreover,
every point on $\pi^{-1}(m)$ intersects some line $L(x)$. \vskip2mm

\noindent 3) The union of the lines $\{L(x)\}_{x\in L_1}$ is a
smooth ruled surface $X=X(L_1,L_2)$ defined over $\R^2\setminus T$,
where $T=T(L_1,L_2)$ denotes the line parallel to
$\pi(L_1),\pi(L_2)$, which passes through $m(L_1,L_2)$.  At no point
$x\in X$ is $X$ tangent to the horizontal subspace $H_x$.
\end{lemma}
\begin{proof}
Let $O_\theta$ be as in the subsection of Section \ref{3steps}
entitled ``the Heisenberg group as a Lie group''. By applying
$O_\theta$ for suitable $\theta$, and then applying a suitable left
translation we can assume that $L_1=(0,0,0)\cdot(t,0,0)$. Next,
after a suitable translation $(a'',0,0)$ (which leaves $L_1$
invariant) we can assume that assume that $L_2$ intersects the plane
$(u,v,0)$ along the $v$-axis, and hence can be written as
$(0,b'',0)\cdot(s,0,0)$ (we are using here the fact that
$\pi(L_1)\neq\pi(L_2)$). Finally, after a translation of the form
$(0,-b''/2,0)$ we can assume that
\begin{equation}
\label{canonparallel} L_1=(t,e,-et)\qquad L_2=(s,-e,es)\,
\end{equation}
for some $e\in \R$.

 For this pair,
 if $x=(t,e,-et)$, then
\begin{equation}
\label{q1}
x^*=(-t,-e,-et)\, .
\end{equation}
To see this, observe that the line $L(x)$ in $\R^3$ which joins
these points, is parallel to the plane $(u,v,0)$ and passes through
the point $(0,0,-et)$. So $L(x)\in \lines(\H)$; see
(\ref{horizontal1}). The fiber $\pi^{-1}((0,0))$ is the unique fiber
through which all of these lines pass as $t$ varies in
$(-\infty,\infty)$, and every point on $\pi^{-1}((0,0))$ arises in
this way. Thus, for this pair,
\begin{equation}
\label{m}
m(L_1,L_2)=(0,0)\,  .
\end{equation}
The unique shortest line segment joining $L_1,L_2$ has length $2e$
and is contained in the line
\begin{equation}
\label{b}
b(L_1,L_2)\defeq (0,s,0)\, .
\end{equation}
The surface, $X$, is given by
\begin{equation}
\label{X}
(u,v,-e^2u/v)\, .
\end{equation}
 From these facts, all remaining statements are straightforward to verify.
\end{proof}

\begin{lemma}
\label{lemgeometry2} If  $L_1,L_2\in \lines(\H)$ are skew, then
there is a hyperbola $Y=Y(L_1,L_2)\subset \R^2$ with asymptotes
$\pi(L_1)$, $\pi(L_2)$, such that every tangent line of $Y$ has a
unique lift, $L\in\lines(\H)$, which intersects both $L_1$ and
$L_2$. Conversely, if $L\in \lines(\H)$ and $L\cap
L_i\neq\emptyset$, for $i=1,2$, then $\pi(L)$ is tangent to $Y$.
Except for the unique point
$L_1\cap F(L_1,L_2)$ (respectively $L_2\cap F(L_1,L_2)$)
every point $x$ of $L_1$ lies on a unique line,
$\underline{L}(x)$, intersecting
$L_2$ (respectively
every point $x$ of $L_2$ lies on a unique line,
$\underline{L}(x)$, intersecting
$L_1$).  
\end{lemma}
\proof By  applying a suitable left translation and then an isometry
$O_\theta$ and after possibly interchanging the subscripts of
$L_1,L_2$, we can assume the canonical normalization:
\begin{equation}
\label{canonskew}
L_1=(t,wt,c), \qquad L_2=(s,-ws,-c),\quad w,c>0,\ \pi(L_1)\cap\pi(L_2)=(0,0).
\end{equation}
 From the geometric interpretation of the
correction term given after (\ref{groupstruct}), it follows that the
points $p=(t,wt,c)\in L_1$,  and $q=(s,-ws,-c)\in L_2$ lie on some
$L\in{\rm lines}(\H)$, precisely when there is parallelogram in $\R^2$
spanned $(s,-ws)$, $(t,wt)$,  with oriented area $2c$.
Define $\theta(L_1,L_2)\in(0,\pi)$ by
\begin{equation}
\label{angledef}
\tan(\frac{1}{2}\theta(L_1,L_2))= w\, ,
\end{equation}
or equivalently,
\begin{equation}
\label{hyp1}
\tan(\frac{1}{2}\theta(L_1,L_2))\cdot ts=w\cdot ts=c\, .
\end{equation}
Then $ (1+w^2)\cdot(\sin{\theta(L_1,L_2)})\cdot ts=2c$.
The envelope of the resulting family of lines is
the  hyperbola, $Y$, with equation:
\begin{equation}
\label{hyp}
w^2 u^2-v^2=cw.
\end{equation}
\qed

\begin{corollary}
\label{geom3} Let $L_1,L_2$ be skew lines as in the proof of Lemma
\ref{lemgeometry2}. Put $x=(t,wt,c)$, $\widehat{x}=(-t,-wt,c)$
(where $x,\widehat{x}\in L_1$). Then as elements of $\lines(\H)$,
the lines $\underline{L}(x)$, $\underline{L}(\widehat{x})$, are
parallel, and
$$
m(\underline{L}(x),\underline{L}(\widehat{x}))=(0,0)\, .
$$
\end{corollary}

There are precisely two {\it shortest} horizontal segments joining
$L_1$ and $L_2$.  In the
notation of the proof of Lemma~\ref{lemgeometry2}, these are the
segments which project to the segments that touch
$\pi(L_1),\pi(L_2)$ and are tangent to the hyperbola $Y$ at its
focal point $\pm\sqrt{c/w}$. Thus these segments are contained in the lines:
\begin{equation}
\label{bar}
\begin{aligned}
b_+(L_1,L_2)&\defeq\left(\sqrt{\frac{c}{w}},s,\sqrt{\frac{c}{w}}\cdot s\,\right)\\
 b_-(L_1,L_2)&\defeq \left(-\sqrt{\frac{c}{w}},s,-\sqrt{\frac{c}{w} }\cdot s\right)\, .
\end{aligned}
\end{equation}



\begin{remark}
\label{limitingcase} A pair of parallel lines, $L_1,L_2$, with
distinct projections, can be viewed as a limiting case of a pair of
skew lines $L_{1,t}$, $L_{2,t}$ in which
$\lim_{t\to\infty}\theta(L_{1,t},L_{2,t})=0$ (where
$\theta(L_{1,t},L_{2,t})$ is the angle between $\pi(L_{1,t})$ and
$\pi(L_{2,t})$), and as $t\to \infty$, one  of
 $b_\pm(L_{1,t},L_{2,t})$ moves off to infinity while the other converges to $b(L_1,L_2)$,
 of the pair of parallel lines $L_1,L_2$;
compare (\ref{b}).
\end{remark}

\no {\bf Proof of 1)}. Let $p\in \partial E$. Then the exists
$p_i\in E$, with $p_i\to p$. Let $q\in {\rm int}(E)$,
the interior of $E$,
 and assume that there is a line,
 $L\in{\rm lines}(\H)$, with  $L(0)=p$ to $L(\ell)=q$.  If $L_i$ denotes
the line parallel to $L$ passing through $p_i$ then $L_i(\ell)\to L(\ell)$
and hence $L_i(\ell)\in {\rm int}(E)$ for $i$ sufficiently large.
Since $E$ is convex,   by (the non-quantitative version
of) Corollary \ref{step1},
 we get $L_i(s)\in {\rm int}(E)$, for $0<s\leq\ell$ and $i$ sufficiently large
and by passing to the limit
$L(s)\in {\rm int}(E)$, for $0<s\leq\ell$.  Since,
$E'$ is also convex, the corresponding statement holds for $E'$
as well. This completes the proof of sub-step 1).



\no {\bf Proof of 2).} Let $p\in \partial E$, $L\in{\rm lines}(\H)$
with $p=L(0)$ and fix $q=L(\ell)\neq p$. We need to show that $q\in
\partial E$.
By 1), we can assume say $L(\ell)\in{\rm int}(E)$. Similarly,
$L(-\ell)\in{\rm int}(E)$ or $L(-\ell)\in{\rm int}(E')$. By the same
argument as in in the proof of 1), if $L(\ell)\in{\rm int}(E)$ and
$L(-\ell)\in{\rm int}(E)$, then $p=L(0)\in{\rm int}(E)$.  Therefore,
we may assume that $L(-\ell)\in{\rm int}(E')$.
Since the collection of unit tangent vectors to the collection of
horizontal geodesics at $p$ is a circle, and in particular,
connected, it follows by continuity that there exists a horizontal
geodesic, $\underline{L}$, with $\underline{L}(0)=p$ and
$\underline{L}(\ell)\in \partial E$.   The conclusion of 2) now
follows from sub-step 1). 

\no {\bf Proof of 3).} By  2), $\partial E$ is a union of lines. If
$\partial E$ contains only parallel lines with the same projection,
then it is contained in a vertical plane.  Likewise, if $\partial E$
contains only lines passing through some fixed point, then it is
contained in the horizontal $2$-plane through that point. Thus, we
can assume that $\partial E$ contains either a pair of parallel
lines with distinct projections or a pair of skew lines.
By Lemma \ref{lemgeometry1}, if $\partial E$ contains a pair of
parallel lines $L_1,L_2$ with distinct projections, then distinct
members of the family of lines $\{L(x)\}_{x\in L_1}$ are
 skew lines which intersect $\pi^{-1}(m(L_1,L_2))$. So we can
assume that $\partial E$ contains a pair of skew lines $L_1,L_2$.

\no
{\bf Claim}: If $L_1,L_2$ are pair of skew lines, with
$L_1\cup L_2\subset \partial E$, then
$\pi^{-1}(\pi(L_1)\cap\pi(L_2))\subset \partial E$.

To see this, note that by Corollary \ref{geom3}, for every line
$L_3$  intersecting $L_1$ and $L_2$, there is a parallel line,
$L_4$, intersecting $L_1$ and $L_2$, such that
$m(L_3,L_4)=\pi(L_1)\cap\pi(L_2)$. Since by step 1), $L_3\cup
L_4\subset \partial E$, it follows from Lemma  \ref{lemgeometry1},
that $\pi^{-1}(m(L_3,L_4))\in \partial E$.

 For $Y=Y(L_1,L_2)$ as in Lemma \ref{lemgeometry2}.
Let $U$ denote the component of $\R^2\setminus Y$, which intersects
both branches of $Y$. Every point of $V=U\setminus (\pi(L_1)\cup
\pi(L_2))$ is the intersection of a pair of distinct tangent lines
of $Y$. Each such pair of lines lifts to a pair of skew
lines $L_5,L_6$ in $\partial E$. By the Claim, $\pi^{-1}(\pi(L_5)\cap\pi(L_6))
\in
\partial E$, so $\pi^{-1}(V)\subset \partial E$. Thus $\partial E$
has non-empty interior, as required.
\end{proof}

\section{Proof of the stability of monotone sets}\label{s:dm}

The proof of Theorem~\ref{stability}, i.e., the stability of monotone sets, has
two main steps, A$'$), B$'$), which are quantitative
versions of steps  A), B), in the classification of monotone
subsets of $\H$ given in Section
\ref{cms}.  In fact, A$'$)
 is just Lemma \ref{leqi}, which is proved in Section \ref{pfleqi}.
In this section, assuming assuming some technical
geometric preliminaries which are proved in Section~\ref{ndic},we carry out step B$'$),
thereby completing the proof of Theorem~\ref{stability}.

 The proof of B$'$)
has substeps 1$'$)--3$'$), which
correspond to substeps 1)--3) of B).
Substeps 1$'$), 2$'$), are just
Corollaries \ref{step1}, \ref{step2}, specialized to our specific context.
Substep 3$'$) requires a substantial preliminary discussion; see Section~\ref{ndic}, and
in particular, Lemma \ref{nondegen} and Corollaries \ref{nondegen1},
\ref{nondegen2}.
The point of the discussion in Section~\ref{ndic} is to quantify
the following statement which occurs at the end of the first
paragraph of the proof of substep 3):
``Thus, we can assume that $\partial E$ contains
either a pair of parallel lines with distinct
projections or a pair of skew lines''. Once this statement has been quantified,
the proof of substep 3$'$) is completed by repeating, mutadis mutandis, the
proof of substep 3).

The preliminaries to substep 3$'$), i.e., Section~\ref{ndic}, constitute the part of the argument which necessitates
assuming that $E$ is $\epsilon^a$-monotone on $B_{\epsilon^{-3}}(x)$,
rather than on $B_1(x)$.
The following example shows
that such an
assumption cannot be entirely avoided
even if $\epsilon^a$-monotonicity is strengthened to monotonicity.

\begin{example}
\label{global}
Consider the set
\begin{equation}
\label{counterexample}
E=\{(x,y,z)\,|\, z\leq xy,\  y>0\}\, .
\end{equation}
 We claim that if $B_r(p)\subset\{(x,y,z)\, |\, y>0\}$, then
 $E\cap B_r(p)$
is  monotone.  This is equivalent to the statement
that if a line $L\in \lines(\H)$ passes through two points
 in $\partial(E)=\{(x,y,z)\, |\, z=xy,\, \,\, y>0\}$,
then $L$ is contained in $\partial E$. Note that the horizontal
line $L=(a,t,at)$  passes through $(a,b,ab)$. If $b>0$, then $L$ is
contained in $\partial E$. Conversely, if
$(a,b,ab), (a',b',a'b')\in \partial E$, then $b,b'>0$.
If these points lie on a horizontal line,  since
 $H_{(a,b,ab)}=(x,y,ab-bx+ay)$, we get
\begin{equation}
\label{ce}
a'b'=ab-ba'+ab'\, ;
\end{equation}
see (\ref{horizontal1}).  The solutions of (\ref{ce}) are: $a'=a$ ,
$b'$ arbitrary, which gives the above line $L$,
or:  $a'$ arbitrary, $b'=-b$, which
contradicts the assumption $b'>0$.
\end{example}

\begin{proof}[Proof of Theorem~\ref{stability}] Recall that in Theorem \ref{stability},
we are given $E\subset B_1(x)\,{}$ which is
$\epsilon^a$-monotone on $B_1(x)$. Our goal is to find a half-space $\mathcal P\subseteq \H$ such that
$$
\frac{\L_3\left((E\cap B_{\e^3}(x))
\triangle \P\right)}{\L_3\left(B_{\e^3}(x)\right)}\lesssim \e\, .
$$
For convenience of notation in the ensuing argument we will rescale the ball
$B_1(x)$ above by $\e^{-3}$. Thus our
assumption is that $E\subset B_{\e^{-3}}(x)$ is
$\epsilon^a$-monotone on $B_{\e^{-3}}(x)$,
and, for the sake of contradiction,
that for no half-space $\P$ we have:
$$
\L_3\left((E\cap B_{1}(x))\triangle \P\right)\lesssim \e\, .
$$
This contrapositive assumption will  be used via the following simple lemma. In what follows,
for $A\subset \H$ and $\epsilon>0$,
the closed $\epsilon$-tubular neighborhood of $A$ is denoted by:
\begin{equation}\label{eq:def tubular}
\overline{T_\epsilon(A)}
=\{x\in \H:\ d^\H(x,A)\le\epsilon\}\, .
\end{equation}
\begin{lemma}\label{lem:contrapositive}
Let $Q\subset \H$ be a $2$-plane.
Assume that for some
$\e,r\in \left(0,\frac12\right)$ a subset $E\subset \H$ satisfies:
$$
\partial_{\e,r}(E)\cap B_1(x)\subset \overline{T_\e(Q)}\, .
$$
Then there exists a half-space $\P\subset \H$ such that:
$$
\L_3\left((E\cap B_{1}(x))
\triangle \P\right)\lesssim \e\, .
$$
\end{lemma}

\begin{proof}
The set $B_1(x)\setminus \overline{T_\e(Q)}$
has at most two connected components,
say $C_1,C_2$ (one of which might be empty).
Since $\e<\frac12$, by continuity, for
$j\in \{1,2\}$ we have
$C_j\subset \mathrm{int}_{\e,r}(E)$
or $C_j\subset \mathrm{int}_{\e,r}(E')$. For
definiteness assume that
$C_1\subset \mathrm{int}_{\e,r}(E)$.
Let $\mathfrak{N}\subset C_1$ be
an $r$-net in $C_1$. Thus the balls
$\{B_{r/2}(y)\}_{y\in \mathfrak N}$ are disjoint, and the balls $\{B_{r}(y)\}_{y\in \mathfrak N}$
cover $C_1$. Moreover, since for each
$y\in \mathfrak N$ we have $y\in \mathrm{int}_{\e,r}(E)$, we know that
$\L_3(B_r(y)\cap E')\le \e \L_3(B_r(y))$. Hence
\begin{multline*}
\L_3(C_1\cap E')\le \sum_{y\in \mathcal N} \L_3(B_r(y)\cap E')
\le \sum_{y\in \mathcal N} \e\cdot \L_3(B_r(y))=16\e \cdot \sum_{y\in \mathfrak N} \L_3(B_{r/2}(y))\\
= 16\e\cdot \L_3\left(\bigcup_{y\in \mathfrak N}B_{r/2}(y)\right)
\le 16\e\cdot \L_3\left(B_{1+\frac{r}{2}}(x)\right)\lesssim \e\, .
\end{multline*}
This argument shows that for $j\in \{1,2\}$ either
$\L_3(C_j\cap E')\lesssim \e$ or
$\L_3(C_j\cap E)\lesssim \e$.
Thus we can take $\P$ to be either one of
the half-spaces bounded by $Q$, or a half space that contains $B_1(x)$.
\end{proof}

By virtue of Lemma~\ref{lem:contrapositive}, we can assume from now on that for
every $2$-plane $P\subset \H$ we have:
\begin{equation}\label{eq:contra with r power of eps}
\partial_{c_1\e,\e^{k_1}}(E)\cap B_1(x)
\not\subset \overline{T_{c_1\e}(P)}\, ,
\end{equation}
where $k_1>1$, $c_1>0$ are constants which will be determined presently.

By rescaling the metric $d^\H$ by a suitable multiple of $\e^{-3}$ we can apply
Corollary~\ref{step2} with $\alpha\asymp\e$ and $u\asymp\e^{k_1+3}$,
and (after rescaling back to our present setting) deduce that provided
\begin{equation}\label{first condition on powers}
\begin{aligned}
a&>29+71(k_1+3)=242+71k_1 \, ,\\
k_2&<\frac{k_1+3}{4}\, ,\\
h_2&>19+46(k_1+3)=157+46k_1\, ,
\end{aligned}
\end{equation}
and $\e$ is smaller than a small enough universal constant, we can associate to every
point $y\in  \partial_{\e,\e^{k_1}}(E)\cap B_1(x)$ an open set $\mathcal O_y$
of lines through the origin of $H_y$, such that for all $L\in \mathcal O_y$ we have:
\begin{equation}\label{eq:inclusion in boundary}
L\left(\left[-c_2\e^{-3},c_2\e^{-3}\right]\right)
\subset \partial_{\e^{h_2},\e^{k_2}}(E)\, ,
\end{equation}
for some small enough constant $c_2>0$. We have used here the fact that the
definition of the quantitative
boundary in~\eqref{equanbddy} is an open condition.

As we shall see in Section~\ref{ndic} (see Lemma~\ref{ddineq} and
Corollaries~\ref{nondegen1},
\ref{nondegen3}, \ref{nondegen2}) the above discussion implies the following:
If $c_1$ in
assumption~\eqref{eq:contra with r power of eps} is a small enough universal constant,
then
there exists a pair of
skew lines $L,L'\in \lines (\H)$, with the following properties:

\no
$\bullet$ The distance of $L$ and $L'$
 from $x$ is at most $\frac{c_2}{10}\e^{-3}$.

\no
$\bullet$  Both $L$ and $L'$ intersect $\partial_{\e^{h_2},\e^{k_2}}(E)\cap
B_{c_2\e^{-3}/10}(x)$.

\no
$\bullet$ $\theta(L,L')$, the angle between
$L$ and $L'$,  is bounded away from $0$ and $\pi$, say,
$\tan \theta(L,L')\in$
${}\,\,\,\,\left[\frac25,\frac52\right]$; for the definition of $\theta(L,L')$  see \eqref{angledef}.

 \no
$\bullet$  $L,L'$ have separation $\gtrsim \e^3$, i.e.
$d^V(L,L')\gtrsim \e^3$, where the vertical distance $d^V(\,\cdot\, ,\,\cdot\,)$ is as in (\ref{F1}).

We now use the lines $L,L'$ to perform substep 3$'$), that is, to produce a quantitative version
 of the argument in substep $3)$ of the
classification of precisely monotone sets.

 First of all, we will apply again Corollary~\ref{step2} (rescaled, as before, by $\asymp \e^{-3}$)
to deduce from the fact that $L$ and $L'$ intersect $\partial_{\e^{h_2},\e^{k_2}}(E)\cap B_{c_2\e^{-3}/10}(x)$ that:
\begin{equation}\label{eq:in new boundary}
\left(L\cup L'\right) \cap B_{c_3\e^{-3}}(x)\subset \partial_{\e^{h_3},\e^{k_3}}(E)\, ,
\end{equation}
where $c_3>0$ is a universal constant,
provided that:
\begin{equation}\label{second condition on powers}
\begin{aligned}
a&>29h_2+71(k_2+3)=29 h_2+71k_2+213\, , \\
k_3&<\frac{k_2+3}{4}\, ,\\
h_3&>19h_2+46(k_2+3)=19h_2+46k_2+138\, .
\end{aligned}
\end{equation}

Now, as in the non-quantitative proof, let
$Y=Y(L,L')\subset \R^2$ be the hyperbola from
 Lemma~\ref{lemgeometry2}, and let
$U$ be the component of
$\R^2\setminus Y$ which intersects both branches of $Y$. Let $c_4$
be a small enough constant,
and take a point $q\in U$ such that
$B_{c_4\e^3}(q)\cap \R^2\subset U$ and
$d^\H(q,\pi(L)\cap\pi(L'))\le 10\e^3$.
As in the proof of
substep $3)$ of the classification of monotone sets,
using Lemma~\ref{lemgeometry2},
any point $z\in B_{c_4\e^3}(q)\cap \R^2$ is of the form
$z= \pi(L^*)\cap \pi(L^{**})$, where, provided $c_4$ is small enough,
 $L^*,L^{**}$ are skew lines which intersect
 $L\cap B_{c_3\e^{-3}}(x)$
and  $L'\cap B_{c_3\e^{-3}}(x)$, both of which
are contained in $\partial_{\e^{h_3},\e^{k_3}}(E)$.
Another application of Corollary~\ref{step2} implies that:
\begin{equation}\label{eq:stars in new boundary}
\left(L^*\cup L^{**}\right) \cap B_{c_5\e^{-3}}(x)
\subset \partial_{\e^{h_4},\e^{k_4}}(E)\, ,
\end{equation}
provided that:
\begin{equation}\label{third condition on powers}
\begin{aligned}
a&>29h_3+71(k_3+3)=29 h_3+71k_3+213 \, ,\\
k_4&<\frac{k_3+3}{4}\, ,\\
h_4&>19h_3+46(k_3+3)=19h_3+46k_3+138\, .
\end{aligned}
\end{equation}

We continue to argue as is substep $3)$.
We use Corollary~\ref{geom3}
to deduce that $z=m\left(\underline{L}^*,\underline{L}^{**}\right)$,
where (once more, provided $c_4$
is small enough), $\underline{L}^*, \underline{L}^{**}$
are parallel lines which intersect the line segments
$L^*\cap B_{c_5\e^{-3}}(x)$ and
$L^{**}\cap B_{c_5\e^{-3}}(x)$, both of which are
contained in $\partial_{\e^{h_4},\e^{k_4}}(E)$.
So, another application of Corollary~\ref{step2} gives:
\begin{equation}\label{eq:underline stars in new boundary}
\left(\underline{L}^*\cup \underline{L}^{**}\right)
 \cap B_{c_6\e^{-3}}(x)\subset \partial_{\e^{h_5},\e^{k_5}}(E)\, ,
\end{equation}
provided that:
\begin{equation}\label{fourth condition on powers}
\begin{aligned}
a&>29h_4+71(k_4+3)=29 h_4+71k_4+213\, , \\
k_5&<\frac{k_4+3}{4}\, , \\
h_5&>19h_4+46(k_4+3)=19h_4+46k_4+138\, .
\end{aligned}
\end{equation}

By Lemma~\ref{lemgeometry1}, any point on the fiber $\pi^{-1}(z)$
lies on a line which touches both of $\underline{L}^*, \underline{L}^{**}$.
Moreover (using the explicit formula for this line, which is contained in
the proof of Lemma~\ref{lemgeometry1}), since the vertical separation between
$L$ and $L'$ is $\gtrsim \e^3$, and $z$ is an arbitrary point in
$B_{c_4\e^3}(q)\cap \R^2$, a (final) application of Corollary~\ref{step2}, together
with the ball-box principle, shows that there exists $p\in \H$ such that:
\begin{equation}\label{eq:found interior}
B_{r}(p)\subset \partial_{\e^{h_6},\e^{k_6}}(E)\cap B_{\e^{-3}}(x),
\end{equation}
where $r\asymp \e^3$, provided that:
\begin{equation}\label{fifth condition on powers}
\begin{aligned}
a&>29h_5+71(k_5+3)=29 h_5+71k_5+213\, , \\
k_6&<\frac{k_5+3}{4}\, ,\\
h_6&>19h_5+46(k_5+3)=19h_5+46k_5+138\, .
\end{aligned}
\end{equation}

Now, by choosing $k_1$
sufficiently large, we will use~\eqref{eq:found interior} to contradict Lemma~\ref{leqi}.
provided $k_1$ is large enough. Modulo the proofs of technical lemmas which
were postponed to the following sections, this will
 conclude the proof of
Theorem~\ref{stability} Before doing so,
note that for the conditions~\eqref{first condition on powers},
\eqref{second condition on powers}, \eqref{third condition on powers},
\eqref{fourth condition on powers}, \eqref{fifth condition on powers}
to be satisfied we can ensure that~\eqref{eq:found interior} is satisfied for, say,
$a=2^{40}k_1$, $k_6=2^{-10}k_1$, $h_6=2^{30}k_1$. (In actuality, these are big over-estimates.)

It remains to show how to choose $k_1$ so that~\eqref{eq:found interior}
will contradict Lemma~\ref{leqi}. Since $E$ is
$\e^a$-monotone on $B_{\e^{-3}}(x)$, it is also $\asymp\e^{a-24}$-monotone
on $B_r(p)$ (recall that $r\asymp \e^3$).
By Lemma~\ref{leqi} we can find $p'\in B_{r/2}(p)$ such that
$p'\notin \partial_{\alpha,r'}(E)$, where $\alpha\asymp \e^{(a-24)/2}$
and $r'\asymp r\asymp \e^3$. Without loss of generality assume that
$p'\in \mathrm{int}_{\alpha,r'}(E)$, so that
\begin{equation}\label{eq:use interior}
\L_3\left(E'\cap B_{r'}(p')\right)\le \alpha \L_3\left(B_{r'}(p')\right)\asymp \e^{\frac{a}{2}-12}\e^{12}=\e^{a/2}\, .
\end{equation}
But, at the same time, $p'\in \partial_{\e^{h_6},\e^{k_6}}(E)$, which means that
\begin{equation}\label{eq:use boundary}
\L_3\left(E'\cap B_{\e^{k_6}}(p')\right)>\e^{h_6}\L_3\left(B_{\e^{k_6}}(p')\right)\asymp \e^{h_6+4k_6}\, .
\end{equation}

In order for~\eqref{eq:use boundary} to contradict~\eqref{eq:use interior} assume that $\e^{k_6}>r'$,
which would hold
for small enough $\e$ if $k_6=\frac{k_1}{2^{10}}>3$. In this case
$B_{\e^{k_6}}(p')\subset B_{r'}(p')$, and the desired
contradiction would follow (for small enough $\e$) if $\frac{a}{2}>h_6+4k_6$.
Choosing $k_1=2^{12}$, and thus $a=2^{52}$
and $h_6=2^{42}$ yields the required contradiction, and completes the
proof of Theorem~\ref{stability}. \end{proof}

\section{Nondegeneracy of the initial configuration.}
\label{ndic}

In this section, we prove the assertion on nondegeneracy of the initial
configuration, which was used in the proof of substep $3'$) given in Section \ref{s:dm};
see the four items marked with bullets
in paragraph following (\ref{eq:inclusion in boundary}).
Specifically
under the assumptions of Lemma \ref{nondegen} below,
we show that at distance $\lesssim\epsilon^{-3}$
  from $x$
  we can find a  pair of skew lines in the controlled quantitative boundary
which make standard angle as in (\ref{final2}), and have separation
$\gtrsim \epsilon^{3}$; see (\ref{nd25}), (\ref{final3}).

Given a pair of (unordered) skew lines $L,\widetilde L$, define
$0<\theta(L,\widetilde L)<\pi$ to be the angle between $\pi(L)$ and $\pi(\widetilde L)$
which faces the sector which contains the hyperbola $Y(L,\widetilde L)$ from Lemma~\ref{lemgeometry2}.
The vertical distance
$d^V(L,\widetilde L)$ is defined as in (\ref{F1}). As in (\ref{bar}), denote by
$b(L,\widetilde L)$ the union
of the (two) shortest horizontal segments joining
$L$ and $\widetilde L$,
and denote the length of each of these segments by
$\frak d(L,\widetilde L)$.
Then with the
geometric interpretation of the multiplication in $\H$, we get
\begin{equation}
\label{ddineq}
\frac12\cot(\frac{1}{2}\theta(L,\widetilde L))\cdot(\frak d(L,\widetilde L))^2
=(d^V(L,\widetilde L))^2\, .
\end{equation}
Since $\tan\left(\frac{\theta}{2}\right)\ge \frac{\sin\theta}{2} $ for all $\theta\in [0,\pi]$,
it follows from~\eqref{ddineq} that
\begin{equation}\label{numerical}
(\frak d(L,\widetilde L))^2\gtrsim \sin\theta(L,\widetilde L)\cdot
 (d^V(L,\widetilde L))^2.
\end{equation}

In what follows, for $A\subset \H$ and $\epsilon>0$, the closed $\epsilon$-tubular neighborhood of
$A$ is denoted $\overline{T_\epsilon(A)}=\{x\in \H:\ d^\H(x,A)\le\epsilon\}$.
The required nondegeneracy of the initial configuration, namely the existence of the
``quantitatively skew" lines $L,L'$ that are described in the paragraph
preceding~\eqref{eq:in new boundary}, is a consequence of the
following Lemma \ref{nondegen}. The hypothesis (\ref{nd}) below
corresponds to the assumption concerning  $2$-planes that we arrived at
in~\eqref{eq:contra with r power of eps}, and the discussion preceding~\eqref{eq:inclusion in boundary}.
Recall that $\overline{T_\epsilon(\cdot)}$ denotes the closed $\e$-tubular neighborhood, as defined in~\eqref{eq:def tubular}.
 \begin{lemma}
\label{nondegen}
 Fix $x\in \H$, $U\subset B_1(x)$ and $\epsilon\in (0,1)$. For all $y\in U$, fix an open set $\mathcal O_y$
of lines through the origin of $H_y$. Assume that for all $2$-planes $P$ we have:
\begin{equation}
\label{nd}
U\not\subset \overline{T_\epsilon(P)}\, .
\end{equation}
Then there exist $y,z\in U$ and $L_y\in \mathcal O_y$, $L_z\in \mathcal O_z$,
such that $L_y,L_z$ are skew and:
\begin{equation}
\label{nd25}
(\mathfrak d(L_y,L_z))^2\gtrsim\sin\left(\theta(L_y,L_z)\right)\cdot(d^V(L_y,L_z))^2 \gtrsim\epsilon^{6}\, .
\end{equation}yes
\end{lemma}
\begin{proof} Denote
the Euclidean distance in $\R^2$ by $d(\,\cdot\, ,\, \cdot\, )$.
Fix $y_1\in U$.   Let $V$ denote the unique vertical $2$-plane containing $L_{y_1}$.
It follows from (\ref{nd}), with $P=V$,  that there exists $y_2\in U$ such that
\begin{equation}
\label{nd3}
d(\pi(L_{y_1}),\pi(y_2))\geq \epsilon\, .
\end{equation}

 Let $H=H_q$ denote the unique
horizontal $2$-plane containing the line, $L_{y_1}$,
and the point $y_2\not\in L_{y_1}$. Note that $q\in L_{y_1}$.
By (\ref{nd}), there exists $y_3$, such that
\begin{equation}
\label{nd9}
d^\H(y_3,H_q)\geq\epsilon\, .
\end{equation}
For arbitrary $q,y_3\in \H$, we  have
\begin{equation}
\label{symm}
d^\H(q,H_{y_3})=d^\H(y_3,H_q)\, .
\end{equation}
To see~\eqref{symm}, note that by left-invariance of $d^\H$,  the left-hand side
equals $d^\H(y_3^{-1}\cdot q,H_{(0,0,0)})$, and the right-hand side equals
$d^\H(q^{-1}\cdot y_3,H_{(0,0,0)})$.
If $z\in H_{(0,0,0)}$ is closest
to $y_3^{-1}\cdot q$, then  $z^{-1}\in H_{(0,0,0)}$, and
by (\ref{inverses}), $d^\H((y_3)^{-1}\cdot q,z)=
d^\H(q^{-1}\cdot y_3,z^{-1})$. By symmetry, this gives (\ref{nd9}).

  Since $d^\H(q,L_{y_3})\geq d^\H(q,H_{y_3})$,
from (\ref{nd9}), (\ref{symm}), we get
\begin{equation}
\label{nd90}
d^\H(q,L_{y_3}) \geq \epsilon\, .
\end{equation}


Since $\mathcal O_{y_1},\mathcal O_{y_2}, \mathcal O_{y_3}$ are open, we can assume that no
two of  $L_{y_1},L_{y_2},L_{y_3}$ are   parallel. By applying a translation we may also assume that $q=(0,0,0)$.
Thus, after applying a suitable rotation we may assume that $L_{y_1}$ is the $x$-axis and $y_1=(a,0,0), y_2=(b,c,0)$
for some $a,b,c\in \mathbb R$.
Set $L_{y_1}\cap \pi(L_{y_2})=n_3=(s,0,0)$ and
let $\alpha\in (0,\pi)$ denote the angle between $L_{y_1}$ and $\pi(L_{y_2})$.

 It follows from~\eqref{nd3} that $|c|\ge \epsilon$. Since $|a-b|+|c|+\sqrt{|ac|}\asymp d^\H(y_1,y_2)\leq 2$,  we have:
\begin{equation}\label{nd4}
|a|\lesssim \frac{1}{\epsilon}\quad \mathrm{and}\quad |a-b|\lesssim 1.
\end{equation}
Note that
\begin{equation}
\label{nd5}
\sin\alpha= \frac{c}{d(\pi(y_2),n_3)}\asymp \frac{c}{c+|b-s|}\ge\frac{\epsilon}{\epsilon +|b|+|s|},
\end{equation}
and by the geometric interpretation of the multiplication,
 \begin{equation}
\label{nd6}
\left(d^V\left(L_{y_1},L_{y_2}\right)\right)^2=|c|\cdot|s|\ge\epsilon|s|.
\end{equation}
It follows by multiplying together (\ref{nd5}) and (\ref{nd6})
that  we may assume that $\epsilon+|b|\ge |s|$, since otherwise (\ref{nd25}) holds with $\epsilon^6$ replaced by $\epsilon ^2$.
Using (\ref{nd4}) we deduce that $|b|,|s|\lesssim \frac{1}{\epsilon}$, and therefore it follows from (\ref{nd5}) that
\begin{equation}
\label{nd7} \sin\alpha\gtrsim\epsilon^2\, .
\end{equation}
Using~\eqref{nd6} and~\eqref{nd7} that we may assume that
\begin{equation}
\label{nd61}
d(n_3,q)=|s|\lesssim \epsilon^{3}\, ,
\end{equation}
since otherwise (\ref{nd25}) holds.


Let $S$ denote the triangle with vertices $\pi(L_{y_2})\cap\pi(L_{y_3})=n_1$,
$L_{y_1}\cap\pi(L_{y_3})=n_2$ and $n_3$. Let $\alpha_j$ denote the angle at the
vertex $n_j$ (thus $\alpha_3$ is either $\alpha$ or $\pi-\alpha$). Put $\ell_j=d(n_{j+1},n_{j+2})$.
(Here and in the rest of the proof of Lemma~\ref{nondegen}, indices are taken mod 3.)
It follows from (\ref{nd7}) that there exists $j\in \{1,2\}$, such that
\begin{equation}
\label{nd91}
\sin\alpha_{j}\geq \frac{c}{2}\epsilon^2\, .
\end{equation}
Write $\{k\}=\{1,2\}\setminus\{j\}$.
 From (\ref{nd7}), (\ref{nd91}), we may assume that
\begin{equation}
\label{nd92} \left(d^V(L_{y_1},L_{y_2})\right)^2\lesssim
\epsilon^{4} \, ,
\end{equation}
\begin{equation}
\label{nd93} \left(d^V(L_{y_k},L_{y_3})\right)^2\lesssim
\epsilon^{4}\, ,
\end{equation}
since otherwise (\ref{nd25}) holds.

  From the geometric interpretation of the multiplication, it follows
that there exists $r\in \{1,2,3\}$ such that
\begin{equation}
\label{nd10}
\left(d^V(L_{r+1}, L_{r+2})\right)^2
\gtrsim {\rm area}(S)=\frac{1}{2}\ell_1\ell_2\sin\alpha_3\, .
\end{equation}
By multiplying both sides of (\ref{nd10})
by $\sin\alpha_{r}$
and using the law of sines to express  $\sin\alpha_{r}$
 in terms of $\sin\alpha_3$, we obtain
with (\ref{nd7}) and the triangle inequality,
\begin{equation}
\label{nd11}
 \left(d^V(L_{r+1}, L_{r+2})\right)^2 \sin\alpha_{r}
\gtrsim  (\sin\alpha_3)^2\cdot\min\left\{\ell_{1}^2,\ell_{2}^2\right\}\\
\gtrsim
\epsilon^4\cdot\min\left\{\ell_{1}^2,\ell_{2}^2\right\}\, .
\end{equation}
Arguing similarly with (\ref{nd91}), we have:
\begin{equation}
\label{nd111}
 \left(d^V(L_{r+1}, L_{r+2})\right)^2 \sin\alpha_{r}
\gtrsim
\epsilon^4\cdot\min\left\{\ell_{k}^2,\ell_{3}^2\right\}\, .
\end{equation}
We may therefore assume that
$\min\left\{\ell_{1}^2,\ell_{2}^2\right\}\le \frac{\epsilon^2}{16}$
and $\min\left\{\ell_{k}^2,\ell_{3}^2\right\}\le
\frac{\epsilon^2}{16}$, since otherwise (\ref{nd25}) holds. By the
triangle inequality $\ell_k\le \ell_j+\ell_3$, and therefore
\begin{equation}
\label{nd112} \ell_{k}\leq \frac{1}{2}\epsilon\, .
\end{equation}

 By adding (\ref{nd61}), (\ref{nd92}), (\ref{nd93}), (\ref{nd112}), we get
$d^\H(L_{y_3},q)\leq \frac{1}{2}\epsilon+O(\epsilon^2)$, which
contradicts (\ref{nd90}) for $\epsilon$ small enough.  This
contradiction completes the proof.
\end{proof}

In deriving the consequences of Lemma \ref{nondegen}, we will work in coordinates.
Thus, we write $L_1,L_2$ for $ L_y,L_z$. As in (\ref{canonskew}), we
can canonically assume that:
$$
L_1=(t,wt,c), \qquad L_2=(s,-ws,-c)\, .
$$
As in (\ref{angledef}), the angle
$0<\theta(L_1, L_2)<\pi$ is determined by $\tan\left(\frac{1}{2}\theta(L_1,L_2)\right)=w$.
Let $b_+(L,\widetilde L)$, $ b_-(L,\widetilde L)$ contain the shortest
line joining $L_1,L_2$; see (\ref{bar}). Thus,
$$
b_\pm(L_1,L_2)= \left(\pm\sqrt{\frac{c}{w}},u,\pm\sqrt{\frac{c}{w}}u\right)\, ,
$$
and so,
\begin{equation}
\label{dtb}
d^\H(b_+(L_1,L_2),b_-(L_1,L_2))=2\sqrt{\frac{c}{w}}\, .
\end{equation}
 Put:
\begin{equation}
\label{endpoints}
e_{i,\pm}(L_1,L_2)=L_i\cap b_\pm(L_1,L_2)\, ,
\end{equation}
or equivalently,
\begin{equation}
\label{endpointscoor}
\begin{aligned}
e_{1,+}(L_1,L_2)&=\left(\sqrt{\frac{c}{w}},\sqrt{cw},c\right)\, ,\\
e_{2,+}(L_1,L_2)&=\left(\sqrt{\frac{c}{w}},-\sqrt{cw},-c\right)\, .
\end{aligned}
\end{equation}
Then with the
geometric interpretation of the multiplication in $\H$,
$\left(d^V(L_1,L_2)\right)^2$ is twice the area of the triangle with vertices
$\pi(L_1)\cap\pi(L_2)$, $\pi(e_{1,+}(L_1, L_2))$,
$\pi(e_{2,+}(L_1,L_2))$. It follows that $(d^V(L_1,L_2))^2=2c$. Note that
$$
\mathfrak d(L_1,L_2)=d^\H((e_{1,\pm}(L_1,L_2),e_{2,\pm}(L_1,L_2))\, .
$$
Therefore,
\begin{equation}
\label{mfd}
\mathfrak d(L_1,L_2)=2\sqrt{cw}\, .
\end{equation}
 Equation~\eqref{ddineq} now becomes:
\begin{equation}
\label{ddineq-new}
(\mathfrak d(L_1, L_2))^2
=2w\cdot\left(d^V(L_1,L_2)\right)^2\, .
\end{equation}

Corollary \ref{nondegen1} below will be derived directly from  Lemma \ref{nondegen}
without further reference to the assumption on $2$-planes.
The proof will rely
 on the following additional information. Write $o_1=y$ and $o_2=z$.
Since $o_1,o_2\in B_1(x)$, we have $d^\H(o_1,o_2)< 2$, which implies
$d^V(o_1,o_2)<2$,  $d(\pi(o_1),\pi(o_2))<2$.

We put
$$
\begin{aligned}
o_1&=(t_1,wt_1,c)\, , \\
 o_2&=(s_2,-ws_2,-c)\, ,
\end{aligned}
$$
where without loss of generality, we can assume either case i),  $t_1\geq 0, s_2\leq 0$,
or  case ii), $t_1\geq 0$, $s_2\geq 0$.

Below, to avoid confusion, we will write $\Omega$ for the constant $c$,  in (\ref{nd25}).
We also assume $\epsilon \leq 1$.

\begin{corollary} The following bound holds true:
\label{nondegen1}
\begin{equation}
\label{restate}
c\cdot \frac{w}{1+w^2}\gtrsim \epsilon^6\, .
\end{equation}
Moreover,

\no
i) If $t_1\geq 0, s_1\leq 0$,
\begin{equation}
\label{boundsi1}
\epsilon^6\lesssim w\lesssim  \epsilon^{-6}\, .
\end{equation}
\begin{equation}
\label{restatei1}
\epsilon^3\sqrt{1+w^2}
\lesssim\mathfrak d(L_1,L_2)
\lesssim \sqrt w\, .
\end{equation}
\begin{equation}
\label{disttobari}
d^\H(o_1,e_{1,+}(L_1,L_2))
\lesssim\sqrt{w}+\frac{1}{\sqrt{w}}\, .
\end{equation}
\begin{equation}
\label{boundsi2}
\epsilon^6\lesssim c\lesssim 1\, .
\end{equation}
\begin{equation}
\label{boundsi3}
|t_1|+|s_2|\lesssim 1\, .
\end{equation}
\begin{equation}
\label{boundsi4}
|t_1+s_2|\lesssim \frac{1}{w}\, .
\end{equation}
\begin{equation}
\label{boundsi5}
w\cdot |t_1|\cdot |s_2|\lesssim 1\, .
\end{equation}

\no
ii) If $t_1\geq 0$, $s_2\geq 0$,
\begin{equation}
\label{boundsii1}
 w\lesssim \epsilon^{-6}\, .
\end{equation}
\begin{equation}
\label{restateii1}
\epsilon^3\sqrt{1+w^2}
\lesssim\mathfrak d(L_1,L_2)\lesssim \sqrt{1+w}\, .
\end{equation}
\begin{equation}
\label{disttobarii}
d^\H(o_1,e_{1,+}(L_1,L_2))
\lesssim\epsilon^{-3}\, .
\end{equation}
\begin{equation}
\label{boundsii2}
\epsilon^6\lesssim c\lesssim 1+\frac{1}{w}\, .
\end{equation}
\begin{equation}
\label{boundsii3}
|t_1-s_2|\lesssim 1\, .
\end{equation}
\begin{equation}
\label{boundsii4}
t_1+s_2\lesssim \frac{1}{w}\, .
\end{equation}
\begin{equation}
\label{boundsii5}
|c-wt_1s_2|\lesssim 1\, .
\end{equation}
\end{corollary}
\begin{proof} Since $\sin \theta(L_1,L_2)=2\cdot\frac{w}{1+w^2}$,  $(d^V(L_1,L_2))^2=2c$,
(\ref{restate}) is just a rewriting of (\ref{nd25}).
The  lower bounds in (\ref{restatei1}), (\ref{restateii1}), follow from  (\ref{restate}) and (\ref{mfd}).

Next note that since $d^\H(o_1,o_2)\leq 2$, from the box-ball principle
we have
\begin{equation}
\label{leqPhi}
|t_1-s_2|+w|t_1+s_2| +\sqrt{|2c-2wt_1s_2|}\lesssim 1\, .
\end{equation}

\no
Case i).
Since $t_1\geq 0, s_2\leq 0$ implies that   in the third term on the right-hand side of
(\ref{leqPhi}) we have $-wt_1s_2>0$, we get $c\lesssim 1$,
 which
gives the upper bound in (\ref{boundsi2}) and since $\mathfrak d(L_1,L_2)=2\sqrt{cw}$,
the upper bound in (\ref{boundsi1}) as well.
Using $c\lesssim 1$ and multiplying both sides of (\ref{restate}) by
$\frac{1+w^2}{w}$, gives (\ref{boundsi1}).
Relations (\ref{boundsi3})--(\ref{boundsi5}) follow immediately from
(\ref{leqPhi}) and $t_1\geq 0, s_2\leq 0$.

To prove (\ref{disttobari}),
 by considering the cases, $w\leq 1$, $w\geq 1$, and using the box ball principle,
(\ref{leqPhi}),
it suffices to show that
$$
t_1+\sqrt{\frac{c}{w}}\lesssim \frac{1}{\sqrt w}\, .
$$
 Since $c\lesssim 1$ we get $\sqrt{\frac{c}{w}}\lesssim \frac{1}{\sqrt {w}}$.
 From (\ref{boundsi3}) we get $t_1\lesssim 1$, which gives the case
$w\leq 1$.
 For $w\geq 1$, relations
(\ref{boundsi4}), (\ref{boundsi5}) imply $t_1\lesssim \frac{1}{\sqrt w}$,
which completes the proof.

\no
Case ii). Relation (\ref{boundsii2}) was already shown in the proof of case i).
 By using $t_1\geq 0, s_2\geq 0$,
relations (\ref{boundsii3})--(\ref{boundsii5}) follow directly from (\ref{leqPhi}).
 From (\ref{boundsii4}),
$$
c\lesssim 1+wt_1s_2\lesssim 1+w\cdot(t_1+s_2)^2\, ,
$$
which, by (\ref{boundsii4}), gives the upper bound for $c$ in
 (\ref{boundsii2}).  From this, the upper bound in (\ref{restateii1}) follows immediately,
as does the implication
$c\leq 5$ if $w\geq 1$.
 This, together  with (\ref{restate}), gives the upper bound for $w$ in
(\ref{boundsii1}).

To prove (\ref{disttobarii}), note that
 from (\ref{boundsii5}), by dividing through by $w$ and factoring, we get,
$$
\left|\sqrt{\frac{c}w}-\sqrt{t_1s_2}\right|\cdot \left|\sqrt{\frac{c}w}+\sqrt{t_1s_2}\right|
\lesssim \frac{1}{w}\, ,
$$
which, together with (\ref{boundsii4}), gives
$$
(1+w)\left|\sqrt{\frac{c}{w}}-t_1\right|\lesssim (1+w)\frac{1}{\sqrt{wc}}+(1+w)\left|\sqrt{t_1s_2}-t_1\right|\, .
$$
 From (\ref{boundsii3}), in case $w\leq 1$, and (\ref{boundsii4}), in
case $w\geq 1$,
it follows that  the second term on
the right-hand side is $\lesssim 1$. By using the lower bound in (\ref{restatei1})
to bound the first term on the right-hand side, the proof is complete.
\end{proof}

In the next corollary, by considering
separately the cases
$0<w\leq 1$, $1\leq w <\infty$, we obtain estimates
which depend only on $\epsilon$.
\begin{corollary}
\label{nondegen3}
${}$

\no
1) If $0<w\leq 1$, then
\begin{equation}
\label{restate11}
\epsilon^3
\lesssim\mathfrak d(L_1,L_2)
\lesssim 1\, ,
\end{equation}
\begin{equation}
\label{disttobar1}
d^\H(o_1,e_{1,+}(L_1,L_2))
\lesssim \epsilon^{-3}\, .
\end{equation}

\no
2) If $1\leq w<\infty$,
\begin{equation}
\label{bardist2}
\epsilon^3
\lesssim d^\H(b_+(L_1,L_2),b_-(L_1,L_2))
\lesssim 1
\end{equation}
\begin{equation}
\label{restate1}
\epsilon^3
\lesssim\mathfrak d(L_1,L_2)\lesssim \epsilon^{-3}\, .
\end{equation}
\begin{equation}
\label{disttobar2}
d^\H(o_1,e_{1,+}(L_1,L_2))
\lesssim\epsilon^{-3}\, .
\end{equation}
\end{corollary}
\begin{proof}
Relation (\ref{bardist2}), i.e.,
the bound on $2\sqrt{\frac{c}{w}}$,
for $w\geq 1$,  is a direct consequence of (\ref{boundsi1}),
(\ref{boundsi2}), (\ref{boundsii1}), (\ref{boundsii2}). The remaining relations
can be read off from Corollary \ref{nondegen2}.
\end{proof}

If $w\geq 1$, as a consequence of~\eqref{bardist2} we obtain a pair of
parallel lines $b_\pm(L_1,L_2)$ whose separation, $2\sqrt{\frac{c}{w}}$,
 is bounded below by $\gtrsim \epsilon^3$ and
above by a universal constant,
and  with a transversal $L_1$,  such that (\ref{restate1}),
(\ref{disttobar2}) hold. This will suffice for our application.

If $w\leq 1$, although $2\sqrt{\frac{c}{w}}$ is not bounded above,
$\mathfrak d(L_1,L_2)$ is bounded above by a universal constant, and below by
$\gtrsim \epsilon^3$, and
$d^\H(o_1,e_{1,+}(L_1,L_2))$ is bounded by a definite multiple of $\epsilon^{-3}$.
In this case we obtain a pair of skew lines making standard angle and controlled
separation as follows.

Let
$$
r=\sqrt{1+w^2}\cdot \sqrt{\frac{c}{w}}\, $$
 denote the distance from
the origin in $\R^2$ to $(\sqrt{\frac{c}{w}}, \pm\sqrt{cw}))=\pi(e_{1,\pm}(L_1,L_2))$.
The points
$$
\begin{aligned}
&\frac{\mathfrak d(L_1,L_2)+r}{r}\cdot \left(\sqrt{\frac{c}{w}}, \sqrt{cw}\right)\, ,\\
&{}\\
&\frac{r}{\mathfrak d(L_1,L_2)+r}\cdot \left(\sqrt{\frac{c}{w}}, \sqrt{cw}\right)\, ,
\end{aligned}
$$
lie on the projection of a line, $L\in{\rm lines}(H)$, with $z_i=L\cap L_i$,
 where
$$
\begin{aligned}
d^\H(z_1,e_{1,+}(L_1,L_2))&=\mathfrak d(L_1,L_2)\, ,\\
d^\H(z_2,e_{2,+}(L_1,L_2))&=\frac{r}{r+\mathfrak d(L_1,L_2)}\cdot\mathfrak d(L_1,L_2)\, .\\
\end{aligned}
$$
By direct computation, the slope $m(L)$ of $\pi(L)$ satisfies
\begin{equation}
\label{slopeL}
\frac{2}{3}\leq m(L)\leq\frac{5}{2}\, .
\end{equation}
 From this we get:

\begin{corollary}
\label{nondegen2}
The lines $L,b_+(L_1,L_2)$ are skew. The point
$\pi(L)\cap \pi(b_+(L_1,L_2))$ lies on the line segment
from $\pi(e_{1,+}(L_1,L_2))$ to $\pi(e_{2,+}(L_1,L_2))$. Moreover,
\begin{equation}
\label{final2}
\frac{2}{5}\leq \tan\theta(L,b_+(L_1,L_2))\leq \frac{3}{2}\, ,
\end{equation}
and
\begin{equation}
\label{final3}
\sqrt{\frac{2}{3}}\cdot\mathfrak d(L_1,L_2)
\leq d^V(L,b_+(L_1,L_2))
\leq \mathfrak d(L_1,L_2)\, .
\end{equation}
\end{corollary}
\begin{proof}
Relation (\ref{final2}) follows directly from (\ref{slopeL}).
By the discussion above,
the triangle with vertices, $(0,0),\pi(e_{1,+}(L_1,L_2)),\pi(e_{2,+}(L_1,L_2))$,
and the triangle with vertices $(0,0),\pi(z_1),\pi(z_2)$,
have equal areas. It follows that the triangles with vertices,
$\pi(e_{1,+}(L_1,L_2))$,
$\pi(z_1),\pi(L)\cap \pi(b_+(L_1,L_2)$,
$\pi(e_{2,+}(L_1,L_2)),\pi(z_2),\pi(L)\cap \pi(b_+(L_1,L_2)$,
(whose union is the symmetric difference of the previous ones)
have equal areas as well. From this and the geometric interpretation of
multiplication, we get (\ref{final3}).
\end{proof}

\section{Proof of Lemma \ref{leqi}}
\label{pfleqi}

In this section we prove Lemma \ref{leqi}.
We begin with some measure
theoretic preliminaries which play a role here and in
Section \ref{pfpdrrvc}.
Then we give the proof
for case of precisely monotone sets (for which the preliminaries
are not required); compare also the proof in \cite{ckmetmon}.
 Finally, show how to modify the argument to
obtain the general case of Lemma \ref{leqi}.

\subsection*{Measure theoretic preliminaries.}

Recall that $\mathcal L_3$ denotes  Lebesgue measure on $\H=\R^3$
which is a Haar measure of $\H$. Given $L\in  \lines (\H)$
we denote by $\mathcal H_L^1$ the Hausdorff measure induced by the
metric $d^\H$ on $L$. Recall that $\mathcal N$ denotes the unique left invariant
invariant measure on $\lines (\H)$ which is normalized so that $\mathcal
N(\lines(B_1(p)))=1$. Given $L\in \lines (\H)$ we denote by $[L]$ its
Heisenberg parallelism class, i.e., $[L]=\{gL:\ g\in \H\}$. $\H$
acts transitively on $[L]$. Therefore there exists a left invariant
measure $\mu_{[L]}$ on $[L]$, normalized so that for every measurable
$A\subset \H$ we have
\begin{equation}\label{eq:muL} \int_{[L]}
\mathcal H^1_{L'}(L'\cap A)\, d\mu_{[L]}(L')=\mathcal L_3(A)\, .
\end{equation}

The space of all parallelism classes will be denoted
$\mathfrak P=\{[L]:\ L\in \lines(\H)\}$. Each such parallelism class
$[L]$ is uniquely determined by an angle $\theta\in [0,2\pi)$,
corresponding to the angle of $\pi(L)\subset \R^2$ (recall that the
lines $L\in \lines (\H)$ are oriented). We thus there is an induced
measure $d\theta$ on $\mathfrak P$ which corresponds to the standard
measure on the circle $S^1$. By uniqueness
there
exists a constant $c>0$ such that for all integrable $f:\lines
(\H)\to \R$ we have:
\begin{equation}\label{eq:first splitting}
\int_{\lines(\H)} f\, d\mathcal N=c\int_{\mathfrak P}\left(\int_{L'\in [L]}f(L')d\mu_{[L]}(L')\right)\, d\theta([L])\, .
\end{equation}

Let
$\mathbb P\subset {\rm lines}(\H)\times \H$ denote
the space of oriented pointed lines, i.e., $
\mathbb P=\{(L,x):\  x\in L\}$. We shall use below the measure $\nu$ on $\mathbb P$ which is defined by
setting for every compactly supported continuous $f:\mathbb P\to \R$:
\begin{equation}\label{eq:def nu}
\int_{\mathbb P} f\, d\nu=\int_{\lines (\H)}\left( \int_{L} f(L,x)d\mathcal H_L^1(x)\right)d\mathcal N(L)\, .
\end{equation}

In the proof of Lemma \ref{leqi} we shall use the space of {configurations} $\mathcal C$, where a
 {\em configuration}
 is a quadruple $(L,x_1,x_2,[L'])$, where $L\in \lines (\H)$, $x_1,x_2\in L$, and $[L']\in \mathfrak P$
(i.e., a doubly pointed line
and a parallelism class). $\mathcal C$ carries two measures $\sigma_1,\sigma_2$, which are defined as follows.
Given $f:\mathcal C\to \R$ which is continuous and compactly supported let:
\begin{equation}\label{eq:def sigma1}
\int_\mathcal C f\, d\sigma_1=\int_{\lines(\H)}\int_{\mathfrak P}\left(\int_{L\times L}f(L,x_1,x_2,[L'])\,
d\left(\mathcal H_L^1\times\mathcal H_L^1\right)(x_1,x_2)\right)d\theta([L'])\, d\mathcal N(L)\, ,
\end{equation}
\begin{multline}\label{eq:def sigma2}
\int_\mathcal C f\, d\sigma_2\\=\int_{\mathfrak P}\left(\int_{[L']\times [L']} \left(\int_{L_1}
f(L(x),x,L(x)\cap L_2,[L'])d\mathcal H^1_{L_1}(x)\right)d\left(\mu_{[L']}\times\mu_{[L']}\right)(L_1,L_2) \right)\,
d\theta([L']),
\end{multline}
where in~\eqref{eq:def sigma2}, as in Section~\ref{cms}, given two parallel lines $L_1,L_2\in [L']$ with
distinct projections and $x\in L_1$,
$L(x)$ denotes the unique element of $\lines (\H)$ which passes through $x$ and intersects $L_2$.
 (In Section~\ref{cms}, $L(x)\cap L_2$ was denoted $x^*$).

Each of $\sigma_1,\sigma_2$ is absolutely continuous with respect to the other. Moreover, given a compact
subset of the Heisenberg group $K\subset \H$ and $a>0$, let $\mathcal C(K,a)$
denote the set of configurations
$(L,x_1,x_2,[L'])\in \mathcal C$ with $x_1,x_2\in K$ and
$d\left(\pi(x_1L'(0)^{-1}L'),\pi(x_2L'(0)^{-1}L')\right)\ge a$, i.e.,
 the unique lines in $[L']$ which pass through $x_1,x_2$ have
projections of distance at least $a$, and $x_1,x_2$ are in
the compact set $K$. One checks from the definitions that on the compact set $\mathcal C(K,a)$,
the Radon--Nikodym
derivatives $\frac{d\sigma_1}{d\sigma_2},\frac{d\sigma_2}{d\sigma_1}$ are continuous.
Hence for every measurable
$A\subset \mathcal C(K,a)$, we have
\begin{equation}\label{eq:non constant equivalence}
\sigma_1(A)\asymp \sigma_2(A),
 \end{equation}
 where the implied constants depend only on $K$ and $a$.
This observation will be used in the proof of Lemma \ref{leqi} below.

\subsection*{A consequence of small nonconvexity.}

Recall that the nonconvexity ${\rm NC}_{B_r(x)}(E,L)$ was defined in (\ref{dnc}).
The  following consequence of small nonconvexity will be used
in the present section and repeatedly in Section \ref{pfpdrrvc} (where the
main result concerns $\delta$-convex sets, which are more general than
$\delta$-montone sets).
\begin{lemma}
\label{lnc} Fix $p\in \H$ and $E\subset \H$.
Let  $L\in{\rm lines}(\H)$ satisfy
${\rm NC}_{B_1(p)}(E,L)<\delta$. Assume that
$[c,d]\subset[a,b]\subset L\cap B_1(p)$ and
\begin{equation}
\label{endlb}
\mathcal H^1_L([a,c]\cap E)>\delta\, ,\qquad \mathcal H^1_L([d,b]\cap E)>\delta\, ,
\end{equation}
then
$$
\mathcal H^1_L([c,d]\cap E')\leq \delta\, .
$$
\end{lemma}
\begin{proof}
 For all sufficiently small $\eta >0$, we have
 $\mathcal H^1_L([a,c]\cap E)\geq\delta+\eta$, $\mathcal H^1_L([d,b]\cap E)\geq\delta+\eta$.
 If $I\subset L\cap B_1(p)$ is an interval
which exceeds  the infimum on the right-hand side
of (\ref{dnc})
 by at most $\eta/2$, then the
intersection of  $I$ with both $[a,c]$ and $[d,b]$ must have positive $\mathcal H^1_L$
measure. Thus, $I=[e,f]\supset [c,d]$.  By the choice of $I$, we also have
 $\mathcal H^1_L([e,f]\cap E')< \delta+\eta/2$, which implies
$\mathcal H^1_L([c,d]\cap E')< \delta+\eta/2$. Letting $\eta$ tend to $0$ completes the proof.
\end{proof}

\subsection*{The case of precisely monotone sets.}



We call a set $E\subset B_1(p)$  {\it precisely monotone} if $E\cap
L$ and $E'\cap L$ are connected for all $L\in \lines(\H)$. Therefore, given such a
pair $(E,L)$, either $L\cap B_1(p)\subset E$, $L\cap B_1(p)\subset
E'$, or there exists a unique $q_{L,E}\in L$ such that
$(L\setminus\{ q_{L,E}\})\cap B_1(p)$ consists of two open
intervals, one of which is contained in $E$ and the other of which
is contained in $E'$.

Choose a pair of parallel lines, $L_1,L_2$ with distinct projections
 and recall that
$X=X(L_1,L_2)$ denotes the ruled surface which is the
union of lines, $L(x)$,  passing through $L_1$ and $L_2$.
Here, $\{x\}=L(x)\cap L_1$.
Note that $\H$ acts on $L_1,L_2, X(L_1,L_2)$ by left translation.

\begin{lemma}
\label{monochrome} Fix $p\in \H$. There is a left-invariant function
defined on pairs  $L_1,L_2$ of parallel lines with distinct
projections and taking values in $6$-tuples of relatively open
subsets of the surface $X(L_1,L_2)\cap B_1(p)$,
$\left\{W_j(L_1,L_2)\right\}_{j=1}^6$, such that if $E$ is a
precisely monotone subset of $B_1(p)$, then for some
$j(L_1,L_2,E)\in\{1,2,3,4,5,6\}$, $W(L_1,L_2,E)\defeq
W_{j(L_1,L_2,E)}(L_1,L_2)$ consists entirely of points of $E$, or
entirely of points of $E'$.
\end{lemma}
\begin{proof}
 We can assume that say
$L_1=(t,b,-bt)$, $L_2=(s,-b,-bs)$, with $\frac{1}{2}\leq b\leq 1$.
Then, in the notation of Lemma \ref{lemgeometry1},
 for $x=(t,b,-bt)\in L_1$, we have $x^*=(-t,-b,bt)$,
$L(x)=(r,r/b,-bt)$.

Let $I_1,I_2,I_3,$ denote the intervals, $(0,\frac{1}{6})$, $(\frac{1}{6},\frac{1}{3})$,
$(\frac{1}{3},\frac{1}{2})$, respectively, and let $-I_j$ be defined by
 $t\in -I_j$ if and only $-t\in I_j$. Define $W_1(L_1,L_2)$ to be
 $B_1(p)$ intersected with the union of those segments of the lines $L(x)$ which join
 $x\in L_1(I_1)$ to $x^*\in L_2(-I_1)$. Also, define $W_2(L_1,L_2)$
 to be $B_1(p)$ intersected with the union of those subrays of the lines $L(x)$
 for which  $x\in L_1(I_1)$, $x^*\in L_2(-I_1)$, whose endpoint is $x$ and which are disjoint to the
 segment $[x,x^*]$. Define $W_3(L_1,L_2), W_4(L_1,L_2)$ and
 $W_5(L_1,L_2),W_6(L_1,L_2)$ analogously, corresponding to $j=2,3$,
 respectively.


 There exists some $i(L_1,L_2,E)\in\{1,2,3\}$,
such that $q_{E,L_1}\not\in L_1(I_i)$  and $q_{E,L_2}\not\in
L_2(-I_i)$ (where possibly, one or both of $q_{E,L_1}, q_{E,L_1}$ do
not exist altogether).

 Fix $i=i(L_1,L_2,E)$ as above. There are four possibilities:
\vskip2mm

\no
 1a) $L_1(I_i)\subset E$, $L_1(-I_i)\subset E$,
\vskip1mm

\no
1b) $L_1(I_i)\subset E'$, $L_1(-I_i)\subset E'$,
\vskip1mm

\no
2a) $L_1(I_i)\subset E$, $L_2(-I_i)\subset E'$,
\vskip1mm

\no
2b) $L_1(I_i)\subset E'$, $L_2(-I_i)\subset E$.
\vskip2mm

In cases 1a), 1b),
we can take $j(L_1,L_2,E)=2i-1$, while in  cases 2a), 2b), we can take $j(L_1,L_2,E)=2i$.
\end{proof}

Take a small interval $[0,d]$ in the center of $\H$ and consider all
left translates $g(X(L_1,L_2))=X(g(L_1),g(L_2))$, where
$g\in [0,d]$. Clearly, for some subset $S=S(L_1,L_2,E)\subset
[0,d]$, of measure $\geq\frac{d}{12}$, if $g\in S$, then the
integers $j(g(L_1),g(L_2),E)$ will all coincide
 and in addition, either for all $g\in S$, $W(g(L_1),g(L_2),E)\subset E$, or for all $g\in S$,
$W(g(L_1),g(L_2),E)\subset E'$. Without essential loss of
generality, we can assume $0\in S$ and say $W(L_1,L_2,E)\subset E$.
Hence, $W(g(L_1),g(L_2),E)=g(W(L_1,L_2,E))\subset E$, for all
$g\in S$.

Choose an interior point $y\in W(L_1,L_2,E)\cap B_1(p)$, lying at a
definite distance from the boundary, choose a line $L(x)$ such that
$y\in L(x)$, and choose $L\in {\rm lines}(\H)$ making a definite
angle with $L(x)$. Let $[L]$ denote the collection of lines parallel
to $L$. By Lemma \ref{lemgeometry1} (and continuity) we can assume
that $d$ has been  chosen so small that if $g\in [0,d]$, $L'\in [L]$
and $L'\cap g(X(L_1,L_2))\cap B_{10d}(y)\neq \emptyset$, then $L'$
intersects $g(X(L_1,L_2))$ transversely. Parameterize each such
$L'$ such that $L'(0)=L'\cap W(L_1,L_2,E)$. Let $J$ denote the
smallest interval containing $S$.  Then the length of $J$ is $\geq
\frac{d}{12}$ and by monotonicity of $E$, we have $L'(J)\subset E$. This
completes the proof in the precisely monotone case.



\subsection*{Proof of Lemma \ref{leqi}.}

Suppose now that $E\subset B_1(p)$ is $\epsilon^2$-monotone on $B_1(p)$.
Choose a measurable mapping $L\mapsto I_{L,E}\subset L$ (e.g., via an application of the
measurable selection theorem in~\cite{Kuratowski-selection}) such that
 $I_{L,E}\cap B_1(p)$, $(I_{L,E})'\cap B_1(p)$ are intervals and
for almost all $L$,
\begin{equation}
\label{mi}
\mathcal H_L^1(I_{L,E}\, \triangle \, (E\cap L\cap B_1(p))\leq 2\cdot {\rm NM}_{B_1(p)}(L,E)\, .
\end{equation}
 Let $\overline{q}_{L,E}$ denote
the common boundary point of $I_{L,E},(I_{L,E})'$.

Define $I_1,I_2,I_3$ as above.
  Given a pair $L_1,L_2$
of parallel lines with distinct projections, define $i(L_1,L_2,E)$
as above, but replacing $q_{L,E}$ by $\overline{q}_{L,E}$.
Then, mutatis mutandis, define cases 1a)--2b) and
$\left\{W_j(L_1,L_2)\right\}_{j=1}^6,j(L_1,L_2,E),W(L_1,L_2,E)$ as above.

In the Claim below, ${\rm Center}(\H)$ is equipped with its natural measure $\L$
and $W(L_1,L_2,E)\subset X(L_1,L_2)$ is equipped with the natural surface
measure $\mathcal M$
on $X(L_1,L_2)$.

\no
{\bf Claim.} There exists a universal constant $c>0$ with the following properties.
Assume that $E\subset B_1(p)$ is $\epsilon^2$-monotone on $B_1(p)$. Then
there exists
a pair of parallel lines $L_1,L_2$ whose projections lie
at distance $\ge c$, such that for a fraction $\ge c$
of $g\in {\rm Center}(\H)\cap B_1(p)$,  the surface $W(g(L_1),g(L_2),E)$ has measure $\ge c$ and,
apart from a subset of measure
 $\leq \epsilon/c$, it
consists either entirely of points of $E$ or entirely of points in $E'$.

Assume provisionally, that the Claim holds.
Since $E\subset B_1(p)$ is
$\epsilon^2$-monotone, by~\eqref{eq:first splitting} with $f(L)={\rm NM}_{B_1(p)}(E,L)$ and
Markov's inequality~\eqref{cheb},
 we can choose a  parallelism class $[L]$ of lines which are all a definite amount transverse
to $X(L_1,L_2)$ (i.e., the angle between these lines and the surface $X(L_1,L_2)$ is larger
than a universal constant) such that apart
from a subset of measure $\lesssim \epsilon$ of lines in $[L]$, the remaining lines
are all $\lesssim \epsilon$-monotone.  From this, together the  Claim
 and Lemma \ref{lnc}, we directly obtain Lemma \ref{leqi}.
\vskip2mm

\begin{proof}[Proof of the Claim]
Rather than studying the space  of
surfaces $X(L_1,L_2)$ directly, it is advantageous to decompose
each such $X(L_1,L_2)$ into its collection of ruling lines, $L$, and then to
lift  considerations to
space of all triples $(L_1,L_2,L)$,
 where a {\it triple} $(L_1,L_2,L)$ is a set of lines  $L_1,L_2,L\in{\rm lines}(B_1(p))$, such that
$L_1,L_2$ are parallel in the Heisenberg sense, and $L_1\cap L=x_1$, $L_2\cap L=x_2$ are nonempty
and lie in $B_1(p)$. Recall the set of configurations $\mathcal C$ that was defined in the
paragraph following~\eqref{eq:def nu}.
Each configuration $(L,x_1,x_2,[L'])\in \mathcal C$ with $x_1,x_2\in B_1(p)$ determines, and is
determined by, the triple $(L_1,L_2,L)$,
where $L_1=x_1L'(0)^{-1}L'$ and $L_2=x_2L'(0)^{-1}L'$ (thus using previous notation,
$L=L(x_1)$ and $x_2=x_1^*$). Let $\mathcal C_1$
denote the set of configurations $(L,x_1,x_2,[L'])\in \mathcal C$ with $x_1,x_2\in B_1(p)$.

\begin{definition}
\label{consistent}
A pointed line $(L,x)$ is called {\it consistent} if one of the following two options holds true: either
 $x\in I_{L,E}$ and $x\in E$, or  $x\in (I_{L,E})'\cap B_1(p)$ and $x\in E'$. A configuration $(L,x_1,x_2,[L'])$
is called {\it consistent} if $(L_1,x_1)$,
 $(L_2,x_2)$, $(L,x_1)$, $(L,x_2)$ are all consistent, where $L_1=x_1L'(0)^{-1}L'$ and $L_2=x_2L'(0)^{-1}L'$.
Such a configuration is called {\it $\epsilon$-monotone} on
 $B_1(p)$  with respect to $E$ if $L_1,L_2,L$ are all $\epsilon$-monotone on $B_1(p)$ with respect to $E$.
\end{definition}

Let $G_\epsilon$ denote the set of configurations which are  consistent and
$\epsilon$-monotone on $B_1(p)$ and let $(G_\epsilon)'
=\mathcal C_1\setminus G_\epsilon$. By using Lemma \ref{lnc}, it is immediate to verify from the definitions
that if for some parallel lines
$L_1,L_2\in \lines (B_1(p))$ we  have:
\begin{equation}\label{eq:for discussion}
\mathcal H_{L_1}^1\left(x\in L_1\left(I_{j(L_1,L_2,E)}\right):\ (L(x),x,x^*,[L_1])\in
(G_{\epsilon})'\right)\lesssim \epsilon,
\end{equation}
then apart from a set of measure $\lesssim\epsilon$, the set $W(L_1,L_2,E)$
consists either entirely of points of $E$ or entirely of points in $E'$.

\begin{lemma}
\label{mbc} If $E$ is $\epsilon^2$-monotone on $B_1(p)$ then:
\begin{equation}
\label{badtriples} \sigma_1((G_{\epsilon})')\lesssim \epsilon,
\end{equation}
where $\sigma_1$ is defined as in~\eqref{eq:def sigma1}.
\end{lemma}
\begin{proof} We use the notation introduced in our discussion of measure
theoretical preliminaries at the beginning of this section. Fix $L\in \lines(B_1(p))$ and let
$G'(L,E)$ denote the set of $x\in L\cap B_1(p)$ such that $(L,x)$ is {\it not} consistent.
By definition, we have
\begin{equation}
\label{gp1} \mathcal H_L^1(G'(L,E))\leq 2\cdot {\rm
NM}_{B_1(p)}(L,E)\, ,
\end{equation}
so
\begin{equation}
\label{gp2} \int_{{\rm lines}(B_1(p))} \mathcal H_L^1(G'(L,E))\,
d\mathcal N(L)\leq  2\cdot {\rm NM}_{B_1(p)}(E)\le \epsilon^2\, .
\end{equation}

 For $x\in B_1(p)$, denote by $m(x)$ measure (with respect to $d\theta$)
of the set of those  $[L]\in \mathcal P$ such that for the unique
$L'\in [L]$ such that $x\in L'$, the pointed line $(L',x)$ is {\it
not} consistent. Then by~\eqref{eq:muL} and~\eqref{eq:first
splitting} we  deduce from~\eqref{gp2} that
\begin{equation}
\label{gp21} \int_{B_1(p)}m(x)\, d \L_3(x)\lesssim \epsilon^2\, .
\end{equation}
 For $x_1,x_2\in L$, define $F(x_1,x_2)$  by
\begin{equation}
\label{gp22} F(x_1,x_2)=m(x_1)+m(x_2)+\chi_{I_{L,E}\triangle (E\cap
L\cap B_1(p))}(x_1) +\chi_{I_{L,E}\triangle ( E\cap L\cap
B_1(p))}(x_2)\, .
\end{equation}
Then,
\begin{multline}
\label{gp23} \int_{{\rm lines}(B_1(p))}\left(\int_{x_1,x_2\in L\cap
B_1(p)} F(x_1,x_2)d\left(\mathcal H_L^1\times \mathcal
H_L^1\right)(x_1,x_2)\right) d\mathcal N(L) \\
\stackrel{\eqref{gp21}}{\lesssim} \epsilon^2+
\int_{{\rm lines}(B_1(p))}\mathcal H_L^1(I_{L,E}\, \triangle \, (E\cap L\cap B_1(p))d\mathcal N(L)
\stackrel{\eqref{mi}}{\lesssim} \epsilon^2+
{\rm NM}_{B_1(p)}(E)\lesssim \epsilon^2 \, .
\end{multline}
It follows from~\eqref{gp23} and~\eqref{eq:def sigma1}  that if we denote by $A\subseteq \mathcal C$
the set of
configurations $(L,x_1,x_2,[L'])$ which are
not consistent then $\sigma_1(A)\lesssim\epsilon^2$.

Let $B \subset \mathcal C$ denote the set of configurations $(L,x_1,x_2,[L'])$ which are not
$\epsilon$-monotone on
$B_1(p)$ with respect to $E$. This implies that for
$(L,x_1,x_2,[L'])\in B$ we have:
$$
{\rm
NM}_{B_1(p)}(L,E)+{\rm
NM}_{B_1(p)}(x_1L'(0)^{-1}L',E)+{\rm
NM}_{B_1(p)}(x_2L'(0)^{-1}L',E)\ge \epsilon.
$$
Hence, using Markov's inequality~\eqref{cheb}, combined with the identities~\eqref{eq:def sigma1},
\eqref{eq:first splitting},
we deduce that
$$
\sigma_1(B)\lesssim \frac{{\rm
NM}_{B_1(p)}(E)}{\epsilon}\le \epsilon.
$$
Since, by definition, $\sigma_1((G_\epsilon)')\le \sigma_1(A)+\sigma_1(B)$,  the proof of (\ref{badtriples}) is
complete.
\end{proof}

Set $a=\frac{1}{100}$ and recall that, as defined before~\eqref{eq:non constant equivalence}, that
$\widetilde{\mathcal{C}}=\mathcal{C}(B_1(p),a)$
is the set of all configurations $(L,x_1,x_2,[L'])\in \mathcal C$ such that $x_1,x_2\in B_1(p)$  and the
projections of the parallel lines $x_1L'(0)^{-1}L',x_2L'(0)^{-1}L'$
have distance at least $a$. It follows from~\eqref{eq:non constant equivalence} and~\eqref{badtriples} that
$\sigma_2\left((G_\epsilon)'\cap \widetilde{\mathcal{C}}\right)\lesssim \epsilon$, where the measure
$\sigma_2$ is as in~\eqref{eq:def sigma2}. It follows from~\eqref{eq:def sigma2}
that there exists a parallelism class $[L']\in \mathfrak P$ such that if we define
$$
F=\{L_1,L_2\in [L']:\ d(\pi(L_1),\pi(L_2))\ge a\}\, ,
$$
then
\begin{equation}\label{eq:before equiv}
\int_{F} \mathcal H_{L_1}^1\left(x\in
L_1\cap B_1(p):\ (L(x),x,x^*,[L'])\in (G_{\epsilon})'\right)d\left(\mu_{[L']}
\times\mu_{[L']}\right)(L_1,L_2)\lesssim \epsilon.
\end{equation}

Partition the set of all $(L_1,L_2)\in F\times F$ into equivalence classes,
where $(L_1,L_2)$ is equivalent to
$(L_1',L_2')$ if there exists
$g\in {\rm Center}(\H)$ such that $L_1'=g(L_1), L_2'=g(L_2)$ We
deduce from~\eqref{eq:before equiv} that
there exist $L_1,L_2\in F$, an index $j\in \{1,2,3,4,5,6\}$, and a subset
$D\subset {\rm Center}(\H)\cap B_1(p)$ of
measure $\gtrsim 1$,
such that for all $g\in D$ we have $j(g(L_1),g(L_2),E)=j$,
the surface $W(g(L_1),g(L_2),E)$ has measure $\gtrsim 1$, and, by another application of Markov's
inequality~\eqref{cheb},
$$
\mathcal H_{L_1}^1\left(x\in L_1(I_j):\ (L(x),x,x^*,[L_1])\in (G_{\epsilon})'\right)\lesssim \epsilon\, .
$$
The claim now follows from the discussion following~\eqref{eq:for discussion}, and hence the proof of
Lemma \ref{leqi} is complete.
\end{proof}

\section{Proof of Proposition \ref{pdrrvc}}
\label{pfpdrrvc}


In this section we will consider both lines in $\R^2$ and their horizontal
lifts to $\H$. It will be convenient  to use
notation, which differs somewhat
from that employed elsewhere in the paper. This change will also
prevent an undesirable proliforation of subscripts, which would result if we continued
to write $L$ (exclusively) for  horizontal lines in $\H$.

Let $\pi:\H\to \R^2$ denote  the canonical map .
Given a line segment $\tau_1\subset \R^2$ with $\tau_1(0)=(0,0)$,
 denote by $\widetilde{\tau_1}$
its unique lift to a segment of a horizontal  line in $\H$ emanating from the origin.
Similarly, consider a once broken line segment  $\tau_1\cup\tau_2$, i.e.,
a pair of line segments such that
the initial point of the line segment $\tau_2$ is the final point of $\tau_1$.
Again, we assume $\tau_1(0)=(0,0)$.
Denote by $\widetilde{\tau_1\cup\tau_2}$ the lift of $\tau_1\cup\tau_2$ to
a continuous broken horizontal line in $\H=\R^3$, emanating from
 the origin.

In place of the line $L$ in Proposition \ref{pdrrvc}, we will write
$\widetilde{\gamma}$.
Since both the statement and proof are somewhat
 technical, for purposes of exposition, we first consider the model case in which
 $\widetilde{\gamma}(0)\in E$,
$B_{Cr}(\widetilde{\gamma}(1))\subset E$ and
$E$ is precisely convex. In this case, we will show that for a large enough universal constant $C$, there exists
$c>0$ such that $B_{csr}(\widetilde{\gamma}(s))\subset E$, for all $0\leq s\leq 1$.

Recall that in Proposition~\ref{pdrrvc} we are given parameters $\kappa,\eta,\xi,r,\rho\in (0,1)$ such that
$\rho\le\frac12 \kappa r^2$ and $\rho\le c$, where $c\in (0,1)$ will be a small enough universal constant to
be determined below. For the reader's convenience we recall here the values of $\delta_1,\delta_2$ in
Proposition~\ref{pdrrvc}, namely~\eqref{delta1first} and~\eqref{delta2}:
\begin{equation}
\label{delta1first1}
\delta_1=\frac{c\kappa^3\eta^2\xi\rho^3}{r}\, ,
\end{equation}
\begin{equation}
\label{delta21}
\delta_2= c\kappa^6\eta^3\xi^2\rho^3r^6\, .
\end{equation}
In the proof we will also use an auxiliary parameter $\omega>0$, which will be given by:
\begin{equation}
\label{sm1}
\omega=\overline{c}\kappa^3\eta\xi r^2\, ,
\end{equation}
where $\overline{c}\ge \sqrt{c}$ will be an appropriately chosen  universal constant.

namely~\eqref{delta1first} and~\eqref{delta2}, we have:


\subsection*{Values of constants.}
For the reader's convenience we include here
a summary of the steps of the proof of Proposition \ref{pdrrvc},
in which the particular values of the constants
enter explicitly.  For the constraint $\rho\le \frac12\kappa r^2$, (\ref{fixedeta2}),
(\ref{ignore}), \eqref{eq:constraint r};
for $\omega$, \eqref{eq:first assumption on omega}, \eqref{delta22},
\eqref{eq:omega constraint upper};
 for $\delta_1$, \eqref{eq:first assumption on omega}, \eqref{eq:delta1 constraint};
for $\delta_2$, (\ref{delta22}), \eqref{eq: delta2 constraint}.


\subsection*{Proof of model case.}
 For reasons of exposition, we begin with the proof of the model case described above.

\begin{remark}
\label{cylinders}
Even though the statement of Proposition \ref{pdrrvc} pertains to
balls, the multiplicative structure of $\H$ and its relation to
the geometry will necessitate the introduction of certain product
cylinders, where the product structure corresponds to $\H=\R^2\times\R$.
In fact,
our argument will show that if a cylinder centered at $\widetilde{\gamma}(1)$
wtih base  radius
$r^2$  and  height $r$, is contained in $E$,
then a cylinder with base radius $\gtrsim sr^2$ and height $sr$
centered at $\widetilde{\gamma}(s)$, will be contained in $E$.
Here, both heights are measured with respect to the metric $d^\H$.
\end{remark}

\begin{remark}
\label{above}
Let $O_\theta$ be as in Section \ref{3steps}; see the subsection
titled: ``The Heisenberg group as a PI space".
By applying a suitable left translation and action by
$O_\theta$, we can assume $\widetilde{\gamma}$ is the line $(s,0,0)$.
In the present subsection and in the proof of Proposition \ref{pdrrvc}, we consider only
points $(\overline{a},\overline{b},\overline{c})\in B_{csr}(\widetilde{\gamma}(s))$
with $\overline{c}\geq 0$. Points with $\overline{c}\leq 0$
are handled by the symmetric argument i.e. by interchanging the roles of
$\phi$ and $\psi$ below.
\end{remark}

Let $\Theta\subset\R^2$ denote angular sector which is the
union of rays $\tau$ such that
$\tau(0)=\gamma(0)=(0,0)$ and  $\tau(1)\in B_{r/2}(\gamma(1))$, where here $B_{r/2}(\gamma(1))$ denotes the Euclidean disk in $\R^2$.
Let $\widetilde{\Theta}$ denote the union of
the lifts of rays $\tau\subset\Theta$.
By the convexity
of $E$, it follows that $\tau\subset\Theta$ implies
$\widetilde{\tau}(s)\in E$, for $0\leq s \leq 1$, where $s$ denotes arc length.
To prove the non-quantititative
version of (\ref{weakened}), it suffices to show
show that for any $\tau\subset\Theta$,
\begin{equation}
\label{fiber}
\pi^{-1}(\tau(s))\cap B_{csr}(\widetilde{\tau}(s))
\subset E\, .
\end{equation}

Set
\begin{equation}
\label{lambda}
\lambda=\arctan(r^2)\, ,
\end{equation}
and let $\phi,\psi$ denote the rays in $\Theta$ which make
an angle $\lambda$ with $\tau$. For the sake of proving~\eqref{fiber}, by applying a suitable rotation,
assume that $\tau$ is the $x$-axis, so that
we have $\widetilde{\tau}(s)=(s,0,0)$.
Choose the parameterization $u$ such that
$\widetilde{\phi}(u)=(u,-ur^2,0)$,
$\widetilde{\psi}(u)=(u,ur^2,0)$.

 Fix $s_0\in [0,1]$ and let $u_0$
be the unique value of $u$ such that the line segment
$\chi$, from $\phi(u_0)$ to $\psi(1)$, intersects
$\tau$ at $\tau(s_0)$, where
$\chi(0)=\phi(u_0)$. Thus
\begin{equation}
\label{u}
u_0=\frac{s_0}{2-s_0}\, .
\end{equation}
 Let $\phi_1$ denote
the segment of $\phi$ from $\tau(0)=\phi(0)$ to
$\phi(u_0)$. Let $\widetilde{\chi}$ denote the
segment of $\widetilde{\phi_1\cup\chi}$ which lifts $\chi$.
Choose the parameterization
$t$ such that
\begin{equation}
\label{chi0}
\widetilde{\chi}(t)
=u_0(1,-r^2,0)
\cdot t\left((1,r^2,0)-u_0(1,-r^2,0)\right)\, ,
\end{equation}
where above, as usual, ``$\cdot$'' denotes multiplication in $\H$.
If we define $d_0$ by $\chi(d_0)=\tau(s_0)$,
then from (\ref{chi0}),
it follows that
\begin{equation}
\label{d}
d_0=\frac{s_0}{2}\, ,
\end{equation}
\begin{equation}
\label{chi3}
\widetilde{\chi}(d_0)
=\left(s_0,0,\frac{2s_0^2 \cdot r^2}{2-s_0}\right)\, .
\end{equation}
By the assumed convexity of $E$
we have
$\widetilde{\phi}(u)\subset E$, for all $0\leq u\leq1$. From (\ref{chi0}),
for $t=1$, and the box-ball principle,
(\ref{boxball}), we can assume that $C$ has been chosen such that
 $\widetilde{\chi}(1)
\in B_{Cr}(\widetilde{\gamma}(1))\subset E$. Since $E$ is convex, this gives
$\widetilde{\chi}(d_0)\in E$. Similarly, by (\ref{boxball}), for
$$
0\leq v\leq\frac{2s_0^2\cdot r^2}{2-s_0} \, ,
$$
 the line
$(0,0,-v)\cdot \widetilde{\chi}$, intersects $\widetilde{\Theta}\cap E$ and
satisfies, $(0,0,-v)\cdot \widetilde{\chi}(1)\in B_{Cr}(\widetilde{\gamma}(1))\subset E$.
Again by convexity of $E$, we get $(0,0,-v)\cdot \widetilde{\chi}(d_0)\in E$.
It is straightforward to check that this implies that $B_{crs_0}(\widetilde{\tau}(s_0))\subset E$ for some universal constant $c>0$, as required.

\begin{remark}
\label{forlater}
 In our proof of Proposition \ref{pdrrvc},
it is necessary to replace $\widetilde{\chi}$ by a $1$-parameter a family of lines
$\widetilde{\chi_\theta}$, to be defined below. To this end, we
observe that there is  some flexibility in the definition
of $\widetilde{\chi}$. First note that from
(\ref{lambda}), (\ref{u}), (\ref{d}) and the law of sines, it follows that for
some universal $c_1>0$ we have
\begin{equation}
\label{angle}
c_1\cdot r^2\leq
\angle(\widetilde{\tau}'(s_0),
\widetilde{\chi}'(d_0))\,  .
\end{equation}
For  all
\begin{equation}
\label{theta}
0\leq \theta\leq c_2\cdot r^2\, ,
\end{equation}
with $c_2>0$ a small constant, let
$\widetilde{\chi_\theta}$ denote the horizontal line through
$\widetilde{\chi}(d_0)=\widetilde{\tau}(s_0)$
obtained by rotating $\widetilde{\chi}$ by an angle $\theta$ about the point
 $\widetilde{\chi}(d_0)$, in the clockwise direction. The constant $C$ can be chosen such that
$\widetilde{\chi_\theta}$ intersects $B_{Cr}(\widetilde{\gamma}(1))$.
It will also intersect $\widetilde{\Theta}\cap E$. In fact,
if $q=g(\widetilde{\tau}(s_0))\in B_{cs_0r^2}(\widetilde{\tau}(s_0))$,
then $g(\widetilde{\chi}_\theta)$  intersects
both $\Theta\cap E$ and $B_{Cr}(\widetilde{\gamma}(1))$.
It follows that for {\it any} $q\in B_{cs_0r^2}(\widetilde{\tau}(s_0))$,
and any $\widetilde{\chi_\theta}$ with $\theta$ as in (\ref{theta}),  in verifying that
$\pi^{-1}(\pi(q))\cap B_{crs_0}(\widetilde{\tau}(s_0))\subset E$,
we can use only the lines which are parallel to some $\widetilde{\chi_\theta}$.
\end{remark}


\subsection*{Intermediate case and proof.}
Before going to the proof of Proposition \ref{pdrrvc}, we consider a slight
generalization of the model case.
As before, assume that $B_{Cr}(\widetilde{\gamma}(0))\in E$ and
$E$ is convex. Assume in addition that $g(\widetilde{\gamma}(0))\in E$ for some $g$, with
\begin{equation}
\label{fixedeta}
d^\H(g(\widetilde{\gamma}(0)),\widetilde{\gamma}(0))\leq  \rho\, ,
\end{equation}
where recall that we are assuming that  $\rho\le \kappa r^2$. The box-ball principle
(\ref{boxball}), together with~\eqref{fixedeta}, implies that for all $s>0$ we have
\begin{equation}
\label{dist}
d^\H(\widetilde{\gamma}(s),g(\widetilde{\gamma}(s))\lesssim \max\left\{\rho,\sqrt{s\rho}\right\}.
\end{equation}
Since $\rho\le r^2$ it follows from~\eqref{dist} that provided $C$ is large enough, the assumption
$B_{Cr}(\widetilde \gamma(1))\subseteq E$ implies that
$B_{\frac12 Cr}(g\widetilde \gamma(1))\subseteq E$. Now, a computation like those
above shows that for all $0\leq s\leq 1$ we have
\begin{equation}
\label{fixedeta1}
B_{csr}(g\widetilde{\gamma}(s))\subset E\quad {\rm for\,\, all}\  s\in[0,1]\, ,
\end{equation}
where $c>0$ is a universal constant. It follows from~\eqref{fixedeta}, the bound $\rho\le \kappa r^2$,
and~\eqref{fixedeta1}, that there exists a constant $c'>0$ such that
\begin{equation}
\label{fixedeta2}
B_{\frac12 csr}(\widetilde{\gamma}(s))\subset E\, ,\quad {\rm for\,\, all}\  s\in[c'\kappa,1]\, .
\end{equation}
\begin{remark}
\label{rrho}
It is worthwhile to emphasize here that in order to obtain (\ref{fixedeta2}), we must take $\rho$
proportional to $r^2$ (rather than to $r$); compare also  the use
cylinders, $B_{csr^2}(\tau(0))\times[-(sr)^2,(sr)^2]$
in the model case above and below.
\end{remark}

\begin{proof}[Proof of Proposition \ref{pdrrvc}]
 The arguments in the model and intermediate cases
use $\widetilde{\Theta}\subset E$, respectively,
$g(\widetilde{\Theta})\cap B_1(\widetilde{\gamma}(0))\subset E$.
Since in the general case, our hypotheses are measure theoretic,
we will employ integral geometry and in particular,
the measure theoretical preliminaries given in Section \ref{pfleqi}.

The first (and most involved) step is
to show the existence
of  a thickened version $\bbS$ of the
angular sector, $g(\widetilde{\Theta})\subset E$,
of the intermediate case:
\begin{equation}
\label{bbSdef}
\bbS=\bigcup_{\zeta\in A} g\zeta(\widetilde{\Theta})\, ,
\end{equation}
for some $g\in \H$, where $A$ is a suitably chosen subset of the center of $\H$ such that
$g\zeta(\widetilde{\gamma}(0))\in B_\rho(\widetilde{\gamma}(0))\cap E$ for all $\zeta\in A$.
We think of  $\bbS$ as a ``thin slab''.

The slices $g\zeta(\widetilde{\Theta})$
of $\bbS$ will be shown to consist almost entirely of points of $E$.
Hence, the vertical structure exhibited in (\ref{bbSdef}) implies
that the mass of $E\cap \bbS$ is very evenly distributed. In the presence
of sufficiently small nonconvexity of $E$,
this compensates for the fact that the mass of
$E\cap \bbS$, while subject to a definite lower bound depending on $\eta$ (see (\ref{bbsp3}))
might be still be very small, which might otherwise prevent $\bbS$ from serving as
an adequate substitute for $\widetilde{\Theta}$.

The following lemma, summarizes the essential properties of $\bbS$.
Eventually, we will be interested
in the behavior of the intersection of
$\bbS$   with the cylinder (with coordinate description)
$$B_{csr^2}(\tau(s))\times [-c\cdot(s\cdot r)^2,c\cdot(s\cdot r)^2],$$
whose intrinsic  height is $2(rs)$ and whose
measure is $\asymp (s\cdot r^2)^2\cdot (s\cdot r)^2$.
This may help to explain the right-hand sides of (\ref{bbsp1})--(\ref{bbsp4}).

\begin{lemma}
\label{lbbS}
There exists $g\in \H$ satisfying~\eqref{fixedeta} and $A\subset \mathrm{Center}(\H)$ such that for $\bbS$ as in (\ref{bbSdef}):

\no
a)
\begin{equation}
\label{bbsp1}
\bbS\subset g\left(\mathbb R^2\times I\right)\, ,
\end{equation}
for some $I\subset [-\rho^2,\rho^2]\subset\mathrm{Center}(\H)$, with
\begin{equation}
\label{bbsp11}
\L(I)\asymp\xi\rho^2\, ,
\end{equation}
where $\L$ denotes the Hausdorff measure on the center of $\H$, with its induced $\frac12$-snowflake metric.


\no
b) $\bbS$ has a definite thickness:
\begin{equation}
\label{bbsp3}
\L(A)\gtrsim
\eta\xi \rho^2\, .
\end{equation}

\no
c) $\bbS$
consists almost entirely of points of $E$, i.e., for all $\zeta\in A$:
\begin{equation}
\label{bbsp4}
\L_2  \left(\left( g\zeta\left(\widetilde{\Theta}\right)\setminus
(B_\rho(\widetilde{\gamma}(0))\cup B_r(\widetilde{\gamma}(1)))\right)\cap E'\right)
\lesssim \omega r^2 \, .
\end{equation}
\end{lemma}

Granted Lemma \ref{lbbS}, the proof of
Proposition \ref{pdrrvc} is not difficult to complete. We will now give a brief
overview of the argument to come.
Since $E$ is $\delta_2$-convex, we can chose $\widetilde{\chi_\theta}$,
as in Remark \ref{forlater} such that for a fraction very close to $1$ of lines
$L\in [\widetilde{\chi_\theta}]$, the nonconvexity,
${\rm NC}_{B_{Cr}(\widetilde \gamma(0))}(E,L)$ is very small.
By a quantitative version of the proofs of previously considered model and
 intermediate cases,
it follows that such lines consist almost entirely of points of $E$.
The argument uses  (\ref{bbsp3}),  (\ref{bbsp4}), and Lemma \ref{lnc} below.  In fact,
Lemma \ref{lnc} enters, in  similar fashion, in the proof of Lemma \ref{lbbS}.

\begin{proof}[Proof of Lemma \ref{lbbS}.]
 In the next few subsections we will show that  for most points,
$g(\widetilde{\gamma}(0))\in E\cap B_\rho(\widetilde{\gamma}(0))$,
the sector $g(\widetilde{\Theta})$
consists almost entirely of points of $E$. The remainder of the construction
will be a simple consequence of this fact.



\subsection*{
Implication of almost full measure of $E$ in $B_{Cr}\left(\widetilde{\gamma}(1)\right)$.}
By the assumption~\eqref{definite} we have:
\begin{equation}
\label{lowermeaseta}
\L_3(E\cap B_\rho(\widetilde{\gamma}(0)))
\ge\eta\cdot\L_3(B_\rho(\widetilde{\gamma}(0)))\gtrsim \eta\rho^4\, .
\end{equation}
Let $C>0$ be the (large enough, but fixed) constant that was chosen in the intermediate case of the proof.
Assume as in~\eqref{bigdensity} that $L(1)\in {\rm int}_{\delta_1,Cr}(E)$. 
chosen to be a large enough universal
Recall that the parameter $\omega$ is defined in~\eqref{sm1}. At the end of the proof,
it will become clear why it is necessary to choose this value of $\omega$;
see \eqref{eq:first assumption on omega}, \eqref{delta22},
\eqref{eq:omega constraint upper}.

 For $\tau\subset \Theta$,
set
\begin{equation}
\label{Sdef}
S_{[\widetilde \tau]}=\left\{\hat\tau\in [\widetilde{\tau}]\, \left| \,
0<\mathcal H_{\hat{\tau}}^1(\hat\tau\cap E\cap B_\rho(\widetilde{\gamma}(0)))
\leq\omega\cdot\eta\cdot\rho\right.\right\}\, ,
\end{equation}
\begin{equation}
\label{Tdef}
 T_{[\widetilde \tau]}=\left\{\hat{\tau}\in [\widetilde{\tau}]\, \left|\, 0<
\mathcal H_{\hat{\tau}}^1(\hat{\tau}\cap E\cap B_\rho(\widetilde{\gamma}(0))\, {\rm and}\,
\mathcal H_{\hat{\tau}}^1\left(\hat{\tau}\cap E\cap\,\, B_{\frac{1}{2}r}(\widetilde{\gamma}(1))\right)
\leq\frac{1}{2}\cdot r\right.\right\}\, .
\end{equation}

Without loss of generality, we can assume that
$\kappa$
is  small enough that if
$\hat{\tau}\in S_{[\widetilde \tau]}$ then
$\hat{\tau}(1)\in B_{\frac{1}{2}r}(\widetilde{\gamma}(1))$; see (\ref{dist}).
 From~\eqref{bigdensity}, (\ref{delta1first}), (\ref{eq:muL}), we get:
\begin{equation}
\label{L2est}
\mu_{[\widetilde \tau]}\left( T_{[\widetilde \tau]}\right)\lesssim \delta_1r^3\le  \omega\eta\rho^3\,  ,
\end{equation}
provided that the following lower bound on $\omega$ holds:
\begin{equation}
\label{eq:first assumption on omega}
\omega\ge \frac{\delta_1}{\eta}\cdot\left(\frac{r}{\rho}\right)^3\, .
\end{equation}
An inspection of our choice of parameters~\eqref{delta1first1}, \eqref{sm1} shows
that~\eqref{eq:first assumption on omega}
does indeed hold in our setting. (Note that here, we used  the requirement
 $\overline c\ge \sqrt c\ge c$.)



 \subsection*{Estimates in the space of pointed lines.} Recall that $\nu$ denotes the measure on $\mathbb{P}$,
the space of pointed lines, as in~\eqref{eq:def nu}.
 Let
\begin{equation}
\label{product6}
\begin{aligned}
U^\# &=\left\{(\hat{\tau},\hat{\tau}(0))\, \left|\,
  \hat{\tau}(0)
\in E\cap B_\rho(\widetilde{\gamma}(0))\ {\rm and}\ \hat{\tau}\subset \hat{\tau}(0)\widetilde \gamma(0)^{-1}\left(\widetilde{\Theta}\right)\right.\right\}\, ,\\
&{}\\
S^\# &=\left\{(\hat{\tau},\hat{\tau}(0))\, \left|\, \hat{\tau}\in S\cap g(\widetilde{\Theta}) ,\,
       \hat{\tau}(0)=g(\widetilde{\gamma}(0))\in
E\cap B_\rho(\widetilde{\gamma}(0))\right.\right\}\, ,\\
&{}\\
T^\#&=\left\{(\hat{\tau},\hat{\tau}(0))\, \left|\,
\hat{\tau}(0)\in
E\cap B_\rho(\widetilde{\gamma}(0))\ {\rm and}\
\hat{\tau}\in T_{[\hat{\tau}]}\ {\rm and} \  \hat{\tau}\subset\hat{\tau}(0)\widetilde \gamma(0)^{-1}\left(\widetilde{\Theta}\right.\right)
\right\}\, .
\end{aligned}
\end{equation}
Then
\begin{equation}
\label{product2}
\nu\left(U^\#\right)\asymp r^2\L_3(E\cap B_\rho(\widetilde{\gamma}(0)))\gtrsim  \eta r^2\rho^4 ,
\end{equation}
\begin{equation}
\label{product3}
\nu\left(S^\#\right)\lesssim\omega\eta r^2\rho^4\, ,
\end{equation}
\begin{equation}
\label{product4}
\nu\left(T^\#\right)\lesssim \omega\eta r^2\rho^4\, .
\end{equation}

\no
Relation~\eqref{product2}  follows from~\eqref{eq:muL}, \eqref{eq:first splitting}, \eqref{eq:def nu},
\eqref{lambda} and~\eqref{lowermeaseta}.
Relations (\ref{product3}), (\ref{product4}) follow similarly, using~\eqref{Sdef} and~\eqref{L2est}, respectively.

Recall that $E$ is $\delta_2$-convex on $B_{2C}(\widetilde \gamma(0))$, for
$\delta_2$  as in (\ref{delta2}).
Define:
\begin{equation}
\label{Vdef}
V=\left\{L\in \lines\left(B_{2C}(\widetilde \gamma(0))\right)\, \left|\, {\rm NC}_{B_{2C}(\widetilde \gamma(0))}(E,L)
\geq\frac{\omega}{2}\right.\right\},
\end{equation}
\begin{equation}
\label{vsharpdef}
V^\# =\left\{(L,x)\in \mathbb P\, \left|\, L\in V\  {\rm and} \  x\in E\cap B_\rho(\widetilde{\gamma}(0))\right.\right\}.
\end{equation}
By (\ref{dnc}), Markov's inequality~\eqref{cheb} and~\eqref{eq:def nu}, we have
\begin{equation}
\label{product5}
\nu\left(V^\#\right)\lesssim \frac{\delta_2}{\omega}\cdot\rho\le \omega \eta r^2\rho^4\, ,
\end{equation}
provided that in addition to~\eqref{eq:first assumption on omega} we also assume that
\begin{equation}
\label{delta22}
\omega\ge \sqrt{\frac{\delta_2}{\eta r^2\rho^3}}\, .
\end{equation}
An inspection of our choice of parameters~\eqref{delta21}, \eqref{sm1} shows that~\eqref{delta22} holds in
our setting (note that we used here the requirement $\overline c\ge \sqrt c$).

It follows from (\ref{product2})--(\ref{product4}) and
(\ref{product5}) that
\begin{equation}
\label{Wsharpmeas}
  \nu \left(S^\#\cup T^\#\cup V^\#\right)
\lesssim \omega\cdot\nu\left(U^\#\right).
\end{equation}
Define
\begin{equation}
\label{Wsharpdef}
W^\#=U^\#\setminus\left(S^\#\cup T^\#\cup V^\#\right)\, .
\end{equation}
Then, for  $(\hat \tau,\hat\tau(0))\in W^\#$, the definitions (\ref{product6}), \eqref{vsharpdef}, together with
Lemma~\ref{lnc}, imply that if we set
$$
J=\left[\rho+\frac{\omega}{2},1-\frac{r}{2}-\frac{\omega}{2}\right]\, ,
$$
then
\begin{equation}
\label{Wsharpest}
\mathcal H^1_{\hat\tau}(E'\cap \hat\tau(J))\leq \frac{\omega}{2}\, .
\end{equation}

 For $x\in E\cap B_\rho(\widetilde{\gamma}(0))$, $x=g(\widetilde{\gamma}(0))$, put
\begin{equation}
\label{frakGdef}
\frak G(x)=\left\{\hat \tau\subset g(\widetilde\Theta)\, \left|\,
(\hat \tau,x)\in W^\#\right.\right\}\, .
\end{equation}
Note that, using the above notation, the set of all
$\hat \tau\subset g(\widetilde\Theta )$ with
$\hat\tau(0)=x$ is parameterized by the corresponding angle
$\phi$ in  the sector given by~\eqref{lambda}. It follows from (\ref{Wsharpmeas}),
 \eqref{product2},
\eqref{lowermeaseta}, that there exists a large enough
universal constant $c>0$ such that if we put
\begin{equation}
\label{Eonedef}
E_1=\left\{x\in E\cap B_\rho(\widetilde{\gamma}(0))\, \left|\, d\phi(\frak G(x))
\leq c\omega r^2\right.\right\}\, .
\end{equation}
then
\begin{equation}
\label{Eoneest}
\L_3(E_1)\ge
\frac12 \L_3(E\cap B_\rho(\widetilde{\gamma}(0)))\gtrsim \eta\rho^4\, .
\end{equation}

In the following corollary and below, we consider
the splitting of the measure
\begin{equation}
\label{centprod}
\L_3=\L\times\L_2=\mathcal H^2\times \L_2\, ,
\end{equation}
relative to the  product structure on
$\H=\R^3$, where
$\L=\mathcal H^2$ corresponds to the center of $\H$
(with its induced
$\frac{1}{2}$-snowflake metric) and
$\L_2$ to the plane $(x,y,0)\subset \R^3$.
In connection with (\ref{goodcoset}) below, recall
that
$\L_3(B_\rho(\widetilde{\gamma}(0)))\asymp\rho^4=\rho^2\times\rho^2$,
where the product corresponds to the decomposition in (\ref{centprod}).
\begin{corollary}
\label{gofTheta}
There exists  $q\in \pi(E_1)$ such that
\begin{equation}
\label{goodcoset}
\eta\rho^2\lesssim\L(\pi^{-1}(q)\cap E_1)\lesssim \rho^2 \, .
\end{equation}
 If $g(\widetilde{\gamma}(0))
\in \pi^{-1}(q)\cap E_1$ then
\begin{equation}\label{EoneThetaest}
\L_2  \left(E'\cap \left(g(\widetilde{\Theta})\setminus \left(B_\rho\left(\widetilde{\gamma}(0)\right)
\cup B_r\left(\widetilde{\gamma}(1)\right)\right)\right)\right)
\lesssim \omega r^2\, .
\end{equation}
\end{corollary}
\begin{proof}
Relation (\ref{goodcoset}) follows from the splitting of the measure in (\ref{centprod}) and \eqref{Eoneest}.
Relation (\ref{EoneThetaest})
follows from the definitions of $U^\#, S^\#,T^\#, V^\#,W^\#$, together with \eqref{Wsharpest}, (\ref{frakGdef}),
(\ref{Eonedef}).
\end{proof}



\subsection*{Construction of $\bbS$; completion of the proof of Lemma \ref{lbbS}.}

By virtue of~\eqref{goodcoset} we can choose an interval $I\subset [-\rho^2,\rho^2]$  such that
\begin{equation}
\label{I1}
\L(I)\asymp\xi\rho^2\, ,
\end{equation}
\begin{equation}\label{I2}
\L\left(\pi^{-1}(q)\cap E_1\cap q(I)\right)\gtrsim \eta\xi\rho^2\, .
\end{equation}


Define a subset $A$ of the center of $\H$ by
$$
A= \pi^{-1}((0,0))\cap q^{-1}(\pi(E_1))\cap I\, .
$$
Choose $g=q\widetilde \gamma(0)^{-1}$.  Then $g$ satisfies~\eqref{fixedeta},
and by~\eqref{I1}, \eqref{I2},
a), b) of Lemma  \ref{lbbS} hold. Since by the definition of $q$ ,
 for all $\zeta\in A$, we have
$g\zeta\widetilde \gamma(0)\in \pi^{-1}(q)\cap E_1$, Corollary \ref{gofTheta}
implies that part c) holds as well.
\end{proof}

\subsection*{Completion of the proof of Proposition \ref{pdrrvc}.}



By choosing $c$ in Proposition \ref{pdrrvc}  sufficiently small, we can  assume without essential
loss of generality that $g$ of
Lemma \ref{lbbS}  is the point $(0,0,0)$ and that $I$ is symmetric about the origin.
 Fix $s\in [\kappa,1]$ and let $\tau$,
 $\widetilde{\chi_\theta}$ with  $0\leq \theta\leq \frac{r^2}{100} $ be as in
Remark \ref{forlater}, with $\theta$ to be determined below.

Let $\S$ denote the square with center $\tau(s)$ and  side length $csr^2$, with
one side parallel to $\chi_\theta$.  Of the two sides  of $\S$ which are orthogonal
to $\chi_\theta$, let $\alpha$ denote
the one which is furthest from the origin $(0,0)\in \R^2$.  Let $\mathcal R$ denote
the rectangle with one side parallel to $\chi_\theta$ and of length, $2sr^2+(s/2)$,
and such the two side which are orthogonal to $\chi_\theta$, the one furthest from
$(0,0)\in\R^2$ is $\alpha$. In particular, $\S\subset\mathcal R$. By a standard covering argument,
 it will suffice to show that at most a fraction
$\xi$ of points of
$\S\times \left[0,(csr)^2\right]$
lie in $E'$; see (\ref{weakened}).
Since by our assumptions $\rho\le \frac12\kappa r^2$ and $s\ge \kappa$,
\begin{equation}
\label{ignore}
\xi\rho^2\leq \frac12 \xi(s r)^2\, ,
\end{equation}
it suffices to show that at most a fraction $\xi$ of the points of
$\S\times ([0,(c s r)^2]\setminus I)$ lie in $E'$.

Let
$$
\frak F_\theta= \left\{L\in \left[\widetilde{\chi_\theta}\right]\,
\left|\, L\cap (\alpha\times ([0,(csr)^2]\setminus I))\ne\emptyset\right.\right\}\, .
$$
Note that $\mathfrak F$ can (essentially) be described alternatively as consisting
of those lines in $\left[\widetilde{\chi_\theta}\right]$ which intersect $\mathcal R$. If follows from~\eqref{eq:muL} that
\begin{equation}
\label{ell3}
\mu_{\left[\widetilde{\chi_\theta}\right]}\left(\frak F_\theta\right)\gtrsim  (s r^2) (s r)^2=s^3r^4\, .
\end{equation}
Moreover, since $E$ is $\delta_2$-convex on $B_{2C}\left(\widetilde \gamma(0)\right)$,
there exists $\theta\in \left(0,\frac{r^2}{100}\right)$ such that
\begin{equation}
\label{nc1}
\int_{L\in \frak F_\theta} {\rm NC}_{B_{2C}\left(\widetilde \gamma(0)\right)}(E,L)
d\mu_{\left[\widetilde{\chi_\theta}\right]}(L)\leq \frac{100\delta_2}{r^2}\, .
\end{equation}
We shall fix $\theta$ as in~\eqref{nc1} from now on. It follows from~\eqref{ell3}, \eqref{nc1}
that there exists a universal constant $c''>0$ such that if we define
\begin{equation}\label{eq: def star}
\mathfrak F^*_\theta=\left\{L\in \mathfrak F_\theta\left| {\rm NC}_{B_{2C}\left(\widetilde \gamma(0)\right)}(E,L)\le
\frac{c''\delta_2}{\xi s^3r^6}\right.\right\}
\end{equation}
then
\begin{equation}\label{eq:lower star}
\mu_{\left[\widetilde{\chi_\theta}\right]}\left(\frak F_\theta^*\right)\ge \left(1-\frac{\xi}{8}\right)
\mu_{\left[\widetilde{\chi_\theta}\right]}\left(\frak F_\theta\right)\, .
\end{equation}
Lemma~\ref{lbbS} implies that
$$
\L_3  \left(\left( \bbS\setminus
(B_\rho(\widetilde{\gamma}(0))\cup B_r(\widetilde{\gamma}(1)))\right)\cap E'\right)\lesssim \omega r^2\xi\rho^2\, ,
$$
and hence, using~\eqref{eq:muL} and Markov's inequality~\eqref{cheb}, we see that there exists a
universal constant $c'''>0$ such that if we define
\begin{equation}
\label{eq: def doublestar}
\mathfrak F^{**}_\theta
=\left\{L\in \mathfrak F^*_\theta\left| \mathcal H_L^1\left(\left( \bbS\setminus
(B_\rho(\widetilde{\gamma}(0))\cup B_r(\widetilde{\gamma}(1)))\right)\cap E'\cap L\right)\le
 \frac{c'''\omega\rho^2}{s^3r^2}\right.\right\}\, ,
\end{equation}
then
\begin{equation}\label{eq:lower doublestar}
\mu_{\left[\widetilde{\chi_\theta}\right]}
\left(\frak F_\theta^{**}\right)\ge \left(1-\frac{\xi}{4}\right)
\mu_{\left[\widetilde{\chi_\theta}\right]}
\left(\frak F_\theta\right)\, .
\end{equation}
Note that the definitions of
$\frak F_\theta$ and $\bbS$, together with~\eqref{bbsp11}, \eqref{bbsp3},
 imply that for all $L\in \frak F_\theta$ we have
\begin{equation}
\label{eq:entire slab}
\mathcal H_L^1\left(\left( \bbS\setminus
(B_\rho(\widetilde{\gamma}(0))
\cup B_r(\widetilde{\gamma}(1)))\right)\cap B_{s(1-\frac{r}{8})}\left(\widetilde{\gamma}(0)\right)
\cap L\right)
\gtrsim \eta\xi\rho^2.
\end{equation}
Thus, by the definition of $\frak F_\theta^{**}$, and recalling that $s\ge \kappa$, we see that provided that
\begin{equation}\label{eq:omega constraint upper}
\omega\le
c^* \eta\xi \kappa^3r^2,
\end{equation}
where  $c^*>0$ is a small enough absolute constant, we have for all $L\in \mathfrak F_\theta^{**}$,
\begin{equation}
\label{eq:H1 in s-ball}
\mathcal H_L^1 \left(E\cap B_{s(1-\frac{r}{8})}\left(\widetilde{\gamma}(0)\right)
\cap L\right)\gtrsim \eta\xi\rho^2.
\end{equation}
We shall choose $\overline c=c^*$ in~\eqref{sm1}, thus completing our choice of $\omega$.

Assuming also that for a small enough constant, $c^{**}>0$, we have
\begin{equation}
\label{eq: delta2 constraint}
\delta_2\le c^{**}
\eta\xi^2 \kappa^3 r^6\rho^2,
\end{equation}
which follows from our choice of $\delta_2$ in~\eqref{delta21},
provided $c$ is small enough,
we conclude from the definition of $\mathfrak F_\theta^*$,
together with~\eqref{eq:entire slab}, that for all
$L\in \mathfrak F_\theta^{**}$,
\begin{equation}
\label{NC**}
{\rm NC}_{B_{2C}\left(\widetilde \gamma(0)\right)}(E,L)
\le\frac14 \eta\xi\rho^2,
\end{equation}
and
\begin{equation}\label{eq:NC small on double star}
\mathcal H_L^1 \left(E\cap B_{s(1-\frac{r}{8})}\left(\widetilde{\gamma}(0)\right)
\cap L\right)\ge {\rm NC}_{B_{2C}\left(\widetilde \gamma(0)\right)}(E,L)\,  .
\end{equation}

Recall also,  assumption~\eqref{bigdensity}, which implies  that
$\L_3(E'\cap B_r(\widetilde{\gamma}(1)))\lesssim \delta_1 r^4$.
Moreover, by the
definition of $\frak F_\theta$ we
have $\mathcal H_L^1\left(L\cap B_r(\widetilde{\gamma}(1))\right)
\gtrsim r$ for all $L\in \frak F_\theta$.

 For a sufficiently small
universal constant $c''''>0$ (see below) define
\begin{equation}\label{eq: def triplestar}
\mathfrak F^{***}_\theta=\left\{L\in \mathfrak F^{**}_\theta\left| \mathcal H_L^1
\left(L\cap B_r(\widetilde{\gamma}(1))\cap E\right)\ge c''''r\right.\right\}\, .
\end{equation}
By Markov's inequality
$$
\mu_{\left[\widetilde{\chi_\theta}\right]}(\frak F_\theta^{**}
\setminus\frak F_\theta^{***})\lesssim \frac{\delta_1r^4}{r}\, .
$$
Therefore, if for a sufficently small universal constant, $c^{***}>0$,
we have,
\begin{equation}\label{eq:delta1 constraint}
\delta_1\le c^{***}\kappa^3\xi r,
\end{equation}
it follows that
\begin{equation}
\label{eq:lower triplestar}
\mu_{\left[\widetilde{\chi_\theta}\right]}\left(\frak F_\theta^{***}\right)
\ge \left(1-\frac{\xi}{2}\right)\cdot
\mu_{\left[\widetilde{\chi_\theta}\right]}\left(\frak F_\theta\right)\, ,
\end{equation}
 Note that our assumption~\eqref{delta1first1}, together with $\rho\le \frac12\kappa r^2$,
implies that~\eqref{eq:delta1 constraint} holds provided $c$ is small enough.

 For every $L\in \mathfrak F^{***}_\theta$ we know that
$$
\mathcal H_L^1\left(L\cap B_r(\widetilde{\gamma}(1))\cap E\right)
\ge c'''r \stackrel{\eqref{NC**}}{\ge}{\rm NC}_{B_{2C}\left(\widetilde \gamma(0)\right)}(E,L),
$$
provided that
\begin{equation}\label{eq:constraint r}
\eta\xi \rho^2\le 4c'''r\, ,
\end{equation}
which holds if $c$  is small enough.
In combination with~\eqref{eq:H1 in s-ball}, \eqref{NC**}, \eqref{eq:NC small on double star},
\eqref{eq:lower triplestar}
and Lemma~\ref{lnc}, we obtain the required result.
\end{proof}

\section{Appendix 1: The Sparsest Cut problem}
\label{SC}

As  mentioned in Section \ref{3steps}, Theorem \ref{tmain}, and in particular,
Corollary~\ref{coro:rate distortion}, leads to
 an exponential
improvement of the previously best known~\cite{KR06}
lower bound
 on the integrality gap of the Goemans-Linial
semidefinite relaxation of the Sparsest Cut problem with general demands.
 We now add some additional details to the discussion of Section \ref{3steps};
for further details, we refer to
see also our paper \cite{CKN09}, and the references therein.

The Sparsest Cut problem with general demands (and capacities)
 asks for an
efficient proceedure to partition a  weighted graph into two parts,
 so as to minimize the interface
between them. Formally, we are given an $n$-vertex graph $G=(V,E)$,
with a positive weight (called a capacity) $c(e)$ associated to each edge
$e\in E$, and a nonnegative weight (called a demand) $D(u,v)$ associated
to each pair of vertices $u,v\in V$. The goal is to evaluate in polynomial time
(and in particular, while examining only a negligible fraction of the subsets of $V$) the quantity:
\begin{equation}\label{eq:def Phi}
\Phi^*(c,D)=\min_{\emptyset \neq S\subsetneq V}
\frac{\sum_{uv\in E}c(uv)\left|\1_S(u)-\1_S(v)\right|}{\sum_{u,v\in V}D(u,v)\left|\1_S(u)-\1_S(v)\right|}\,  .
\end{equation}


To get a feeling
for the meaning of $\Phi^*$, consider the case $c(e)=D(u,v)=1$ for all $e\in E$ and $u,v\in V$. This is
an important instance of the Sparsest Cut problem which is called ``Sparsest Cut with Uniform Demands".
In this case $\Phi^*$ becomes:
\begin{equation}\label{eq:def uniform}
\Phi^*= \min_{\emptyset \neq S\subsetneq V}\frac{\#\{\mathrm{edges\ joining}\ S\ \mathrm{and}\
V\setminus S\}}{|S|\cdot |V\setminus S|}\, .
\end{equation}
Thus, in the case of uniform demands, the Sparsest Cut problem essentially amounts to solving
efficiently the combinatorial isoperimetric problem on $G$: determining the subset of the graph
whose ratio of edge boundary to its size is as small as possible.

>From now on, the  Sparsest Cut problem will be understood
to be with general capacities and demands.
These allow one to
tune the notion of ``interface"
between $S$ and $V\setminus S$ to a wide variety of combinatorial optimization
problems, which is
 one of the reasons why the Sparsest Cut problem is one of the most important problems
 in the field
of approximation algorithms. It is used as a subroutine in many approximation
 algorithms for
 NP-hard problems; see the survey article~\cite{Shmoys95}, as well as the references
in~\cite{LR99,ARV04,ALN08,CKN09},
 for some of the (vast) literature on this topic.

The problem of computing $\Phi^*(c,D)$ in polynomial time
is known to be
 NP hard \cite{SM90}.
The most fruitful approach to finding an approximate solution
has been to consider relaxations of the problem.
Observe that the term
$d_S\defeq\left|\1_S(u)-\1_S(v)\right|$, occuring in (\ref{eq:def Phi}) is just
the distance from $u$ to $v$ in the elementary cut metric associated to $S$;
see Section \ref{overview}.
One is therefore led consider the minimization of the functional
\begin{equation}\label{eq:def Phi more general}
\Phi(c,D,d)=
\frac{\sum_{uv\in E}c(uv)d(u,v)}{\sum_{u,v\in V}D(u,v)d(u,v)}\,  ,
\end{equation}
where $d$ varies  over an enlarged class of
metrics. The following successively larger classes of metrics
have played a key role: elementary cut metrics (as in (\ref{eq:def Phi})),
$L_1$ metrics (equivalently, cut metrics), metrics of negative type, arbitrary metrics.
Denote the later two  collections of metrics by ${\rm NEG}, {\rm MET}$, and denote the corresponding
minima by $\Phi^*(c,D)\geq \Phi^*_{L_1}(c,D)\geq \Phi^*_{{\rm NEG}}(c,D)\geq \Phi^*_{{\rm MET}}(c,D)$.
 From the standpoint of theoretical computer science, a key point is that
 $ \Phi^*_{{\rm MET}}(c,D)$  is a linear
program (since the triangle inequality is a linear condition)
and $\Phi^*_{{\rm NEG}}(c,D)$ can be computed by the Ellipsoid Algorithm. Thus, both
are computable in polynomial time with arbitrarily good precision.

Let $d$ denote an $L_1$ metric and let $\Sigma_d$  denote the cut measure occuring
in the cut metric representation $d=\sum_{S\subset V} \Sigma_d(S)d_S$
of $d$; see  (\ref{cutmetdef1}).
By substituting the cut metric representation
 into (\ref{eq:def Phi more general}),
it follows directly that in fact, $\Phi^*(c,D)= \Phi^*_{L_1}(c,D)$, for all $c,D$; \cite{AvisDeza, LLR95, AumanRabani}.
The key point is that $L_1$ metrics are the convex cone generated by elementary cut metrics
(and a corresponding statement would hold for any such convex cone and its generators).

If $d_1,d_2$ are any two metrics on $V$, define their distortion by
$$
{\rm dist}(d_1,d_2)=\left(\max_{u,v\in V}\, \frac{d_1(u,v)}{d_2(u,v)}\right)\cdot
 \left(\max_{u,v\in V} \, \frac{d_2(u,v)}{d_1(u,v)}\right)\, .
$$
It follows immediately that for all $c,D$,
\begin{equation}
\label{lowerbound}
\max\left\{\frac{\Phi(c,D,d_1)}{\Phi(c,D,d_2)},\frac{\Phi(c,D,d_2)}{\Phi(c,D,d_1)}\right\}
\leq{\rm dist}(d_1,d_2)\, .
\end{equation}
Thus,
\begin{equation}
\label{genineq1}
\sup_{c,D}\frac{\Phi^*_{L_1}(c,D)}{\Phi^*_{{\rm NEG}}(c,D)}
\leq \sup_{d\in {\rm NEG}}c_1(V,d) \, ,
\end{equation}
\begin{equation}
\label{genineq2}
\sup_{c,D}\frac{\Phi^*_{L_1}(c,D)}{\Phi^*_{{\rm MET}}(c,D)}\leq \sup_{d\in {\rm MET}} c_1(V,d)\, .
\end{equation}

Recall that by definition, the left-hand sides in (\ref{genineq1}), (\ref{genineq2}) are
the integrality gaps for the corresponding relaxations.
 Since
$\Phi^*(c,D)= \Phi^*_{L_1}(c,D)$,  Bourgain's embedding theorem implies that the first of
these integrality gaps is
$\lesssim \log n$; \cite{LLR95, AumanRabani}.
Similarly,  for the Goemans-Linial semidefinite relaxation,
 the embedding theorem
for metrics of negative type given in
\cite{ALN08} implies the the integrality gap is
$\lesssim(\log n)^{1/2+o(1)}$.
(Of course, the  upper bound in terms of distortion also applies if $L_1$ metrics are replaced
by the original elementary cut metrics, but since these metrics are so highly
degenerate,  it does it not provide useful information.
  This illustrates the power of the observation of  \cite{LLR95, AumanRabani}.)

 From now
on we restrict attention to the
 Goemans-Linial semidefinite relaxation. (What we say also applies
mutadis mutandis to the relaxation to ${\rm MET}$.)
As in (\ref{genineq1}), for all $c,D$, and all $d_2\in{\rm NEG}$,
\begin{equation}
\label{genineq}
\frac{\Phi^*_{L_1}(c,D)}{\Phi^*_{{\rm NEG}}(c,D)}\leq c_1(V,d_2)\, .
\end{equation}
A  duality argument (sketched below)
 shows that for all $d_2$, their exist $c,D$ for which (\ref{genineq}) becomes an equality. Therefore,
the integrality gap
is actually {\it equal} to $\sup_{d\in{\rm NEG}}c_1(V,d)$.

Thus, by Corollary \ref{coro:rate distortion} (see also Remark
\ref{coro: rate distortion implies int gap}) the integrality gap is
$\gtrsim (\log n)^\delta$ for some explicit $\delta>0$. Recall that the
previous best bound was $\gtrsim \log\log n$; see
\cite{KR06}, \cite{KV04}.

Here is a sketch of the duality argument; for further discussion, see Proposition 15.5.2
 of \cite{Mat02}. Consider a set $V$ of cardinality $n$, the vertices of
our graph (whose edge structure  will be determined below).
Define an embedding $F$ of the metrics on $V$ into $\R^{n(n-1)/2}$
 as follows:
The coordinates $x_{u,v}$
correspond to unordered pairs $u,v$ of distinct vertices in $V$,
and $ x_{u,v}(F(d))= d(u,v)$.
The image of the $L_1$ metrics on $V$ is  a convex cone $\mathcal L\subset\R^{n(n-1)/2}$
generated by the images of the elementary cut metrics $d_S$.
Let $d_2$ satisfy $F(d_2) \not\in \mathcal L$. By an easy compactness argument,
the distortion ${\rm dist}(d,d_2)$, for $d$ with $F(d)\in \mathcal L$, is minimized by some
 $d_1$, with $F(d_1)\in \mathcal L$.
Minimality implies
that in fact, $F(d_1)\in \partial \mathcal L$.
  Take a supporting hyperplane, $P$, for $\mathcal L$  which passes
through $F(d_1)$. Let $\ell$ denote a linear functional satisfying $\ell \, |_{P}\equiv 0$,
$\ell \, |_{\mathcal L}\leq 0$ and  $\ell(F(d_2))>0$.  Let $\sum_{(u,v)} \ell_{u,v} \cdot x_{(u,v)}^*$
denote the coordinate representation of $\ell$.
Define capacities
 by: $c(uv)=\ell_{u,v}$ if $\ell_{u,v}>0$ and otherwise $c(uv)=0$.
Define demands by: $D(u,v)=-\ell_{u,v}$ if $ \ell_{u,v}<0$ and otherwise $D(u,v)=0$.
Define a graph structure with vertices, $V$, by stipulating that
$uv\in E$ if and only if $c(uv)>0$.
It is trivial to check that
for the above $c,D$, we have $\Phi(c,D,d)\geq 1$, for all $d$ with
$F(d)\in \mathcal L$. Also  $\Phi(d_1)=1$,
$\Phi(c,D,d_2)\leq \frac{1}{{\rm dist}(d_1,d_2)}=\frac{1}{c_1(V,d_2)}$. These relations imply
that for these $c,D$, (\ref{genineq}) is an equality.
By choosing $d_2\in {\rm NEG}$ such that $c_1(V,d_2)$ is maximal,
it follows that (\ref{genineq1}) is also an equality.
 (If $F(d_2)\in \mathcal L$, this is trivial.)

\section{Appendix 2: Quantitative bounds, coercivity and monotonicity}
\label{general}

In this Appendix,  we briefly discuss from a more general standpoint,
the structure of our argument,
as outlined after the statement of Theorem \ref{tmain} and in Section \ref{outline}.
We point out that in essence, what we have done
follows the general scheme
 of other arguments in
 geometric analysis and nonlinear partial differential equations;
compare Example \ref{tangentcones} below.
Typically, the
results are not stated explicitly in
quantitative form.
Here,  we wish to emphasize
that  the possibility of an estimate of the form of (\ref{compression})
 is actually implicit in the arguments. The crucial ingredient is a quantity which is
coercive, monotone and bounded.


\subsection*{Coercivity and almost rigidity.}

The term, {\it rigid}, connotes special
 (i.e. highly constrained)  structure.
 A standard feature of (the statment and proof of)
rigidity theorems is
 the existence
of a numerical measurement $Q\geq 0$ which is
{\it coercive} in the sense that
 if $Q=0$, then the  desired rigidity holds, and more generally
(and often much harder to prove) if $Q<\epsilon^a$, for some $a<\infty$,
then in a suitable sense,
the structure is $\epsilon$-close the one which is obtained in the rigid case.
Statments of this
type are known as a {\it stability theorems}, {\it $\epsilon$-regularity theorems},
  {\it  almost rigidity theorems}.  A classical example from Riemannian geometry is the  {\it sphere theorem}, in
which the coercive quantity is minus the logarithm of the pinching; see~\cite{BS09} and  the references therein.

 In our case, we are given
$E\subset \H$, and the coercive quantity is $Q(E)={\rm NM}_{B_r(x)}(E)$,
the nonmonotonicity of $E$ on $B_r(x)$.  Coercivitity
is
the statement that monotone subsets are half-spaces,
or more generally, that almost monontone subsets are close to half-spaces; see
Theorem \ref{stability}.


\subsection*{Bounded monotone quantities and existence of a good scale.}

As in Sections \ref{3steps}, \ref{overview}, we point out the general character of the
estimate in Proposition \ref{pscest} for the scale on
which Theorem \ref{stability} can be applied.
Namely, by Markov's inequality~\eqref{cheb} (which in this case amounts to the pigeonhole principle),
such an estimate for the scale
will appear {\it whenever} we are
dealing with an a priori bounded  nonnegative quantity,  which can
be written as a sum of nonnegative terms, which correspond to the various scales,
such that each term is coercive on its own scale (in the suitably scaled sense).
 Such a quantity is {\it monotone} in the sense that the sum is nondecreasing
as we include more and more scales; compare (\ref{scales}), (\ref{nmestimate2}).

Quantities which are coercive and monotone
are well-known to play a
key role in geometric analysis and in
partial differential
 equations. In the latter case,
for evolution equations,  monotonicity
is defined with respect to the time parameter, rather than the scale.

We include below the following illustrative
 example, which requires familiarity with
 Riemannian geometry. Numerous other  choices from diverse areas would
serve equally well; compare Remark~\ref{rem: no scale}.
\begin{example}
\label{tangentcones}
A theorem from Riemannian geometry
 states that for  noncollapsed Gromov-Hausdorff limit
spaces, $M^n_i\stackrel{d_{GH}}{\longrightarrow}Y^n$,
such that $\text{Ric}_{M_i^n}\geq -(n-1)$ for all $i$, every tangent cone $Y_y$
is a metric cone; see  Remark 4.99  of \cite{chco}
 and \cite{chco1}. In this case, the
rigid objects are metric cones and the relevant coercive
quantity, $Q$, is derived from the volume ratio,
\begin{equation}\label{bishopgromov}
\frac{\text{Vol}(B_r(p))}{\text{Vol}(B_r(\underline{p}))}\, ,
\end{equation}
where
$\underline{p}\in \underline{M}^n$ and
$\underline{M}^n$
denotes the hyperbolic $n$-space
with curvature $\equiv -1$.
The coercivity of $Q$ is guaranteed by
the  ``volume cone implies metric cone theorem'' and its
corresponding
almost rigidity theorem, which states that almost volume
cones are close in the Gromov-Hausdorff sense to be almost metric
cones; see \cite{chco1},  \cite{chco}.
The monotonicity of
$Q$ is a consequence of the Bishop-Gromov inequality, which asserts
that the volume ratio in~\eqref{bishopgromov} is a monotone nonincreasing function of $r$.
\end{example}


\bibliography{qu}
\bibliographystyle{abbrv}

\end{document}